\documentclass[final,smallcondensed,numbook]{svjour3}

\usepackage[english]{babel}				\usepackage{multirow}			
\usepackage[utf8]{inputenc}
\usepackage[parfill]{parskip} 
\usepackage{amsmath,amssymb}
\usepackage{graphicx}
\usepackage{authblk}
\usepackage{bbm, bm}
\usepackage{ upgreek }
\usepackage{algorithm}
\usepackage{algpseudocode}
\usepackage[normalem]{ulem}
\usepackage{xspace}
\usepackage[shortlabels]{enumitem}
\usepackage{comment}
\usepackage[dvipsnames]{xcolor}
\usepackage{hyperref}
\usepackage{cleveref}
\usepackage{comment}

\hypersetup{
     colorlinks=true,
     linkcolor=blue,
     filecolor=blue,
     citecolor = BrickRed,      
     urlcolor=cyan,
     }

\providecommand{\norm}[1]{\lVert#1\rVert}
\newcommand{\vertiii}[1]{{\vert\kern-0.25ex\vert\kern-0.25ex\vert #1 
    \vert\kern-0.25ex\vert\kern-0.25ex\vert}}

\DeclareMathOperator*{\argmin}{arg\,min}
\newcommand\bdiv{\mathbf{div}}

\DeclareMathOperator{\diam}{diam}

\let\emptyset\varnothing

\algnewcommand\And{\textbf{and}}
\algnewcommand\Or{\textbf{or}}

\title{An unfitted HDG discretization \\ for a model problem in shape optimization}
\author{
    Esteban Henr\'iquez
    \and
    Tonatiuh S\'anchez-Vizuet
    \and
    Manuel~Solano
    
    }
\institute{
    Esteban Henr\'iquez
    \at Department of Applied Mathematics, University of Waterloo, Canada. \email{ehenriqu@uwaterloo.ca} \and
    Tonatiuh S\'anchez-Vizuet \textbf{(Corresponding author)}
    \at Department of Mathematics, The University of Arizona, USA. \email{tonatiuh@arizona.edu} \and
    Manuel Solano
    \at  Departamento de Ingenier\'ia Mathematical, Facultad de Ciencias Físicas y Matemáticas and Center for Research in Mathematical Engineering CI$^2$MA  Universidad de Concepci\'on, Chile. \email{msolano@ing-mat.udec.cl}
}
\authorrunning{Henr\'iquez, S\'anchez-Vizuet, Solano}

\begin{document}
\maketitle

\begin{abstract}
We apply an unfitted HDG discretization to a model problem in shape optimization. The method proposed uses a fixed, shape regular, non-geometry conforming mesh and a high order transfer technique to deal with the curved boundaries arising in the optimization process. The use of this strategy  avoids the need for constant remeshing and enables a highly accurate description of the domain using a coarse computational mesh. We develop a rigorous analysis of the well-posedness of the problems that arise from the optimality conditions, and provide an \textit{a priori} error analysis for the resulting discrete schemes. The main contribution of this communication is the analysis of the discretization of the equation for the deformation field, which involves a Neumann condition that requires special treatment, and its coupling with the state and adjoint equations. Numerical examples with manufactured problems are provided demonstrating the convergence of the scheme and the efficiency of the transfer path method. The approach proposed yields high resolution approximations of the boundary using grids with coarse mesh parameters.\\
\end{abstract}
%
\section{Introduction}\label{sec:introduction}
%
Shape optimization is an important branch of the study of optimal control theory, which was developed extensively in the 1990s and arises from the need to minimize a quantity (for instance the amount of material needed to build an industrial component or the energy cost of production) through the modification of the design shape. This area has inspired the development of a wide variety of theoretical and purely mathematical tools and has a large number of applications in science and engineering, such as architecture and civil engineering \cite{Be2014}, fluid mechanics \cite{challis2009level,duan2008shape,zhou2008variational}, modeling of quantum chemistry phenomena \cite{braida2022shape,cances2004electrons}, electromagnetism or photonics \cite{jensen2011topology,lebbe2019robust}, among other research fields. From a mathematical point of view, we can see shape optimization as finding the minimum (possibly local) of a cost functional defined over a set of admissible domains. In many cases---and we shall focus on these---the minimization problem is constrained by a partial differential equation (PDE) defined on the target domain. The computation of this minimum often requires the repeated solution of further partial differential equations arising from the optimality conditions associated to the functional in question.

The first Discontinuous Galerkin (DG) method was developed in 1973 by Reed and Hill in the context of a neutron transport problem involving a linear, time-independent, hyperbolic equation \cite{reed1973triangular}. Since then, DG methods have become one of the most widely used tools for the numerical solution of PDEs. However, DG discretizations were criticized due to the fact that the associated linear systems involve too many unknowns and involve a complicated computational implementation compared to Continuous Galerkin (CG) methods. These criticisms were resolved after the development of Hybridizable Discontinuous Galerkin (HDG) methods, first for diffusion problems and later presented in a unified framework \cite{cockburn2009unified}.

During the last two decades, HDG methods have been extensively developed for different types of equations, for instance, diffusion equations \cite{cockburn2008superconvergent,CoGoSa2010,cockburn2009superconvergent,kirby2012cg}, convection-diffusion equations \cite{cockburn2009hybridizable,fu2015analysis,nguyen2015class}, acoustic and elastic waves \cite{ArSaSo:2025,cockburn2014uniform}, Stokes flow \cite{carrero2006hybridized,cockburn2013conditions,gatica2016priori,manriquez2022dissimilar}, Oseen and Brinkman equations \cite{araya2019posteriori,cesmelioglu2013analysis,gatica2018priori}, Navier-Stokes equations \cite{cesmelioglu2017analysis,nguyen2011implicit,rhebergen2013space}, linear and nonlinear elasticity \cite{cockburn2013linearelasticity,nguyen2012hybridizable,soon2009hybridizable}, just to name a few.
\par
In recent years, the development of the \textit{transfer path method} has allowed the application of HDG to domains that are not necessarily polygonal/polyhedral by approximating the solution in a polygonal subdomain. This transfer technique was introduced within the context of HDG discretizations for linear elliptic equations in \cite{CoQiuSo2014,CoSo2012} and allows for the use of simple polygonal, non-interpolating meshes while still maintaining high order of convergence. Since then, this method has been used for Stokes flow \cite{solano2019high}, Oseen equations \cite{solano2022oseen}, the Helmholtz equation \cite{CaSo2021}, convection diffusion equations \cite{cockburn2014convdif}, the Grad-Shafranov equation \cite{SaSo2019,SaCeSo2020}, coupling with integral equations \cite{CoSaSo2012}, non-linear problems \cite{SaSaSo:2021a,SaSaSo:2019,SaSaSo:2022a}, and recently for a distributed optimal control problem \cite{HeSo2023control}, among others.
\par
For the treatment of partial differential equations arising from the shape optimization problem, work has been carried out using a variety of methods, for instance, Finite Element Method (FEM) \cite{dogan2007discrete}, Cut Finite Element Methods (CutFEM) \cite{BuElHaLarLar:2017,burman2018shape}, Boundary Element Method (BEM) \cite{bertsch2008topology,neches2008topology}, level-set methods \cite{ALLAIRE201422,Allaire}, among others. For this manuscript we seek to make a first approach in using the combination of the transfer path method with an HDG discretization to deal with the curved domains that arise naturally in shape optimization problems. These techniques combined shall allow for the use of simple, shape regular triangulations, while maintaining the high order of convergence of these methods. The current experiments suggest that, with the transfer path method and using a mesh size of order $h$, it is possible to obtain a description of the boundary of order $h^2$ or even higher if high order descriptors of the boundary are used.

In Section \ref{sec:model_problem} we introduce a model problem and derive the  optimality conditions associated to the optimization problem as well as the optimization algorithms at the continuous level. Section \ref{sec:Geometry} is devoted to the introduction of the transfer path method and the technical assumptions related to the geometry and triangulation. The unfitted HDG discretizations of the state, adjoint and deformation variables are introduced in Section \ref{sec:hdg_method}. Some of the tools required for the  analysis of the discrete equations involved had been developed previously. The main contribution of this communication is the analysis of the discretization of the equation for the deformation field, which involves a Neumann condition that requires special treatment, and its coupling with the state and adjoint equations. This is carried out in Section \ref{sec:a_priori_error_estimates}. Finally, using manufactured examples, Section \ref{sec:numerical_experiments} presents numerical experiments showcasing the convergence of the discrete variables and the behavior of the suggested shape--optimization algorithm.  
%
\section{The model problem}\label{sec:model_problem}
%
Let $\mathcal U \subset \mathbb R^d$ be a fixed domain, $\mu(\cdot)$ denote the Lebesgue measure in $\mathbb R^n$, and define the set
\[
\mathcal O := \left\{ \Omega: \Omega\subset \mathcal U \;\text{ and }\; \mu(\Omega) = m_0\right\}\,.
\]
We will refer to any $\Omega\in\mathcal O$ as an \textit{admissible domain} and will denote its boundary as $\Gamma:=\partial\Omega$. For a fixed \textit{target function} $\widetilde y \in H^1(\mathcal U)$ we will define the energy functional $J:\mathcal O \times H^1(\mathcal U) \, \to \, \mathbb R$ as
\[
J(\Omega;y): = \frac{1}{2}\,\int_\Omega(y\,-\,\widetilde y)^2. 
\]
In this work we will consider the model problem of finding a domain $\widehat{\Omega}\in\mathcal O$ and a function $\widehat y \in H^1(\mathcal U)$,  such that
\begin{equation}\label{min_J}
    \left(\widehat{\Omega},\widehat y\right) \,=\, \argmin_{\Omega\in \mathcal{O}} J\left(\Omega; y\right)
\end{equation}
where the function $y$ is subject to the \textit{state equation}
\begin{subequations}
\label{pde_1}
\begin{alignat}{2}
-\,\nabla\cdot(a\,\nabla y) & \,=\, f \quad &&\mbox{ in } \Omega\,, \\
y & \,=\, g \quad && \mbox{ on }\Gamma.
\end{alignat}
\end{subequations}
Above, the scalar function $a$ is assumed to be such that there exist real numbers $\underline a$ and $\overline a$ satisfying
\[
0< \underline a \leq a \leq \overline a <\infty \qquad \text{ almost everywhere on }\mathcal U\,,
\]
and $f\in H^1(\mathcal{U})$ and $g\in H^2(\mathcal{U})$ are given problem data. Note that, since $y$ is defined as the solution of a PDE defined on $\Omega$, it is itself a function of $\Omega$. To keep notation as light as possible, we will avoid denoting this fact explicitly as $y(\Omega)$, but the reader should keep this dependence in mind. 

If we consider a Lipschitz mapping $\bm V:\mathcal{U}\to \mathbb{R}^d$ and small deformations of the domain $\Omega$ of the form
\[
\Omega \ni \boldsymbol x \mapsto \boldsymbol x  + \epsilon \boldsymbol V(\boldsymbol x) \qquad \text{ for } 0<\epsilon< 1,
\]
the variation $\updelta J(\Omega;\bm V)$ of the functional $J(\Omega; y)$ can be characterized (cf. \cite[Chapter 11]{manzonioptimal}) in terms of the problem data, $g$ and $\widetilde y$, and the \textit{adjoint function} $z$, as
\begin{equation}
\label{deltaJ}
    \updelta J(\Omega;\bm V)\,=\, \int_{\Gamma}G(\Gamma)\,\bm V\cdot \bm n,
\end{equation}
where 
\begin{equation}\label{eq:G}
    G(\Gamma) :=\, a\,\partial_{\bm n}z\,\partial_{\bm n}(y\,-\,g)\,+\,\,\dfrac{1}{2}(g\,-\,\widetilde{y})^{2}\,, \qquad 
\end{equation}
and $z\in H^2(\Omega)$ is the solution to the \textit{adjoint problem}
\begin{equation}\label{adjoint_eq_mp}
    \begin{array}{r c l c}
        -\,\nabla\cdot(a\,\nabla z) & \,=\, & y\,-\,\widetilde{y} & \qquad \mbox{ in }\Omega\,,  \\[2ex]
        z & \,=\, & 0 & \qquad \mbox{ on }\Gamma\,. 
    \end{array}
\end{equation}

Therefore, we look for a domain $\widehat{\Omega}$ that satisfies
$\updelta J(\widehat{\Omega};\bm V)\,=\,0\,$ for some admissible direction $\bm V$ that we will call the \textit{deformation field} and will be specified below.

It is known \cite{HeSo2005} that shape optimization problems generally do not have a unique solution because the optimality condition $\updelta J(\widehat{\Omega};\bm V)\,=\,0\,$ does not uniquely determine $\bm V$. For instance, to first order, deformations in the direction tangential to $\Gamma$ will not affect the value of the functional. Hence, a common choice is to fix a small portion, $\Gamma_D$, of the boundary and to look for a deformation field that satisfies
\begin{equation}
\label{pde_V}
    \begin{array}{r c l c}
        -\,\Delta \bm V & = & \bm 0 & \quad \mbox{ in }\Omega, \\[2ex]
        \partial_{\bm n}\bm V & = & -\,G(\Gamma)\,\bm n & \quad \mbox{ on }\Gamma_{N},\\[2ex]
        \bm V &=& \bm 0& \quad \mbox{ on }\Gamma_D,
    \end{array}
\end{equation}
where $\Gamma_N$ is the piece of the boundary that can be deformed and $\Gamma_D$ is the piece of the boundary that will remain fixed and $\Gamma = \Gamma_D \cup \Gamma_N$. We will assume that $\mu(\Gamma_N)\,\neq \,0$ and $\mu(\Gamma_D)\,\neq \,0$.

The space
\begin{equation*}
    [H_{D}(\Omega)]^{d}\,:=\,\{\bm w \in [H^{1}(\Omega)]^{d},\quad \bm w \,=\, \bm 0 \mbox{ on } \Gamma_D\}.
\end{equation*}
will be used to enforce the Dirichlet boundary condition in the weak formulation for the problem \eqref{pde_V}, which we can now state as 
\begin{align}
\label{weak_pde_V}
    & \text{Find } \bm V\in [H_{D}(\Omega)]^d \;\; \text{ such that } \\
    \nonumber
    &\int_{\Omega}\nabla \bm V\bm :\nabla \bm w
    +\int_{\Gamma_N}G(\Gamma)\,\bm w\cdot\bm n\,=\,0 \quad \forall\,\,\bm w\in [H_{D}^1(\Omega)]^d\,.
\end{align}
Here, for any tensor fields $\bm\psi \,=\, (\psi_{ij})_{i,j=1,n}$ and $\bm\zeta \,=\, (\zeta_{ij})_{i,j=1,n}$, the tensor inner product is defined as
\[
    \bm\psi\bm :\bm \zeta \,=\, \sum_{i=1}^n\sum_{j=1}^n\psi_{ij}\,\zeta_{ij}.
\]
Thus, by letting $\bm w \,=\, \bm V$, it follows from the weak formulation \eqref{weak_pde_V} and equation \eqref{deltaJ} that
\begin{equation}
\label{descent_direction}
    \updelta J(\Omega;\bm V) \,=\, \int_{\Gamma_N}G(\Gamma)\,\bm V\cdot\bm n\,=\,-\,\int_{\Omega}|\nabla \bm V|^2\,\leq 0\,,
\end{equation}
which implies that deformations of $\Omega$ in a direction $\bm V$ that satisfies \eqref{pde_V} guarantee a decrease in the value of the functional. Note that this is not the only possible choice for $\bm V$, but we will stick to it for the remainder of this work.

The previous argument motivates the following algorithm to compute an approximation for $\Omega^{\text{opt}}$ based on the gradient descent method \cite[Algorithm 11.1]{manzonioptimal}. Starting from an initial guess $\Omega^{(0)}$, the approximation can be updated iteratively as
\[
\Omega^{(k+1)}\,=\,(I\,+\, s_k\,\bm V^{(k)})\Omega^{(k)} \quad (k\in \mathbb{N})\,,
\]
where $s_k$ is a scalar parameter, referred to as the \textit{step size}, that controls the size of the deformation (its precise value will be determined later) and the descent direction $\bm V^{(k)}$ is the solution to
\begin{subequations}\label{pde_Vk}
\begin{alignat}{6}
        -\,\Delta \bm V^{(k)} =&\, \bm 0 \qquad&& \text{ in }\Omega^{(k)}, \\[1ex]
        \partial_{\bm n}\bm V^{(k)}  = &\, -\,G(\Gamma^{(k)})\,\bm n  \qquad&& \mbox{ on }\Gamma_{N}^{(k)},\\[1ex]
        \bm V^{(k)}  = &\, \bm 0 \qquad&& \mbox{ on }\Gamma_{D}^{(k)}.
\end{alignat}
\end{subequations}
Analogously to \eqref{descent_direction} we can prove that $\bm V^{(k)}$ is a descent direction for each $k\in\mathbb{N}$, and therefore 
\[
J(\Omega^{(k+1)},y^{(k+1)}) \,\leq\, J(\Omega^{(k)}, y^{(k)})\quad \forall \,k\in \mathbb{N}\,.
\]
Generating the Neumann data for the problem above requires the successive solution of the state equation \eqref{pde_1} and the adjoint equation \eqref{adjoint_eq_mp} at every iteration. The process can be repeated until the value of the functional falls below a predetermined tolerance $\text{TOL}>0$. 

In summary, given an initial guess for the domain $\Omega^{(0)}$ and a desired tolerance TOL for the value of the functional $J(\Omega,y)$ the sequence of steps to obtain an approximation to the shape optimal shape, is 

\begin{algorithm}[H]
\caption{\cite[Algorithm 11.1]{manzonioptimal}}\label{alg_shp_opt}
\begin{algorithmic}
\Require Initial domain $\Omega^{(0)}$ and tolerance parameter TOL
\State $y^{(0)} \gets y(\Omega^{(0)})$ by solving the state equation (cf. \eqref{pde_1}) in $\Omega^{(0)}$
\State $z^{(0)} \gets z(\Omega^{(0)})$ by solving the adjoint equation (cf. \eqref{adjoint_eq_mp}) in $\Omega^{(0)}$
\State compute $J(\Omega^{(0)})$
\State compute $G(\Gamma^{(0)})$
\State compute a deformation field $\bm V^{(0)}$ by solving \eqref{pde_Vk}
\State $k \gets 0$
\While{$|\updelta J(\Omega^{(k)};\bm V^{(k)})/\updelta J(\Omega^{(0)};\bm V^{(0)})|>$ TOL \Or  $\,|J(\Omega^{(k)})-J(\Omega^{(k-1)})|>$ TOL } 
\State compute the step size parameter $s_{k}$ with a line search routine
\State $\Omega^{(k+1)} \gets (I +s_{k}\bm V^{k})(\Omega^{(k)})$

\State $y^{(k+1)}\gets y(\Omega^{(k+1)})$ by solving the state equation (\ref{pde_1}) in $\Omega^{(k+1)}$
\State $z^{(k+1)}\gets z(\Omega^{(k+1)})$ by solving the adjoint equation (\ref{adjoint_eq_mp}) in $\Omega^{(k+1)}$
\State compute $J(\Omega^{(k+1)})$
\State compute $G(\Gamma^{(k+1)})$
\State compute a deformation field $\bm V^{(k+1)}$ by solving \eqref{pde_Vk}
\State $k\gets k+1$
\EndWhile
\end{algorithmic}
\end{algorithm}

In practice, the \textit{volume constraint} on the admissible domains $\mu(\Omega) = m_0$ can be challenging to enforce computationally. Thus, we follow \cite[Section 11.5.1]{manzonioptimal}, and introduce a Lagrange multiplier $\xi$ and the functional
\[
\widetilde{J}(\Omega,\xi;y): = J(\Omega;y)+\xi( \mu(\Omega)-m_0). 
\]

The variation $\updelta \widetilde{J}(\Omega,\xi;\bm V)$ of the functional $\widetilde{J}(\Omega,\xi; y)$ can be characterized by 
\[
\updelta \widetilde{J}(\Omega,\xi;\bm V) = \updelta J(\Omega;\bm V)+\xi\int_\Gamma \boldsymbol{V}\cdot \boldsymbol{n}= \int_\Gamma \widetilde{G}(\Gamma,\xi)\boldsymbol{V}\cdot \boldsymbol{n},
\]
where $\widetilde{G}(\Gamma,\xi):=G(\Gamma)+\xi$. We compute $\boldsymbol{V}$ by solving \eqref{pde_V} using the quantity $-\widetilde{G}(\Gamma)\boldsymbol{n}$ as a Neumann boundary condition. In addition, the Lagrange multiplier is updated every iteration according to
\begin{align}\label{LM:update}
    \xi^{(k+1)} = \dfrac{\xi^{(k)}+\chi(\Gamma)}{2}+\epsilon (\mu(\Omega)-m_0), \quad \text{ with } \quad \chi(\Gamma) :=-\frac{1}{\mu(\Gamma)}\int_{\Gamma} G(\Gamma)\,,
\end{align}
and $\epsilon>0$ a sufficiently small fixed constant. This leads to the following modified algorithm:

\begin{algorithm}[H]
\caption{Modification of Algorithm \ref{alg_shp_opt}}\label{alg2_shp_opt}
\begin{algorithmic}
\Require Initial domain $\Omega^{(0)}$,  parameters TOL and $\epsilon$
\State $y^{(0)} \gets y(\Omega^{(0)})$ by solving the state equation (cf. \eqref{pde_1}) in $\Omega^{(0)}$
\State $z^{(0)} \gets z(\Omega^{(0)})$ by solving the adjoint equation (cf. \eqref{adjoint_eq_mp}) in $\Omega^{(0)}$
\State compute $\chi(\Gamma^{(0)})$ and set $\xi^{(0)}:=\chi(\Gamma^{(0)})$
\State compute $\widetilde{J}(\Omega^{(0)},\xi^{(0)})$
\State compute $\widetilde{G}(\Gamma^{(0)},\xi^{(0)})$
\State compute a deformation field $\bm V^{(0)}$ by solving \eqref{pde_Vk} with  $\widetilde{G}(\Gamma^{(0)},\xi^{(0)})$ as Neumann boundary data
\State $k \gets 0$
\While{$|\updelta \widetilde{J}(\Omega^{(k)},\xi^{(k)};\bm V^{(k)})/\updelta \widetilde{J}(\Omega^{(0)},\xi^{(0)};\bm V^{(0)})|>$ TOL \Or  $\,|\widetilde{J}(\Omega^{(k)},\xi^{(k)})-\widetilde{J}(\Omega^{(k-1)},\xi^{(k-1)})|>$ TOL } 
\State for a fixed $\xi^{(k)}$, compute the step size parameter $s_{k}$ by a line search routine.
\State $\Omega^{(k+1)} \gets (I +s_{k}\bm V^{k})(\Omega^{(k)})$
\State $y^{(k+1)}\gets y(\Omega^{(k+1)})$ by solving the state equation (\ref{pde_1}) in $\Omega^{(k+1)}$
\State $z^{(k+1)}\gets z(\Omega^{(k+1)})$ by solving the adjoint equation (\ref{adjoint_eq_mp}) in $\Omega^{(k+1)}$
\State compute $\chi(\Gamma^{(k+1)})$ and set $\xi^{(k+1)}=\dfrac{\xi^{(k)}+\chi(\Gamma^{(k+1)})}{2}+\epsilon (\mu(\Omega^{(k+1)})-m_0)$
\State compute $\widetilde{J}(\Omega^{(k+1)},\xi^{(k+1)})$
\State compute $\widetilde{G}(\Gamma^{(k+1)},\xi^{(k+1)})$
\State compute a deformation field $\bm V^{(k+1)}$ by solving \eqref{pde_Vk}  with  $\widetilde{G}(\Gamma^{(k+1)},\xi^{(k+1)})$ as Neumann boundary data
\State $k\gets k+1$
\EndWhile
\end{algorithmic}
\end{algorithm}

%
\section{Unfitted geometric discretization}\label{sec:Geometry}
%
In addition to the solution to three boundary value problems at every step, Algorithm \ref{alg_shp_opt} requires the update of the domain $\Omega^{(k)}$ where the state, adjoint, and deformation equations will be posed on the next iteration. For many numerical schemes, this poses the need to repeatedly generate a new computational mesh. This requirement can render an algorithm too costly, especially if the iteration count is high---which is typically the case. Moreover, to achieve a highly accurate approximation of $\Omega^\text{opt}$, traditional methods would require either the use of an extremely fine interpolatory mesh, or an isogeometric triangulation, or a high order curvilinear mesh. All of these methods can produce a precise description of the target geometry, but they achieve so at an additional computational cost per step that may result unacceptable if the iteration count is high---which is typically the case.

Due to these particularities, the transfer path method \cite{CoQiuSo2014,CoSo2012}---that we describe below---provides an ideal tool for dealing with this iterative process, as it obviates both the need to create a new mesh with every successive iteration and the requirement of a mesh that finely captures the geometric properties of the boundary $\Gamma^{(k)}$. These characteristics can dramatically decrease the additional costs associated with the geometric approximation. The aim of this paper is to propose and analyze unfitted HDG discretizations that take advantage of the transfer path method for the three relevant equations in the process. Here, we describe the geometric setting for the computations and the transfer path method.

\textbf{The computational domain}\\
We will assume that the domain $\Omega$ has a Lipschitz boundary $\Gamma$ and will choose a background polyhedral domain $\mathcal{M}$ such that $\Omega \subset \overline{\mathcal U} \subset \overline{\mathcal M}$. We will denote a generic triangulation (or tetrahedrization, depending on the dimension) of $\mathcal M$ by $T_h$ and a generic element in the triangulation by $K$. As usual, we will denote
\[
h:= \max_{K\in T_h}\left\{\text{diam}(K)\right\} \qquad \text{ and } \qquad \underline h:= \min_{K\in T_h}\left\{\text{diam}(K)\right\}.
\]
We refer to $T_h$ as a \textit{background triangulation} and define 
\begin{equation}
\label{eq:triangulation}
\mathcal{T}_h := \{K\in T_h: K \subset \overline\Omega\},
\end{equation}
i.e. the set of all the elements $K\in T_h$, which are completely contained in $\Omega$. For any given background triangulation $T_h$ we will define the \textit{computational domain} as 
\[
D_h : = \left\{x\in \Big(\bigcup_{K\in \mathcal{T}_h}\overline{K}\Big)^{\circ}\right\}
\]
and will refer to its boundary $\Gamma_h \,:=\, \partial D_h$ as the \textit{computational boundary}. An example of the construction of $D_h$ is shown in the left panel of Fig. \ref{fig:transferpaths}.

We will say that $e$ is an interior edge or face of the triangulation $\mathcal T_h$ if there are two elements $K^{+}$ and $K^{-}$ in $\mathcal{T}_h$, such that, $e \,=\, \partial K^{+} \cap \partial K^{-}$. In the same way, we will say that $e$ is a boundary edge or face if there is an element $K \in \mathcal{T}_h$ such that $e \,=\, \partial K \cap \Gamma_h$.  We will let $\mathcal{E}_h$ be the set of all edges or faces of the triangulation, $\mathcal{E}_h^\circ$ be the set of interior edges or faces, and $\mathcal{E}_h^{\partial}$ the set of exterior edges or faces of $\mathcal{T}_h$. Thus $\mathcal{E}_h\,=\, \mathcal{E}_h^\circ\cup\mathcal{E}_h^{\partial}$.

The outward unit normal of the element $K\in \mathcal{T}_h$ will be denoted by $\bm n$ and will denote it by $\bm n_e$ whenever we want to emphasize that $\bm n$ is the normal to a particular face $e$. Moreover, for each edge or face $e$ of $K$, we denote the height of the element with respect to that edge or face as $h_e^{\perp}$. On a similar vein, for every boundary edge/face $e\in\partial\mathcal E_h$, we define $H_e^{\perp}$ as the length of the longest segment connecting $e$ and $\Gamma$ that is both parallel to the normal direction $\bm n_e$ and completely contained in $K_{ext}^e$. We then define the parameter
\[
    r_e\,:=\, H_e^{\perp}/\,h_e^{\perp}\,,
\]
that serves as a measure of the local distance between $\Gamma$ and $\Gamma_h$, relative to the local mesh size. Related to this, we have the \textit{proximity parameter}
\[
    R \,:=\, \max_{e\in\mathcal{E}_h^{\partial}}r_e\,.
\]

\textbf{Admissible triangulations}\\
For all the analysis that follows, we will consider only families $\{\mathcal T_h\}_{h>0}$ of admissible triangulations, where a triangulation $\mathcal T_h$ is said to be \textit{admissible} if the following conditions are satisfied:
\begin{enumerate}
\item There exists a constant $r>0$ such that $r h \leq \underline h$. This condition is known as \textit{quasi uniformity}.
\item  There exists $\rho >0 $ such that, for any triangle $K\in\mathcal{T}_h$ we have that $\diam(B_K)\,\geq\, \rho\, \diam(K)$, where $B_K$ is the biggest ball inscribed in the element $K$. This condition is called \textit{uniform shape regularity}.

\item If for any two distinct elements $K,K^\prime\in\mathcal T_h$ it holds that the intersection $K\cap K^\prime$ is either empty, it consists of a single common vertex or a single common edge/face. When this holds we say that the triangulation has no \textit{hanging nodes}. This is in fact not a necessary, but it considerably simplifies the analysis \cite{BeMaSo2025,ChCo2012,ChCo2013}.
\item\label{TransferPath} There exists a bijective function
\begin{equation}\label{eq:Phi}
\boldsymbol \phi: \Gamma_h \longrightarrow \Gamma
\end{equation}
That, for every point $\bm x\in \Gamma_h$, assigns a point $\bar{\bm x}:=\boldsymbol \phi(\boldsymbol x)\in \Gamma$ such that the straight segment connecting $\overline{\boldsymbol x}$ to $\boldsymbol x$ satisfies the following conditions:
\begin{enumerate}
\item Does not intersect the interior of the computational domain $D_h$.
\item $|\bm x - \bar{\bm x}|\leq C h^{n+\delta}$, for some $C>0$, $n\in\mathbb N$ and $0<\delta<1$. We refer to this as the \textit{local proximity condition of order} $n$. For Dirichlet problems, it is enough to require $n=1$  \cite{SaSaSo:2019,SaCeSo2020}. This condition can also be expressed as
\begin{equation}\label{eq:localProximity}
 \max_{e \in \mathcal{E}_h^{\partial}} H_e^\perp \leq C h^{n+\delta}\,.
\end{equation}
\vspace{-.5cm} 
\item There exist generic positive constants such that
\begin{equation}
\label{eq:equivalence}
\|\psi \boldsymbol n \circ\bm\phi\|_{\Gamma_h}\lesssim \|\psi\boldsymbol n\|_\Gamma \lesssim \|\psi\boldsymbol n \circ\bm\phi\|_{\Gamma_h}\,
\end{equation}
for every $\psi\in H^1(\Omega)\cap H^1(\Omega_h)$.
\end{enumerate}
\end{enumerate}
Up to this point, the conditions imposed on admissible triangulations have had a purely geometric character. They have pertained solely to the capability of a family of triangulations to smoothly approximate the desired domain as the mesh is refined. The next batch of conditions considers the interaction of the physical parameters, the geometry and a stabilization parameter characteristic of HDG. In particular, they will determine the maximum admissible distance between the computational and physical boundaries in terms of the physical parameters, the stabilization parameter and the polynomial degree of the approximation.

We start by introducing the a pair of edge--wise constant functions that will be useful in the convergence analysis
\[
C^{ext}_{e} :=\,\frac{1}{\sqrt{r_e}}\,\sup_{\bm\zeta\in [\mathbb{P}_k(K^{e})]^{d}\cdot\bm n_e\backslash\{\bm 0\}}\frac{\norm{\bm\zeta}_{K_{ext}^{e}}}{\norm{\bm \zeta}_{K^{e}}}\,, \;\qquad \text{ and } \;\qquad
        C^{inv}_{e} :=\, h_e^{\perp} \,\sup_{\bm \zeta\in [\mathbb{P}_k(K^{e})]^{d}\cdot \bm n_{e}\backslash\{\bm 0\}}\frac{\norm{\partial_{\bm n_e}\bm \zeta}_{K^{e}}}{\norm{\bm \zeta}_{K^{e}}}\,,
\]
where $\mathbb{P}_k(D)$ denotes the set of polynomials of degree at most $k$ over the domain $D$. As proven in \cite{CoQiuSo2014}, these functions can be bounded globally in terms of the polynomial degree of the approximation $k$ and the mesh regularity parameter as

\begin{subequations}
\label{constants_C}
\begin{minipage}{0.49\linewidth}
\begin{equation}
\label{constants_C_a}
C^{ext}_{e} \leq C_1(k+1)^2(3\beta+2)^k
\end{equation}
\end{minipage}
\begin{minipage}{0.49\linewidth}
\begin{equation}
\label{constants_C_b}
C^{inv}_e \leq C_2k^2
\end{equation}
\end{minipage}
\end{subequations}

Finally, we introduce the mesh--dependent function $\tau: \mathcal E_h \to \mathbb R$, which will act as a stabilization parameter and will be used in the analysis of sections \ref{sec:hdg_method} and \ref{sec:a_priori_error_estimates}. In the most general case, this stabilization function needs only to be strictly positive and essentially bounded (i.e. $\tau\in L^\infty(\mathcal E_h)$ ). However, in this work we will consider it to be simply any  positive constant. 

With these definitions we are prepared to establish our next set of assumptions, which will be used from now on for the analysis in the following sections. 
\begin{enumerate} \setcounter{enumi}{4}
\item For every $e\in\mathcal E^\partial_h$, the maximum distance between $\Gamma$ and $\Gamma_h$ satisfies
\begin{equation}\label{eq:Hbound}
H^\perp_e \leq \left(4\tau \cdot \min\left\{1,\max_{\boldsymbol x \in \Gamma_h}(1+a^{-1}) \right\}\right)^{-1}\,.
\end{equation}
Regarding this somewhat esoteric condition we can state the following. Recalling that $a^{-1}$ is the reciprocal of the diffusivity coefficient, it informs about the possible formation of boundary layers---which are likely to be present if $a^{-1}$ is large. Therefore, the distance between the physical and computational boundaries must be aware about this value. On the other hand, the stabilization parameter $\tau$ penalizes the size of the jump of the discrete approximations across adjacent elements. A large value of $\tau$ will reduce the size of the jump and will make it more difficult to approximate a steep gradient (or a boundary layer) across an element. If this is the case, the distance between boundaries must be reduced.
\item The local proximity parameter must satisfy
\begin{equation}\label{eq:Rbound}
R < 2^{-1/3}\left(C^{inv}_e C^{ext}_e\right)^{-2/3}\,.
\end{equation}
In view of the estimates \eqref{constants_C} and the fact that---as we will soon explain in detail---we will extrapolate some approximations from $\Gamma_h$ to $\Gamma$, this estimate sets an acceptable ratio in terms of the polynomial degree of the approximation.
\end{enumerate}

\textbf{Transfer paths} \\
We must now describe how to transfer the boundary data from the problem boundary $\Gamma$ to the computational boundary $\Gamma_h$. We will refer to the straight segment from condition \ref{TransferPath} above as a \textit{transfer path} associated to $\bm x$. We will denote its unit tangent vector and its length respectively by
\[
\bm t(\bm x): = (\boldsymbol x - \overline{\boldsymbol x})/|\boldsymbol x - \overline{\boldsymbol x}|, \qquad \text{ and } \qquad l(\boldsymbol x) := |\boldsymbol x  - \overline{\boldsymbol x}|.
\]
This construction is represented schematically in the center--left panel of Fig. \ref{fig:transferpaths}.

Since $\boldsymbol \phi$ is a bijection, each boundary face $e\in\partial\mathcal E_h$ will be identified with its corresponding image in $\Gamma$, which we will denote by
\[
\Gamma_e := \{\overline{\boldsymbol x}\in \Gamma: \overline{\boldsymbol x} = \boldsymbol\phi(\boldsymbol x) \;\; \text{ for some }\; \boldsymbol x\in e \}.
\]
A depiction of the segment $\Gamma_e$ associated to an edge $e$ is shown in the center right panel of Fig. \ref{fig:transferpaths}. An algorithm for constructing a collection of transfer paths,  satisfying the additional condition that no transfer paths do not intersect each other, was developed in \cite{CoQiuSo2014} for the two dimensional case. The intersection--avoiding condition is not necessary for the analysis, but provides a simple and natural way to construct the extension patches that will be defined below.

\begin{figure}[tb]
\begin{tabular}{cccc}
\centering
\includegraphics[width=0.22\textwidth]{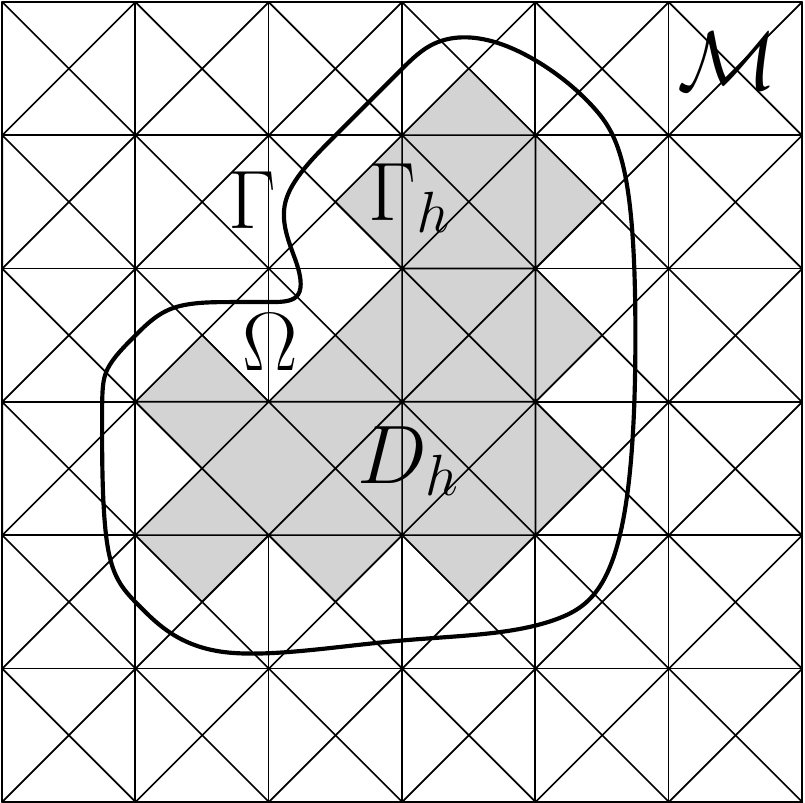} &
  \includegraphics[width=0.22\linewidth]{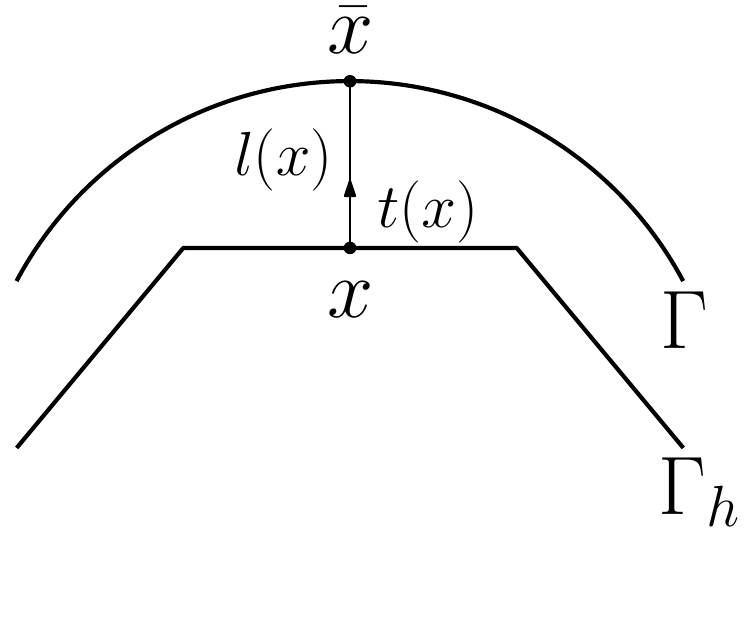}
&
\includegraphics[width=0.27\textwidth]{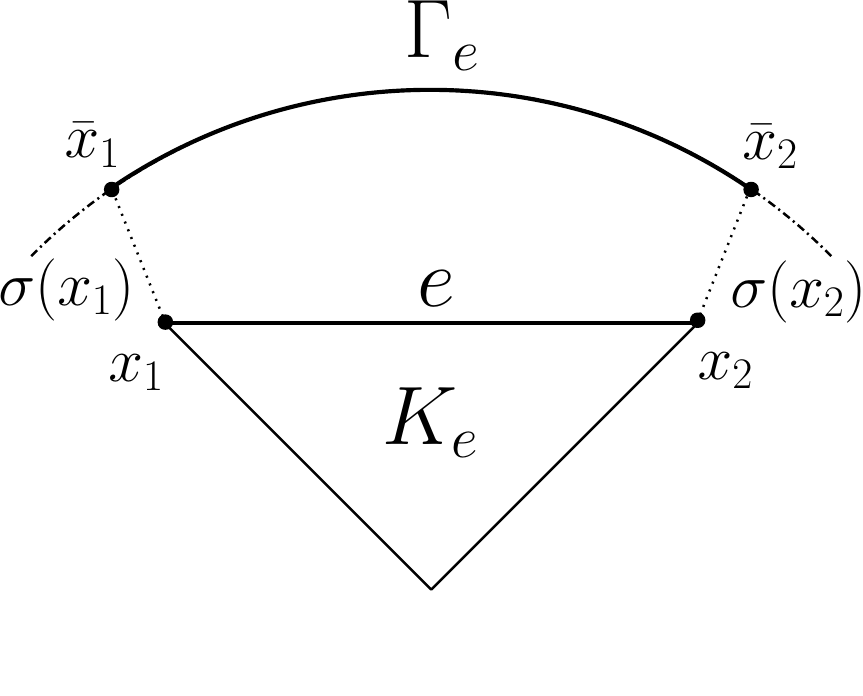}  &
  \includegraphics[width=0.2\linewidth]{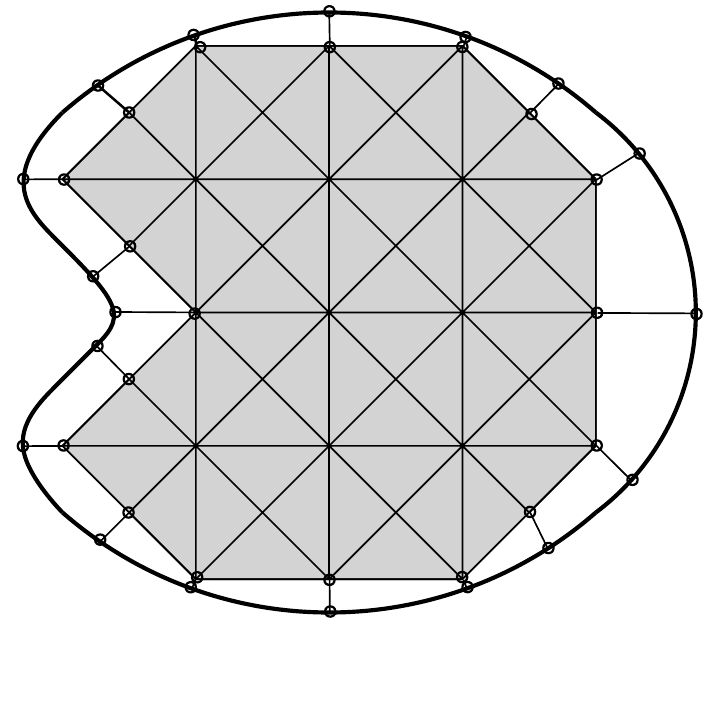} 
\end{tabular}
\caption{Left: Example of a domain $\Omega$, its boundary $\Gamma$, a background domain $\mathcal{M}$ and the construction of the computational domain $D_h$, shaded in gray. Center left: A transfer path associated to $\bm x$. Center right: The extension patch $K_e^{ext}$ is the region enclosed by the edge $e$, the segment $\Gamma_e$ and the transfer paths $\sigma(\boldsymbol x_1)$ and $\sigma(\boldsymbol x_2)$. Right: Tessellation of the full domain. The extension patches are the white tiles filling the space between $\Gamma$ and $\Gamma_h$.}
\label{fig:transferpaths}
\end{figure}
%
\textbf{Extension patches}\\
For computational purposes that will become clear soon, we will have to tessellate the complement of the computational domain $D_h^c\,:=\, \Omega\backslash\overline{D_h}$ in such a way that there is a one-to one correspondence between boundary edges $e\in \mathcal{E}_h^\partial$ and tiles in the tessellation.  We will denote by $K_{e}^{ext}$ the only tile that has $e$ as a face  and will refer to it as an \textit{extension patch}. One possible way of constructing such a tessellation is to use the algorithm from \cite{CoQiuSo2014} and define the extension patch $K_e^{ext}$ as the region bounded by the edge $e$, the transfer paths associated to the endpoints of $e$ and the segment of $\Gamma$ delimited by the transfer paths, as depicted in the two panels on the right of \eqref{fig:transferpaths}.

If $K_e$ is the element that shares the edge $e$ with the extension patch $K^{ext}_e$ and $p$ is a polynomial defined over $K_e$, we define the extrapolation of $p$ from $K_e$ to $K_e^{ext}$ by
\[
E_h(p)(y) \,:=\, p(y) \;\; \text{for all } \;\;y \in K_e^{ext}.
\]
To simplify notation, whenever possible we will simply notation by simply writing $p(y)$ understanding that the extrapolation operator is used tacitly.

%
\section{Unfitted HDG discretizations}\label{sec:hdg_method}
%
In this section we present and analyze HDG schemes for the state, adjoint and deformation field equations---the latter inspired by \cite{QiSoVe2016}. All the schemes will make use of the transfer path technique \cite{CoSo2012,CoQiuSo2014} introduced in the previous section. This will allow us to pose the HDG discretization using only polyhedral elements, even in domains which are not necessarily polyhedral.

The discretization will use the following global polynomial spaces:
\begin{subequations}
    \begin{alignat}{6}
        \bm Z_h&\,:=\,\left\{\bm v\in [L^{2}(\mathcal{T}_{h})]^{d}\,:\, \bm v|_{K}\in [\mathbb{P}_{k}(K)]^{d}\quad \forall\, K\in \mathcal{T}_h\right\}\,,\\
        W_h&\,:=\,\left\{w\in L^{2}(\mathcal{T}_h)\,:\,w|_{K}\in \mathbb{P}_k(K)\quad \forall\, K\in \mathcal{T}_h\right\}\,,\\
        M_h &\,:=\, \left\{\mu \in L^{2}(\mathcal{E}_h)\,:\, \mu|_{e}\in \mathbb{P}_{k}(e)\quad \forall \,e\in\mathcal{E}_h \right\}\,.
    \end{alignat}
\end{subequations}
We will also use the following inner products associated to $\mathcal{T}_h$ and $\partial\mathcal{ T}_h$
\[
(u,v)_{D_h}\,:=\, \sum_{K\in \mathcal{T}_h}\int_{K}u\,v \qquad \mbox{ and } \qquad  \langle s,t\rangle_{\mathcal E_h}\,:=\, \sum_{K\in\mathcal T_h} \int_{\partial K}s\,t\,,
\]
which induce the norms
\[
\norm{u}_{D_h}\,:=\, (u,u)_{D_h}^{1/2} \qquad \text{ and } \qquad \norm{s}_{\mathcal E_h}\,:=\,\langle s, s\rangle_{\mathcal E_h}^{1/2}.
\]
If $w$ is a positive function defined along every edge of the triangulation, we define the weighted norms
\begin{equation}\label{eq:weightedNorms}
    \norm{v}_{\mathcal E_h,w}
    \,:=\, \|w^{1/2}v\|_{\mathcal E_h}
    \qquad \mbox{ and } \qquad 
    \norm{v}_{\Gamma, w}
    \,:=\, \left(\sum_{e\subset \Gamma_h}\int_e w\,|v|^2\right)^{1/2}\,.
\end{equation}
The same notation will be used for tensor- and vector-valued polynomial functions defined on $K^e$.

\textbf{State and adjoint equations}\\
As proposed in \cite{CoQiuSo2014}, we consider the strong mixed formulations for the state and adjoint equations, restricted to the polygonal domain $D_h$:

\begin{subequations}
\begin{minipage}{0.45\textwidth}
    \label{HDG_mixed_formualtion_state}
    \begin{align}
        \label{HDG_mixed_formualtion_state_a}
        a^{-1}\,\bm p \,+\,\nabla y &\,=\, 0 \quad& \mbox{in }D_h\,,\\
        \label{HDG_mixed_formualtion_state_b}
        \nabla\cdot \bm p &\,=\, f \quad& \mbox{in }D_h\,,\\
        \label{HDG_mixed_formualtion_state_c}
        y &\,=\, \varphi_1\quad& \mbox{on }\Gamma_h\,,
    \end{align}
\end{minipage}
\end{subequations}
\begin{subequations}
\begin{minipage}{0.45\textwidth}
    \label{HDG_mixed_formualtion_adjoint}
    \begin{align}
        \label{HDG_mixed_formualtion_adjoint_a}
        a^{-1}\,\bm r\,+\,\nabla z &\,=\, 0\quad& \mbox{in }D_h\,,\\
        \label{HDG_mixed_formualtion_adjoint_b}
        \nabla\cdot\bm r\,-\, y &\,=\, -\widetilde{y}\quad& \mbox{in }D_h\,,\\
        \label{HDG_mixed_formualtion_adjoint_c}
        z&\,=\, \varphi_2\quad& \mbox{on }\Gamma_h\,,
    \end{align}
\end{minipage}
\end{subequations}

where the unknowns $\varphi_1$ and $\varphi_2$ correspond to the traces of $y$ and $z$, respectively, on $\Gamma_h$. Let $\bm x\in \Gamma_h$ and $\bar {\bm x}$ its corresponding point on $\Gamma$. By integrating \eqref{HDG_mixed_formualtion_state_a} and \eqref{HDG_mixed_formualtion_adjoint_a} along the transfer path joining $\bm x$ and $\bar{\bm x}$, we obtain
    \begin{equation} \label{HDG_extension}
        \varphi_{1}(\bm x) \,:=\, g(\bar{\bm x})\,+\,\int_{0}^{l(\bm x)}a^{-1}\,\bm p\,(\bm x \,+\, s\,\bm t)\cdot\bm t\,ds \quad\;\mbox{ and }\;\quad
        \varphi_{2}(\bm x) \,:=\, \int_{0}^{l(\bm x)}a^{-1}\,\bm r\,(\bm x \,+\, s\,\bm t)\cdot\bm t\,ds\,. 
    \end{equation}
The HDG method for the state equation seeks an approximation $(\bm p_h, y_h,\widehat{y}_h)$ of the exact solution $(\bm p, y,y|_{\mathcal{E}_h})$ in the space $\bm Z_h\times W_h\times M_h$ satisfying
\begin{subequations}\label{HDG_state_eq}
    \begin{alignat}{6}
        \label{HDG_state_eq_a}
        (a^{-1}\,\bm p_h, \bm v_1)_{D_h}
        \,-\,(y_h,\nabla\cdot \bm v_1)_{D_h}
        \,+\,\langle\widehat{y}_h,\bm v_1\cdot\bm n_h\rangle_{\mathcal{E}_h} 
        & \,=\, 0\,,\\
        \label{HDG_state_eq_b}
        -\,(\bm p_h,\nabla w_1)_{D_h} 
        \,+\,\langle\widehat{\bm p}_h\cdot \bm n_h, w_1 \rangle_{\mathcal{E}_h} 
        & \,=\, (f,w_1)_{D_h}\,,\\
        \label{HDG_state_eq_c}
        \langle\widehat{\bm p}_h\cdot\bm n_h,\mu_1 \rangle_{\mathcal{E}_h\backslash \Gamma_h} 
        &\,=\, 0\,,\\
        \label{HDG_state_eq_d}
        \langle\widehat{y}_h,\mu_1 \rangle_{\Gamma_h} 
        &\,=\, \langle \varphi_1^h,\mu_1\rangle_{\Gamma_h}\,,
    \end{alignat}
\end{subequations}
for all $(\bm v_1,w_1,\mu_1)\in \bm Z_h\times W_h\times M_h$. In turn, the HDG method for the adjoint equation seeks and approximation $(\bm r_h,z_h,\widehat{z}_h)$ of the exact solution $(\bm r,z,z|_{\mathcal{E}_h})$ in the space $\bm Z_h\times W_h\times M_h$ such that
\begin{subequations}\label{HDG_adjoint_eq}
    \begin{alignat}{6}
        \label{HDG_adjoint_eq_a}
        (a^{-1}\,\bm r_h, \bm v_2)_{D_h}
        \,-\,(z_h,\nabla\cdot \bm v_2)_{D_h}
        \,+\,\langle\widehat{z}_h,\bm v_2\cdot\bm n\rangle_{\mathcal{E}_h} 
        & \,=\, 0\,,\\
        \label{HDG_adjoint_eq_b}
        -\,(\bm r_h,\nabla w_2)_{D_h} 
        \,+\,\langle\widehat{\bm r}_h\cdot \bm n_h, w_2 \rangle_{\mathcal{E}_h} 
        &\,=\, (y_h-\widetilde{y},w_2)_{D_h}\,,\\
        \label{HDG_adjoint_eq_c}
        \langle\widehat{\bm r}_h\cdot\bm n_h,\mu_2 \rangle_{\mathcal{E}_h\backslash \Gamma_h} 
        &\,=\, 0\,,\\
        \label{HDG_adjoint_eq_d}
        \langle\widehat{z}_h,\mu_2 \rangle_{\Gamma_h} 
        &\,=\, \langle \varphi_2^h,\mu_2\rangle_{\Gamma_h}\,,
    \end{alignat}
\end{subequations}
for all $(\bm v_2,w_2,\mu_2)\in \bm Z_h\times W_h\times M_h$. The functions  $\varphi_1^h$ and $\varphi_2^h$ appearing in the systems above are discrete analogs of \eqref{HDG_extension}. They transfer the Dirichlet boundary conditions from $\Gamma$ to $\Gamma_h$ and are defined along all faces in $\mathcal{E}_h^\partial$, by
\begin{alignat}{6}
\begin{split}
    \label{HDG_extensions_eq}
    \varphi_1^h(\bm x)&\,:=\, g(\bar{ \bm x})\,+\,\int_{0}^{l(\bm x)} a^{-1}\,E_h(\bm p_h)(\bm x\,+\,s\,\bm t(\bm x))\cdot\bm t(\bm x) \,ds\,,\\
    \varphi_2^h(\bm x)&\,:=\, \int_{0}^{l(\bm x)} a^{-1}\,E_h(\bm r_h)(\bm x\,+\,s\,\bm t(\bm x))\cdot\bm t(\bm x) \,ds\,,
\end{split}
\end{alignat}

where $E_h(\bm p_h)$ and $E_h(\bm r_h)$ are the extrapolation of $\bm p_h$ and $\bm r_h$ respectively. The numerical fluxes $\widehat{\bm p}_h$ and $\widehat{\bm r}_h$ are defined along every edge $e\in\mathcal E_h$ by
\begin{align}
        \label{Fluxes}
        \widehat{\bm p}_h \,=\, \bm p_h \,+\,\tau\, (y_h\,-\,\widehat{y}_h)\,\bm n_h\qquad\mbox{and}\qquad
        \widehat{\bm r}_h \,=\, \bm r_h \,+\,\tau\, (z_h\,-\,\widehat{z}_h)\,\bm n_h\,,
\end{align}
and $\tau$ is a positive and bounded stabilization function defined on $\mathcal{E}_h$, whose maximum and minimum values will be denoted respectively by $\overline{\tau}$ and $\underline{\tau}$.

The unfitted schemes in \eqref{HDG_state_eq} and \eqref{HDG_adjoint_eq} were studied in \cite{CoQiuSo2014}, where their well-posedness was established. We will use this result without not repeating the argument here. Instead we will now move on to study the discretization of the deformation field.

\textbf{The deformation field equation}\\
We now present the HDG scheme for the deformation field equation which is inspired by the work done on \cite{QiSoVe2016}. It shall be noted that, due to the presence of a Neumann boundary condition, the treatment cannot be the same as in the cases of the state and adjoint equations. In particular, a transfer function in the style of \eqref{HDG_extension} to transfer the Neumann data from the curved boundary $\Gamma$ to the polygonal computational boundary $\Gamma_h$ is not available. Instead, we will make use of the bijection $\bm\phi$ defined on \eqref{eq:Phi} to impose the Neumann condition on the computational boundary. Under these conditions, the deformation field equation can be written in the computational domain $D_h$ as follows:
\begin{subequations}
\label{mixed_Dh_velocity}
\begin{alignat}{6}
    \label{mixed_Dh_velocity_a}
    \bm\sigma \,+\,\nabla\bm V& \,=\, 0 &\quad\mbox{in }D_h\,,\\
    \label{mixed_Dh_velocity_b}
    \bdiv(\bm\sigma)& \,=\, \bm 0 &\quad\mbox{in }D_h\,,\\
    \label{mixed_Dh_velocity_c}
    \bm\sigma\,\bm n_h& \,=\, \bm g_{N} &\quad\mbox{on }\Gamma^{N}_h\,,\\
    \label{mixed_Dh_velocity_d}
    \bm V &\,=\, \bm g_{D} & \quad \mbox{on }\Gamma^{D}_h\,.
\end{alignat}
\end{subequations}
We assume that the computational boundary $\Gamma_h$ is split between $\Gamma_h^{D}$ (the part of $\Gamma_h$ with Dirichlet datum) and $\Gamma_h^{N}$ (the part of $\Gamma_h$ with Neumann datum) in such a way that
\[
\Gamma_h \,=\,\Gamma_h^{D}\,\cup\, \Gamma_h^{N} \qquad \text{ and }  \qquad \Gamma_h^{D}\,\cap\,\Gamma_h^{N}\,=\, \emptyset.
\]
The Neumann datum $\bm g_N$ is defined by
\[
    \bm g_N\,:=\,\left(G(\Gamma)\bm n\right)\circ\bm \phi\,, 
\]
Where $G(\Gamma)$ is as defined in \eqref{eq:G}, and the transfer function $\boldsymbol\phi$ is the one defined on \eqref{eq:Phi}.  Recalling the mixed variables
\[
\nabla y \,=\, -\, a^{-1}\,\bm p \qquad \text{ and } \qquad \nabla z \,=\, -\,a^{-1}\,\bm r
\]
introduced in the state and adjoint mixed formulations, we can rewrite $G(\Gamma)$ as
\begin{equation}\label{eq:G(Gamma)}
    G(\Gamma)\,=\,\bm r\cdot\bm n(a^{-1}\,\bm p\cdot\bm n\,+\,\partial_{\bm n} g)\,+\,\dfrac{1}{2}(g\,-\,\widetilde{y})^2\,.
\end{equation}
The Dirichlet datum $\bm g_D$ is transferred using the same technique used in the state and adjoint equations, that is,
\[
    \bm g_{D}(\bm x)\,:=\, \int_{0}^{l(\bm x)}\bm\sigma(\bm x\,+\,s\,\bm t)\,\bm t\,ds\,.
\]
Before presenting the discrete scheme, we define the vector--valued polynomial spaces used for the discretization, which are defined as
\begin{subequations}
\begin{align*}
    \mathbb{Z}_h &\,:=\, \{\bm \xi\in [L^{2}(\mathcal{T}_h)]^{d\times d}\,:\, \bm \xi|_{K}\in [\mathbb{P}_k(K)]^{d\times d}\,,\quad \forall \,K\in \mathcal{T}_h\}\,,\\
    \bm W_h &\,:=\, \{\bm w \in [L^{2}(\mathcal{T}_h)]^{d}\,:\, \bm w|_{K}\in [\mathbb{P}_{k}(K)]^{d}\,,\quad \forall\, K\in \mathcal{T}_h\}\,,\\
    \bm M_h &\,:=\, \{\bm \mu\in [L^{2}(\mathcal{E}_h)]^{d}\,:\,\bm \mu|_{e}\in [\mathbb{P}_{k}(e)]^{d}\,,\quad \forall\, e\in \mathcal{E}_h\}\,.
\end{align*}
\end{subequations}
Then, our discrete scheme seeks an approximation $(\bm\sigma_h,\bm V_h,  \widehat{\bm V}_h)\in \mathbb{Z}_h\times \bm W_h\times \bm M_h$ of the exact solution $(\bm \sigma, \bm V_h, \bm V_h|_{\mathcal{E}_h})$, which is given by
\begin{subequations}
\label{hdg_scheme_V_1}
\begin{alignat}{6}
    \label{hdg_scheme_V_1_a}
    (\bm \sigma_h,\bm \psi)_{D_h}
    \,-\,(\bm V_h,\bdiv(\bm \psi))_{D_h}
    \,+\,\langle \widehat{\bm V}_h,\bm\psi\,\bm n_h\rangle_{\mathcal{E}_h}
    &\,=\, 0\,,\\
    \label{hdg_scheme_V_1_b}
    -\,(\bm \sigma_h,\nabla\bm w)_{D_h}
    \,+\,\langle\widehat{\bm \sigma}_h\,\bm n_h,\bm w\rangle_{\mathcal{E}_h} 
    &\,=\, 0\,,\\
    \label{hdg_scheme_V_1_c}
    \langle\widehat{\bm \sigma}_h\,\bm n_h,\bm\mu\rangle_{\mathcal{E}_h\backslash\Gamma_h}
    &\,=\, 0\,,\\
    \label{hdg_scheme_V_1_e}
    \langle\widehat{\bm\sigma}_h\,\bm n_h,\bm\mu\rangle_{\Gamma_h^{N}}
    &\,=\, \langle (G_h(\Gamma)\bm n)\circ\bm\phi,\bm\mu\rangle_{\Gamma_h^N}\,,\\
    \label{hdg_scheme_V_1_d}
    \langle\widehat{\bm V}_h,\bm \mu\rangle_{\Gamma_h^{D}}
    &\,=\, \langle\bm g_h^D,\bm\mu\rangle_{\Gamma_h^{D}}\,,
\end{alignat}
for all $(\bm\psi, \bm w,\bm \mu)\in \mathbb{Z}_h\times\bm W_h\times \bm M_h$, where
\begin{align}
\label{trace_sigma_hat}
    \widehat{\bm\sigma}_h\,\bm n_h \,:=&\, \bm\sigma_h\,\bm n_h\,+\,\tau\,(\bm V_{h}\,-\,\widehat{\bm V}_h)\,,\\
\label{eq:G_h(Gamma)}
    G_h(\Gamma) \,:=&\, \bm r_h\cdot\bm n\,(a^{-1}\,\bm p_h\cdot\bm n\,+\,\partial_{\bm n}g)\,+\,\frac{1}{2}(g\,-\,\widetilde{y})^2\,,\\
\label{eq:g_hD}
    \bm g_h^D(\bm x)\,:=&\, \int_{0}^{l(\bm x)}E_h(\bm \sigma_h)(\bm x \,+\, s\,\bm t(\bm x))\,\bm t(\bm x)\, ds\,.
\end{align}
\end{subequations}
From now on, for the sake of simplicity, we assume $\bm t(\bm x) = \bm n_h$ for all $\bm x \in e$ and all $e\subset \Gamma_h^D$. Otherwise, the results follow from assuming $\bm t\cdot \bm n_h>0$ and analyzing the tangential and normal components separately. 

\begin{theorem}
If $\mathcal T_h$ is a sufficiently fine admissible triangulation, then there exists a unique solution of the HDG scheme \eqref{hdg_scheme_V_1} associated to the deformation field equation.
\end{theorem}
\begin{proof}
We will use the Fredholm alternative. Let us start by assuming that $G_h(\Gamma)\,=\, 0$ and by taking test functions
\[
\bm \psi \,=\,\bm \sigma_h\,,\qquad \bm w \,=\, \bm V_h\,, \quad \text{ and } \quad \boldsymbol\mu = \begin{cases} \widehat{\boldsymbol V}_h \,\text{ in }\, \mathcal E_h\setminus \Gamma_h^D \\[.5ex] \widehat{\boldsymbol \sigma}_h\boldsymbol n_h \,\text{ in }\, \Gamma_h^D \end{cases}.
\]
Substituting the chosen value of $\boldsymbol \mu$ into \eqref{hdg_scheme_V_1_c}, \eqref{hdg_scheme_V_1_d}, and \eqref{hdg_scheme_V_1_e}, and adding the resulting expressions shows that
\begin{equation}
\label{sec4-aux1}
\langle  \widehat{\bm \sigma}_h\,\bm n_h, \widehat{\bm V}_h\rangle_{\mathcal{E}_h}\,=\,\langle\widehat{\bm \sigma}_h\,\bm n_h,\bm g_h^D\rangle_{\Gamma_h^D}.
\end{equation}
Then, integrating by parts \eqref{hdg_scheme_V_1_b} and adding the resulting expression to \eqref{hdg_scheme_V_1_a}, we get
\begin{alignat}{6}
\nonumber
    0 =\,& \norm{\bm\sigma_h}^{2}_{D_h}
    \,+\,\langle(\widehat{\bm\sigma}_h-\bm\sigma_h)\bm n_h,\bm V_h\rangle_{\mathcal{E}_h}
    \,+\,\langle\bm\sigma_h\,\bm n_h,\widehat{\bm V}_h\rangle_{\mathcal{E}_h} &&\\[1ex]
    \nonumber
   =\,& \norm{\bm\sigma_h}_{D_h}^{2}
    \,+\,\norm{\tau^{1/2}\,(\bm V_h-\widehat{\bm V}_h)}_{\mathcal{E}_h}^2
    \,+\,\langle \widehat{\bm\sigma}_h\,\bm n_h,\widehat{\bm V}_h\rangle_{\mathcal{E}_h} \qquad \qquad&& \text{\small (Using \eqref{trace_sigma_hat})}\\[1ex]
    \label{aux_1:proof_EyU_V}
    =\,& \norm{\bm\sigma_h}_{D_h}^{2}
    \,+\,\norm{\tau^{1/2}\,(\bm V_h-\widehat{\bm V}_h)}_{\mathcal{E}_h}^2
    \,+\,\langle \widehat{\bm\sigma}_h\,\bm n_h,\boldsymbol g_h^D\rangle_{\Gamma_h^D} \qquad \qquad&& \text{\small (By \eqref{sec4-aux1})}.
\end{alignat}
On the other hand, adding and subtracting $\bm \sigma_h$ in the definition of $\bm g_h^D$, we can write
\[
    \bm g_h^D(\bm x)
    \,=\, \int_0^{l(\bm x)}\left(E_h(\bm \sigma_h)(\bm x\,+\,s\,\bm n_h)\,-\,\bm \sigma_h(\bm x)\right)\cdot\bm n_h\,ds\,+\,\bm\sigma_h(\bm x)\,\bm n_h \,l(\bm x)\,,
\]
which leads to the expression
\[
\bm\sigma_h(\bm x)\,\bm n_h\,=\,l^{-1}(\bm x)\,\bm g_h^D(\bm x)\,-\,\Lambda^{\bm\sigma_h}(\bm x),
\]
where, to keep notation compact, we have defined the term
\begin{equation}
\label{def:Lambda}
\Lambda^{\bm\sigma_h}(\bm x) : = \frac{1}{l(\boldsymbol x)}\int_0^{l(\bm x)}\left(\boldsymbol{E}(\bm \sigma_h)(\bm x\,+\,s\,\bm n_h)\,-\,\bm \sigma_h(\bm x)\right)\cdot\bm n_h\,ds\,.
\end{equation}
Then,substituting the expression for $\bm\sigma_h(\bm x)\,\bm n_h$ obtained above into \eqref{trace_sigma_hat}, we have that
\begin{align*}
    \langle\widehat{\bm\sigma}_h\,\bm n_h,\bm g_h^D\rangle_{\Gamma_h^{D}}
    &\,=\,\norm{l^{-1/2}\,\bm g_h^D}_{\Gamma_h^{D}}^2+\langle l^{1/2}(\tau\,(\bm V_h-\widehat{\bm V}_h) - \Lambda^{\boldsymbol\sigma_h}),l^{-1/2}\,\bm g_h^D\rangle_{\Gamma_h^{D}}\,.
\end{align*}
Substituting this into \eqref{aux_1:proof_EyU_V} we find that
\[
         \norm{\bm\sigma_h}_{D_h}^{2}
         \,+\,\norm{\tau^{1/2}\,(\bm V_h-\widehat{\bm V}_h)}_{\mathcal{E}_h}^2
         \,+\,\norm{l^{-1/2}\,\bm g_h^D}_{\Gamma_h^{D}}^{2} 
         \,=\,\langle l^{1/2}(  \Lambda^{\boldsymbol\sigma_h}-\tau\,(\bm V_h-\widehat{\bm V}_h)),l^{-1/2}\,\bm g_h^D\rangle_{\Gamma_h^{D}}\,.
\]
Now, we use Young's inequality to bound the right-hand side of the expression above as
\begin{align*}
   \langle l^{1/2}(  \Lambda^{\boldsymbol\sigma_h}-\tau\,(\bm V_h-\widehat{\bm V}_h)),l^{-1/2}\,\bm g_h^D\rangle_{\Gamma_h^{D}} \leq\,& \tfrac{1}{2}\left(\|l^{1/2}\Lambda^{\boldsymbol\sigma_h}\|_{\Gamma_h^D}^2 + \|l^{1/2}\tau(\boldsymbol V_h - \widehat{\boldsymbol V}_h)\|_{\Gamma_h^D}^2 + \|l^{-1/2}\boldsymbol g_h^D\|_{\Gamma_h^D}^2\right)\\[1ex]
   \leq\,& \tfrac{1}{2}\left( H^\perp_e\|\Lambda^{\boldsymbol\sigma_h}\|_{\Gamma_h^D}^2 + H^\perp_e\|\tau(\boldsymbol V_h - \widehat{\boldsymbol V}_h)\|_{\Gamma_h^D}^2 + \|l^{-1/2}\boldsymbol g_h^D\|_{\Gamma_h^D}^2\right).
\end{align*}
From the last two expressions it follows that
\begin{alignat*}{6}
 \norm{\bm\sigma_h}_{D_h}^{2}
         \!+\!(1-\tfrac{1}{2}H^\perp_e)\left(\norm{\tau^{1/2}(\bm V_h-\widehat{\bm V}_h)}_{\mathcal{E}_h}^2
         +\norm{l^{-1/2}\,\bm g_h^D}_{\Gamma_h^{D}}^{2}\right)
         \leq\,& \tfrac{1}{2}H^\perp_e \|\Lambda^{\boldsymbol\sigma_h}\|_{\Gamma_h^D}^2 && \\[1ex] 
         \leq\,&  \tfrac{1}{6} H^\perp_e r_e^3(C^{ext}_e C^{inv}_{e})^2\norm{\bm \sigma_h}_{D_h}^2 \qquad \quad&& \text{\small (By \eqref{est:Lambda2})}\\[1ex]
         \leq\,& \tfrac{1}{2}H^\perp_e\norm{\bm \sigma_h}_{D_h}^2  && \text{\small (By \eqref{eq:Rbound})}.
\end{alignat*}
Hence, the local proximity condition \eqref{eq:localProximity} implies that if the mesh is sufficiently fine
\begin{subequations}
\begin{alignat}{6}
\label{eq:sigma0}
\bm \sigma_h =\,& \boldsymbol 0 \qquad&& \text{ in }D_h\,, \\
\label{eq:V=Vhat}
\bm V_h =\,& \widehat{\bm V}_h \qquad&& \text{ on } \mathcal{E}_h\,, \\
\label{eq:gD0}
\bm g_h^D =\,& \boldsymbol 0 \qquad&& \text{ on }\Gamma_h^{D}.
\end{alignat}
\end{subequations}
From here, we see that integrating by parts the second term in equation \eqref{hdg_scheme_V_1_a}, letting $\boldsymbol\psi=\nabla\boldsymbol V_h$, and using \eqref{eq:sigma0}, it follows that $\|\nabla\boldsymbol V_h\|_{D_h} = 0$. On the other hand, \eqref{eq:V=Vhat} and \eqref{eq:gD0} together with \eqref{hdg_scheme_V_1_d} imply
\[
\langle\bm V_h,\bm \mu \rangle_{\Gamma_h^{D}}\,=\,\langle\widehat{\bm V}_h,\bm \mu \rangle_{\Gamma_h^{D}}\,=\,\langle\bm g_h^D,\bm \mu\rangle_{\Gamma_h^{D}}\,=\,0
\]
for all $\mu \in \bm M_h$. Hence, $\bm V_h \,=\, 0 $ in $D_h$. Since the only solution to the homogeneous problem is the trivial one, the problem is uniquely solvable by the Fredholm alternative.
\qed\end{proof}

\section{\textit{A priori} error estimates.}\label{sec:a_priori_error_estimates}

For the error analysis we will make use of the HDG projection introduced in \cite{CoGoSa2010} and summarized in the Appendix \ref{sec:HDGprojection} for convenience. For a discrete space $X_h$ we will denote its HDG projector by $\Pi_X$ and the $L^2$ projector by $P_X$ and use them to define the projections of the error for all our unknowns:
{\small \begin{alignat*}{12}
\bm\varepsilon_{\bm p}\,:=\,& \bm\Pi_{\bm Z}\bm{p}\,-\,\bm{p}_h\,,  
\qquad&& \varepsilon_y\,:=\,&\Pi_W y\,-\,y_h\,, 
\qquad&& \varepsilon_{\widehat{y}} \,:=\,& P_M y\,-\,\widehat{y}_h\,,
\qquad&& \bm\varepsilon_{\widehat{\bm p}} \,:=\,& P_M \bm{p}\,-\, \widehat{\bm{p}}_h\,, \qquad&& \text{\footnotesize (State)} \\[1ex]
\bm\varepsilon_{\bm r}\,:=\,& \bm\Pi_{\bm Z}\bm{r}\,-\,\bm{r}_h\,,  
\qquad&& \varepsilon_z\,:=\,&\Pi_W z\,-\,z_h\,, 
\qquad&& \varepsilon_{\widehat{z}} \,:=\,& P_M z\,-\,\widehat{z}_h\,,
\qquad&& \bm\varepsilon_{\widehat{\bm r}} \,:=\,& P_M \bm{r}\,-\, \widehat{\bm{r}}_h\,, \qquad&& \text{\footnotesize (Adjoint)} \\[1ex]
\bm\varepsilon_{\bm \sigma}\,:=\,& \bm\Pi_{\mathbb Z}\bm{\sigma}\,-\,\bm{\sigma}_h\,,  
\qquad&& \varepsilon_{\bm V}\,:=\,&\Pi_{\bm W} \bm V\,-\,\bm V_h\,, 
\qquad&& \varepsilon_{\widehat{\bm V}} \,:=\,& P_{\bm M} \bm V\,-\,\widehat{\bm V}_h\,,
\qquad&& \bm\varepsilon_{\widehat{\bm \sigma}} \,:=\,& P_{\bm M} \bm{\sigma}\,-\, \widehat{\bm{\sigma}}_h\,, \qquad&& \text{\footnotesize (Deformation)} 
\end{alignat*} }
We also define the interpolation errors as
\begin{alignat*}{10}
     I_{\bm p}\,:=\,& \bm p\,-\,\bm\Pi_{\bm Z}\bm p\,,
     \qquad \quad &&  I_{\bm r}\,:=\,& \bm r\,-\,\bm\Pi_{\bm Z}\bm r\,,
     \qquad \quad &&  \boldsymbol{I}_{\bm \sigma}\,:=\,& \bm\sigma\,-\,\bm\Pi_{\mathbb{Z}}\bm\sigma\,, \\
      I_{y}\,:=\,& y\,-\,\Pi_W{y}\,, \qquad \quad && I_{z}\,:=\,& z\,-\,\Pi_{W}z\,, \qquad \quad && I_{\bm V}\,:=\,&\bm V\,-\,\bm\Pi_{\bm W}\bm V.
\end{alignat*}

\subsection{Error estimates for the state and adjoint equations}\label{sec:ErrorEstimates}
As mentioned before, the analysis for unffitted HDG schemes for independent equations of the type satisfied by the state and adjoint variables were analyzed in \cite{CoQiuSo2014}. The main difference is that, for the present case, the problems are coupled through the adjoint boundary condition \eqref{HDG_mixed_formualtion_adjoint_b}. This is reflected in the presence of both variables in some of the error estimates below. Nevertheless, the results from \cite{CoQiuSo2014} carry over in almost straightforward manner and we shall not repeat the arguments here.  Instead, we summarize the error estimates in the following theorem.

\begin{theorem}[\cite{CoQiuSo2014}{\,Theorem 2.1}]
    \label{theorem:estimates_Omega_p.and.r}
    Let $(y,\boldsymbol p)$ be a solution pair to the state problem \eqref{HDG_mixed_formualtion_state} and $(z,\boldsymbol r)$ be a solution pair to the adjoint problem \eqref{HDG_mixed_formualtion_adjoint}, and $(y_h,\boldsymbol p_h)$ and $(z_h,\boldsymbol r_h)$ be solutions to the corresponding HDG schemes \eqref{HDG_state_eq}  and \eqref{HDG_state_eq}. We have that
    \begin{align*}
        \norm{\bm p-\bm p_h}_{\Omega}
        &\,\lesssim\,
    \norm{I_{\bm p}}_{D_h}
        \,+\,\norm{I_{\bm p}\cdot\bm n}_{\Gamma_h,h^{\perp}} 
        \,+\,h^{k+1}\,|\bm p|_{\bm H^{k+1}(\Omega)}\,,\\
        \norm{\bm r\,-\,\bm r_h}_{\Omega}
        &\,\lesssim \,\norm{I_{\bm r}}_{D_h}
        \,+\,\norm{I_{\bm r}\cdot\bm n}_{\Gamma_h,h^{\perp}}
        \,+\,h^{k+1}\,|\bm r|_{\bm H^{k+1}(\Omega)}
        \,+\,\norm{\varepsilon_h^y}_{D_h}
        \,+\,\norm{I_y}_{D_h}.
    \end{align*}
Moreover, recalling the definition of the proximity parameter $R := \max_{e\in\mathcal{E}_h^{\partial}}r_e$ and 
\[
H_1 := \left(h+\tau^{1/2}Rh+R^2h^{1/2}\right), \qquad
H_2:=\left(1+\tau^{1/2}Rh^{1/2}\right), \qquad
H_3 := \left(h+R^{1/2}h+R^{3/2}\,h^{1/2}\right)\,,
\]
then the following estimate holds
\begin{align*}
        \norm{y\,-\,y_h}_{\Omega}+\norm{z\,-\,z_h}_{\Omega}
        &\,\lesssim\, 
        H_1\,(\norm{I_{\bm p}}_{D_h}\,+\,\norm{I_{\bm r}}_{D_h})
        \,+\,R^{1/2}h\,(\norm{I_{\bm p}\cdot\bm n}_{\Gamma_h,h^{\perp}}\,+\,\norm{I_{\bm r}\cdot\bm n}_{\Gamma_h,h^{\perp}})\\
        &\quad\
        +H_2\,(\norm{I_{y}}_{D_h}+\norm{I_{z}}_{D_h})
        +\,H_3\,h^{k+1}\,(|\bm p|_{\bm H^{k+1}(\Omega)}+|\bm r|_{\bm H^{k+1}(\Omega)})\,.
    \end{align*}    
\end{theorem}

\begin{corollary}\label{corollary:estimates_p_lesssim}
If the stabilization parameter $\tau = \mathcal O(1)$, then,
\[
\norm{\bm p\,-\,\bm p_h}_{\Omega}\,\lesssim\, h^{k+1}\,, \qquad 
\norm{y\,-\,y_h}_{\Omega}\,\lesssim\, h^{k+1}\,, \qquad
\norm{\bm r\,-\,\bm r_h}_{\Omega}\,\lesssim \, h^{k+1}\,, \qquad \|z\,-\,z_h\|_{\Omega}\,\lesssim\,h^{k+1}\,.
\]
\end{corollary}

\subsection{\texorpdfstring{Error estimates for $\bm\sigma\,-\,\bm\sigma_h$}{Error estimates for Esigma}}
We start by recalling that $\bm n_h$ to denotes the unitary normal vector for $\Gamma_h$, while $\bm n$ denotes the unitary normal vector for $\Gamma$. To begin with the analysis of the error for the velocity field equation, we note that the projections of the errors satisfy the following equations
    \begin{subequations}
    \label{lemma:energy_arg_V}
    \begin{align}
        \label{lemma:energy_arg_V_a}
        (\bm\varepsilon_{\bm\sigma},\bm\psi)_{D_h}
        \,-\,(\bm\varepsilon_{\bm V},\bdiv(\bm\psi))_{D_h}
        \,+\,\langle\bm\varepsilon_{\widehat{\bm V}},\bm\psi\bm n_h\rangle_{\mathcal{E}_h} 
        &\,=\, 
        -\,(\boldsymbol{I}_{\bm\sigma},\bm\psi)_{D_h}\,,\\
        \label{lemma:energy_arg_V_b}
        (\bm \varepsilon_{\bm\sigma},\nabla\bm w)_{D_h}
        \,-\,\langle\bm\varepsilon_{\widehat{\bm\sigma}}\bm n_h,\bm w\rangle_{\mathcal{E}_h}
        &\,=\,0\,,\\
        \label{lemma:energy_arg_V_c}
        \langle\bm\varepsilon_{\widehat{\bm\sigma}}\bm n_h,\bm \mu\rangle_{\mathcal{E}_h\backslash\Gamma_h}
        &\,=\,0\,,\\
        \label{lemma:energy_arg_V_d}
        \langle\bm\varepsilon_{\widehat{\bm\sigma}}\bm n_h,\bm\mu\rangle_{\Gamma_h^{N}}
        &\,=\,
        \langle(G-G_h)\bm n\circ\bm\phi,\bm\mu\rangle_{\Gamma_h^{N}}\,,\\
        \label{lemma:energy_arg_V_e}
        \langle\bm\varepsilon_{\widehat{\bm V}},\bm\mu\rangle_{\Gamma_h^{D}}
        &\,=\,\langle\bm g^D-\bm g_h^D,\bm\mu\rangle_{\Gamma_h^{D}}\,,\\
        \label{lemma:energy_arg_V_f}
        \bm\varepsilon_{\widehat{\bm\sigma}}\,\bm n_h
        &\,=\, \bm\varepsilon_{\bm\sigma}\bm n_h
        +\tau\,(\bm\varepsilon_{\bm V}-\bm\varepsilon_{\widehat{\bm V}})\,,
    \end{align}
    \end{subequations}
for all $(\bm \psi,\bm w, \bm \mu)\in \mathbb{Z}_h\times \bm W_h\times \bm  M_h$.

The following lemma provides a useful representation of the normal component of the error on the deformation field flux $\boldsymbol\varepsilon_{\boldsymbol\sigma}$ .

\begin{lemma}\label{cor:identity_projerr_Esigma} 
Consider the definition in \eqref{def:Lambda}. The following equation holds
\begin{equation}
	\label{eq:identity_projerr_Esigma}
	\bm\varepsilon_{\bm\sigma}\,\bm n_h 
	\,=\,
	l^{-1}\,(\bm g^D - \bm g_h^D)
	\,-\,\Lambda^{\boldsymbol{I}_{\bm\sigma}}
	\,-\,\Lambda^{\bm\varepsilon_{\bm \sigma}}
	\,-\,\boldsymbol{I}_{\bm\sigma}\,\bm n_h\,.
\end{equation}
\end{lemma}
\begin{proof}
    Let us note that
    \begin{align*}
        \boldsymbol{g}^D - \boldsymbol{g}_h^D
        \,=\, \int_0^{l(\boldsymbol{x})} \boldsymbol{\sigma}(\boldsymbol{x} + s \,\boldsymbol{n}_h)\,ds
        \,-\,\int_0^{l(\boldsymbol{x})} \boldsymbol{E}_h(\boldsymbol{\sigma})(\boldsymbol{x} + s \,\boldsymbol{n}_h)\,ds
        \,=\,\int_0^{l(\boldsymbol{x})} (\boldsymbol{\sigma}-\boldsymbol{E}_h(\boldsymbol{\sigma}))(\boldsymbol{x} + s \,\boldsymbol{n}_h)\,ds\,.
    \end{align*}
    Adding and subtracting $\boldsymbol{\Pi}_{\mathbb{Z}}\boldsymbol\sigma$ in the integrand above we get
    \begin{align*}
        \boldsymbol{g}^D - \boldsymbol{g}_h^D
        &\,=\, \int_0^{l(\boldsymbol{x})} (I_{\boldsymbol{\sigma}} \, +\,\boldsymbol{\varepsilon}_{\boldsymbol{\sigma}})(\boldsymbol{x} + s \,\boldsymbol{n}_h)\,ds\\
        &\,=\, \int_0^{l(\boldsymbol{x})} (I_{\boldsymbol{\sigma}}(\boldsymbol{x} + s\,\boldsymbol{n}_h)- I_{\boldsymbol{\sigma}}(\boldsymbol{x}))\,\boldsymbol{n}_h\,ds
        \,+\,l(\boldsymbol{x})\,I_{\boldsymbol{\sigma}}(\boldsymbol{x})\,\boldsymbol{n}_h
        \,+\,\int_0^{l(\boldsymbol{x})} (\boldsymbol{\varepsilon}_{\boldsymbol{\sigma}}(\boldsymbol{x} + s\,\boldsymbol{n}_h)- \boldsymbol{\varepsilon}_{\boldsymbol{\sigma}}(\boldsymbol{x}))\,\boldsymbol{n}_h\,ds\\
        &\quad\,+\,l(\boldsymbol{x})\,\boldsymbol{\varepsilon}_{\boldsymbol{\sigma}}(\boldsymbol{x})\,\boldsymbol{n}_h\\
        &\,=\,l(\boldsymbol{x})(\Lambda^{I_{\boldsymbol{\sigma}}}
        \,+\,I_{\boldsymbol{\sigma}}\,\boldsymbol{n}_h
        \,+\,\Lambda^{\boldsymbol{\varepsilon}_{\boldsymbol{\sigma}}}
        \,+\,\boldsymbol{\varepsilon}_{\boldsymbol{\sigma}}\,\boldsymbol{n}_h)(\boldsymbol{x})\,,
    \end{align*}
    after a rearrangement of terms the result follows.
\qed\end{proof}

Now, we present a useful result for the analysis

\begin{lemma}\label{lemma:identity_V_1}
Let $\boldsymbol\varepsilon_{\boldsymbol\sigma},\boldsymbol\varepsilon_{\widehat{\boldsymbol \sigma}},  \boldsymbol\varepsilon_{\bm V}, \boldsymbol\varepsilon_{\widehat{\bm V}}$ be solutions to the system \eqref{lemma:energy_arg_V}; then the following holds:

\begin{equation}
    \label{eq:identity_V_1}
    \begin{array}{c}
        \displaystyle\norm{\bm \varepsilon_{\bm\sigma}}_{D_h}^{2}
        +\norm{\tau^{1/2}\,(\bm\varepsilon_{\bm V}-\bm\varepsilon_{\widehat{\bm V}})}_{\mathcal{E}_h}^2
        +\langle\bm g^D-\bm g_h^D,\bm\varepsilon_{\widehat{\bm\sigma}}\,\bm n_h\rangle_{\Gamma_h^D}
        +\langle(G-G_h)\bm n\circ\bm\phi,\bm\varepsilon_{\widehat{\bm V}}\rangle_{\Gamma_h^{N}}
        =-\,(\boldsymbol{I}_{\bm \sigma},\bm\varepsilon_{\bm\sigma})_{D_h}\,.
    \end{array}
\end{equation}
\end{lemma}
\begin{proof}
Setting $\bm\psi\,=\,\bm\varepsilon_{\bm\sigma}$ and $\bm w\,=\, \bm\varepsilon_{\bm V}$, integrating by parts and adding both equations we obtain
\begin{align*}
    \norm{\bm\varepsilon_{\bm\sigma}}_{D_h}^{2}
    \,+\,\langle\bm\varepsilon_{\widehat{\bm\sigma}}\,\bm n_h-\bm\varepsilon_{\bm\sigma}\,\bm n_h,\bm\varepsilon_{\bm V}\rangle_{\mathcal{E}_h}
    \,+\,\langle\bm\varepsilon_{\widehat{\bm V}},\bm\varepsilon_{\bm\sigma}\,\bm n_h\rangle_{\mathcal{E}_h}
    \,=\,
    -\,(\boldsymbol{I}_{\bm\sigma},\bm\varepsilon_{\bm\sigma})_{D_h}\,,
\end{align*}
then, by \eqref{lemma:energy_arg_V_f}
\begin{align*}
   \norm{\bm\varepsilon_{\bm\sigma}}_{D_h}^{2}
    \,+\,\norm{\tau^{1/2}\,(\bm\varepsilon_{\bm V}-\bm\varepsilon_{\widehat{\bm V}})}_{\mathcal{E}_h}^2
    \,+\,\langle\bm\varepsilon_{\widehat{\bm\sigma}}\,\bm n_h,\bm\varepsilon_{\widehat{\bm V}}\rangle_{\mathcal{E}_h}
    \,=\,
    -\,(\boldsymbol{I}_{\bm\sigma},\bm\varepsilon_{\bm\sigma})_{D_h}\,.
\end{align*}
Furthermore, note that by \eqref{lemma:energy_arg_V_c}, \eqref{lemma:energy_arg_V_d}, and \eqref{lemma:energy_arg_V_e} we obtain 
\begin{align*}
    \langle\bm\varepsilon_{\widehat{\bm\sigma}}\,\bm n_h,\bm\varepsilon_{\widehat{\bm V}}\rangle_{\mathcal{E}_h}
    &\,=\,
    \langle\bm\varepsilon_{\widehat{\bm\sigma}}\,\bm n_h,\bm\varepsilon_{\widehat{\bm V}}\rangle_{\mathcal{E}_h\backslash\Gamma_h}
    \,+\,\langle\bm\varepsilon_{\widehat{\bm\sigma}}\,\bm n_h,\bm\varepsilon_{\widehat{\bm V}}\rangle_{\Gamma_h^{D}}
    \,+\,\langle\bm\varepsilon_{\widehat{\bm\sigma}}\,\bm n_h,\bm\varepsilon_{\widehat{\bm V}}\rangle_{\Gamma_h^{N}}\\
    &\,=\,
    \langle\bm g^D - \bm g^{D}_h,\bm\varepsilon_{\widehat{\bm 
    \sigma}}\,\bm n_h\rangle_{\Gamma_h^{D}}
    \,+\,\langle(G-G_h)\bm n\circ\bm\phi,\bm\varepsilon_{\widehat{\bm V}}\rangle_{\Gamma_h^N}\,.
\end{align*}
and the result follows.
\qed\end{proof}

Let us now introduce a key lemma for the error analysis,
\begin{lemma}\label{lem:key_ineq_V}
Suppose that $\mathcal T_h$ is an admissible triangulation and define
\begin{equation}\label{def:energy-norm-V}
	\vertiii{(\bm\varepsilon_{\bm\sigma},\bm\varepsilon_{\bm V}-\bm\varepsilon_{\widehat{\bm V}},\bm g^D-\bm g_h^D)}
	\,:=\,\left(\norm{\bm\varepsilon_{\bm \sigma}}_{D_h}^2
	\,+\,\norm{\bm\varepsilon_{\bm V}-\bm\varepsilon_{\widehat{\bm V}}}_{\mathcal{E}_h,\tau}^2
	\,+\,\norm{\bm g^{D}-\bm g^{D}_h}_{\Gamma_h^{D},l^{-1}}^2\right)^{1/2}\,.
\end{equation}    
The following inequality holds
\begin{align}
	\label{eq:key_ineq_V}
    \vertiii{(\bm\varepsilon_{\bm\sigma},\bm\varepsilon_{\bm V}-\bm\varepsilon_{\widehat{\bm V}},\bm g^D - \bm g_h^D)}\,
        \lesssim\,& \norm{\boldsymbol{I}_{\bm\sigma}}_{D_h}
        \,+\,R^{1/2}\,\norm{\boldsymbol{I}_{\bm\sigma}\bm n_h}_{\Gamma_h^{D},h^{\perp}}
        \,+\,\left\|\left(h^{-1/2}\tfrac{\epsilon}{2}+\tau^{-1/2}\right)(G-G_h)\bm n\circ\bm\phi\right\|_{\Gamma_h^N} \\
        &+\,R\,\norm{\partial_{\bm n_h}(\boldsymbol{I}_{\bm \sigma}\bm n_h)}_{D_h^c,(h^{\perp})^2}
        \,+\,\|\tfrac{1}{2\epsilon}C_{tr}\boldsymbol{\varepsilon}_{\bm V} \|_{D_h}\,,\nonumber
\end{align}
with $\epsilon>0$ a parameter at our disposal and $C_{tr}$ a positive constant independent of $h$.
\end{lemma}
\begin{proof}

Applying the Cauchy-Schwarz inequality followed by Young's inequality in \eqref{eq:identity_V_1} we get
		\begin{align}
	\label{aux:key_ineq_V_1}
	\begin{array}{c}
	    \displaystyle\frac{1}{2}\,\norm{\bm\varepsilon_{\bm\sigma}}_{D_h}^2
		+\norm{\tau^{1/2}\,(\bm\varepsilon_{\bm V}-\bm\varepsilon_{\widehat{\bm V}})}_{\mathcal{E}_h}^2
		+\langle\bm g^D - \bm g_h^D,\bm\varepsilon_{\widehat{\bm\sigma}}\,\bm n_h\rangle_{\Gamma_h^{D}}
        +\langle(G-G_h)\bm n\circ\bm\phi,\bm\varepsilon_{\widehat{\bm V}}\rangle_{\Gamma_h^{\bm N}}
		\leq \frac{1}{2}\,\norm{\boldsymbol{I}_{\bm\sigma}}_{D_h}^2\,.
	\end{array}
	\end{align}
    On the other hand, note that combining  \eqref{lemma:energy_arg_V_f} and \eqref{eq:identity_projerr_Esigma}, the following equation holds,
	\begin{align*}
		\langle\bm\varepsilon_{\widehat{\bm\sigma}}\,\bm n_h,\bm g^D - \bm g_h^D\rangle_{\Gamma_h^{D}}
		\,=\,& \norm{l^{-1/2}\,(\bm g^D - \bm g_h^D)}_{\Gamma_h^{D}}^{2}
		\,-\,\langle\Lambda^{\boldsymbol{I}_{\bm\sigma}},\bm g^D - \bm g_h^D\rangle_{\Gamma_h^{D}}
		\,-\,\langle\Lambda^{\bm\varepsilon_{\bm\sigma}},\bm g^D - \bm g_h^D\rangle_{\Gamma_h^{D}}\\
		&
		\,-\,\langle\boldsymbol{I}_{\bm\sigma}\,\bm n_h,\bm g^D - \bm g_h^D\rangle_{\Gamma_h^{D}}
		\,+\,\langle\tau\,(\bm\varepsilon_{\bm V}-\bm\varepsilon_{\widehat{\bm V}}),\bm g^D - \bm g_h^D\rangle_{\Gamma_h^{D}}\,.
	\end{align*}
Then, substituting the expression obtained above in \eqref{aux:key_ineq_V_1} and performing some algebraic manipulations, we obtain
    {\small \begin{align}
    \nonumber
	    \frac{1}{2}\,\norm{\bm\varepsilon_{\bm\sigma}}_{D_h}^2
		\,+&\,\norm{\tau^{1/2}\,(\bm\varepsilon_{\bm V}-\bm\varepsilon_{\widehat{\bm V}})}_{\mathcal{E}_h}^2
		\,+\,\norm{ l^{-1/2}\,(\bm g^D - \bm g_h^D)}_{\Gamma_h^{D}}^2\\[1ex]
        \nonumber
        \leq\,& 
		\frac{1}{2}\,\norm{\boldsymbol{I}_{\bm\sigma}}_{D_h}^2
		\,+\,\langle\Lambda^{\boldsymbol{I}_{\bm\sigma}},\bm g^D - \bm g_h^D\rangle_{\Gamma_h^{D}}
		\,+\,\langle\Lambda^{\bm\varepsilon_{\bm\sigma}},\bm g^D - \bm g_h^D\rangle_{\Gamma_h^{D}}
		\,+\,\langle\boldsymbol{I}_{\bm\sigma}\,\bm n_h,\bm g^D - \bm g_h^D\rangle_{\Gamma_h^{D}}\\[1ex]
        \nonumber
        &-\,\langle\tau\,(\bm\varepsilon_{\bm V}-\bm\varepsilon_{\widehat{\bm V}}),\bm g^D - \bm g_h^D\rangle_{\Gamma_h^{D}}
		\,-\,\langle(G-G_h)\bm n\circ\bm \phi,\bm\varepsilon_{\widehat{\bm V}}\rangle_{\Gamma_h^N}\\[1ex]
		\nonumber
        =& \,
		\frac{1}{2}\,\norm{\boldsymbol{I}_{\bm\sigma}}_{D_h}^2 \\[1ex]
        \nonumber
		&\,+\,\langle\Lambda^{\boldsymbol{I}_{\bm\sigma}},\bm g^D - \bm g_h^D\rangle_{\Gamma_h^{D}}
		\,+\,\langle\Lambda^{\bm\varepsilon_{\bm\sigma}},\bm g^D - \bm g_h^D\rangle_{\Gamma_h^{D}}
		\,+\,\langle\boldsymbol{I}_{\bm\sigma}\,\bm n_h,\bm g^D - \bm g_h^D\rangle_{\Gamma_h^{D}} -\,\langle\tau\,(\bm\varepsilon_{\bm V}-\bm\varepsilon_{\widehat{\bm V}}),\bm g^D - \bm g_h^D\rangle_{\Gamma_h^{D}}\\[1ex]
        \label{eq:ineq0}
		& \,+\,\langle\tau^{1/2}(\bm\varepsilon_{\bm V}-\bm\varepsilon_{\widehat{\bm V}}),\tau^{-1/2}\,(G-G_h)\bm n\circ\bm\phi\rangle_{\Gamma_h^{N}} -\,\langle\bm\varepsilon_{\bm V},(G-G_h)\bm n\circ\bm\phi\rangle_{\Gamma_h^{N}}\,.
	\end{align} }
Focusing on the expression \eqref{eq:ineq0}, we note that applying successively the Cauchy-Schwarz inequality and Young's $ab \leq \tfrac{\epsilon}{2}a^2 + \tfrac{1}{2\epsilon}b^2$ inequality (with constants $\epsilon_i$ for $i=1,2,3,4$) to each of the terms on the the second--to last line above, it follows that 
{\small \begin{align}
\nonumber
\langle\Lambda^{\boldsymbol{I}_{\bm\sigma}},\bm g^D &\,-\bm g_h^D\rangle_{\Gamma_h^{D}}
	+\langle\Lambda^{\bm\varepsilon_{\bm\sigma}},\bm g^D - \bm g_h^D\rangle_{\Gamma_h^{D}}
		+\langle\boldsymbol{I}_{\bm\sigma}\bm n_h,\bm g^D - \bm g_h^D\rangle_{\Gamma_h^{D}}
		-\langle\tau(\bm\varepsilon_{\bm V}-\bm\varepsilon_{\widehat{\bm V}}),\bm g^D - \bm g_h^D\rangle_{\Gamma_h^{D}} \\[1ex]
        \nonumber
        \leq\, & \frac{1}{2\epsilon_1}\norm{l^{1/2}\Lambda^{\boldsymbol{I}_{\bm\sigma}}}_{\Gamma_h^D}^2
    +\frac{1}{2\epsilon_2}\norm{l^{1/2}\Lambda^{\bm\varepsilon_{\bm\sigma}}}_{\Gamma_h^D}^2
    +\frac{1}{2\epsilon_3}\norm{l^{1/2}\boldsymbol{I}_{\bm\sigma}\bm n_h}_{\Gamma_h^{D}}^2 +\frac{1}{2\epsilon_4}\norm{l^{1/2}\tau(\boldsymbol{\varepsilon}_{\bm V}-\boldsymbol{\varepsilon}_{\widehat{\bm V}})}_{\Gamma_h^{D}}^2 \\[1ex]
    \nonumber
    \quad&+ \frac{1}{2}(\epsilon_1+\epsilon_2+\epsilon_3+\epsilon_4)\norm{ l^{-1/2}(\bm g^D - \bm g_h^D)}_{\Gamma_h^{D}}^2\\[1ex]
    \label{eq:ineq1}
    \quad
        =\, & 3(\norm{l^{1/2}\Lambda^{\boldsymbol{I}_{\bm\sigma}}}_{\Gamma_h^D}^2
    +\norm{l^{1/2}\Lambda^{\bm\varepsilon_{\bm\sigma}}}_{\Gamma_h^D}^2
    +\norm{l^{1/2}\boldsymbol{I}_{\bm\sigma}\bm n_h}_{\Gamma_h^{D}}^2) +\frac{1}{2}\norm{l^{1/2}\tau(\boldsymbol{\varepsilon}_{\bm V}-\boldsymbol{\varepsilon}_{\widehat{\bm V}})}_{\Gamma_h^{D}}^2 + \frac{3}{4}\norm{ l^{-1/2}(\bm g^D - \bm g_h^D)}_{\Gamma_h^{D}}^2\,,
\end{align} }
where we have chosen $\epsilon_1=\epsilon_2=\epsilon_3 = 1/6$ and $\epsilon_4 = 1$. By a similar argument we have for the last line in \eqref{eq:ineq0},
{\small \begin{alignat}{6}
\nonumber
\langle\tau^{1/2}(&\,\bm\varepsilon_{\bm V}-\bm\varepsilon_{\widehat{\bm V}}),\tau^{-1/2}(G-G_h)\bm n\circ\bm\phi\rangle_{\Gamma_h^{N}}-\langle\bm\varepsilon_{\bm V},(G-G_h)\bm n\circ\bm\phi\rangle_{\Gamma_h^{N}} &&\\[1ex]
\nonumber
\leq\,&\left\|\left(h^{-1/2}\tfrac{\epsilon}{2}+\tau^{-1/2}\right)(G-G_h)\bm n\circ\bm\phi\right\|_{\Gamma_h^N}^2 
+ \tfrac{1}{4}\| \tau^{1/2}(\boldsymbol{\varepsilon}_{\bm V}-\boldsymbol{\varepsilon}_{\widehat{\bm V}}) \|_{\Gamma_h^N}^2 
+ \left\|\tfrac{1}{2\epsilon} h^{1/2} \boldsymbol{\varepsilon}_{\bm V} \right\|_{\Gamma_h^N}^2 &&\quad \text{\footnotesize (Cauchy-Schwarz \& Young)}\\[1ex]
\label{eq:ineq2}
\leq\,& \left\|\left(h^{-1/2}\tfrac{\epsilon}{2}+\tau^{-1/2}\right)(G-G_h)\bm n\circ\bm\phi\right\|_{\Gamma_h^N}^2 
+ \tfrac{1}{4}\| \tau^{1/2}(\boldsymbol{\varepsilon}_{\bm V}-\boldsymbol{\varepsilon}_{\widehat{\bm V}}) \|_{\Gamma_h^N}^2 + \left\|\tfrac{1}{2\epsilon}C_{tr}\boldsymbol{\varepsilon}_{\bm V} \right\|_{D_h}^2 &&\quad \text{\footnotesize (discrete trace inequality)\,,}
\end{alignat}
with $\epsilon>0$ a parameter at our disposal.}
Substituting the estimates \eqref{eq:ineq1} and \eqref{eq:ineq2} into \eqref{eq:ineq0} it follows that
\begin{align*}
    \frac{1}{2}\,\norm{\bm\varepsilon_{\bm\sigma}}_{D_h}^2
    \,+&\,\frac{1}{4}\,\norm{\tau^{1/2}\,(\bm\varepsilon_{\bm V}-\bm\varepsilon_{\widehat{\bm V}})}_{\mathcal{E}_h}^2
    \,+\,\frac{1}{4}\,\norm{ l^{-1/2}\,(\bm g^D - \bm g_h^D)}_{\Gamma_h^{D}}\\
    \leq &\,
    \frac{1}{2}\,\norm{\boldsymbol{I}_{\bm\sigma}}_{D_h}^2
    \,+\,3\,\norm{l^{1/2}\,\Lambda^{\boldsymbol{I}_{\bm\sigma}}}_{\Gamma_h^D}^2
    \,+\,3\,\norm{l^{1/2}\,\Lambda^{\bm\varepsilon_{\bm\sigma}}}_{\Gamma_h^D}^2
    \,+\,3\,\norm{l^{1/2}\,\boldsymbol{I}_{\bm\sigma}\bm n_h}_{\Gamma_h^{D}}^2\\
    &\,+\, \left\|\left(h^{-1/2}\tfrac{\epsilon}{2}+\tau^{-1/2}\right)(G-G_h)\bm n\circ\bm\phi\right\|_{\Gamma_h^N}^2 
    +\,\left\|\tfrac{1}{2\epsilon}C_{tr}\boldsymbol{\varepsilon}_{\bm V} \right\|_{D_h}^2\,.
\end{align*} 
Then applying \eqref{est:Lambda1}, \eqref{est:Lambda2}, and taking into account that $l(\boldsymbol{x})\leq H_e^{\perp} = r_e\,h_e^{\perp}$, we find
\begin{align*}
    \frac{1}{2}\,\norm{\bm\varepsilon_{\bm\sigma}}_{D_h}^2
    \,+&\,\frac{1}{4}\,\norm{\tau^{1/2}\,(\bm\varepsilon_{\bm V}-\bm\varepsilon_{\widehat{\bm V}})}_{\mathcal{E}_h}^2
    \,+\,\frac{1}{4}\,\norm{ l^{-1/2}\,(\bm g^D - \bm g_h^D)}_{\Gamma_h^{D}}\\
    \leq &\,
    \frac{1}{2}\,\norm{\boldsymbol{I}_{\bm\sigma}}_{D_h}^2
    \,+\,R^2\,\norm{\partial_{\boldsymbol{n_h}}(\boldsymbol{I}_{\boldsymbol{\sigma}}\boldsymbol{n}_h)}_{D_h^c, (h^{\perp})^2}^2
    \,+\,R^3(C^{ext}_e C^{inv}_e)^2\norm{\boldsymbol{\varepsilon}_{\boldsymbol{\sigma}}}_{D_h}^2
    \,+\,3\,\norm{r_e^{1/2}(h_e^{\perp})^{1/2}\boldsymbol{I}_{\bm\sigma}\bm n_h}_{\Gamma_h^{D}}^2\\
    &\,+\,\left\|\left(h^{-1/2}\tfrac{\epsilon}{2}+\tau^{-1/2}\right)(G-G_h)\bm n\circ\bm\phi\right\|_{\Gamma_h^N}^2 
    +\,\left\|\tfrac{1}{2\epsilon}C_{tr}\boldsymbol{\varepsilon}_{\bm V} \right\|_{D_h}^2\\
    \leq &\,
    \frac{1}{2}\,\norm{\boldsymbol{I}_{\bm\sigma}}_{D_h}^2
    \,+\,R^2\,\norm{\partial_{\boldsymbol{n_h}}(\boldsymbol{I}_{\boldsymbol{\sigma}}\boldsymbol{n}_h)}_{D_h^c, (h^{\perp})^2}^2
    \,+\,R^3(C^{ext}_e C^{inv}_e)^2\norm{\boldsymbol{\varepsilon}_{\boldsymbol{\sigma}}}_{D_h}^2
    \,+\,3\,R\,\norm{\boldsymbol{I}_{\bm\sigma}\bm n_h}_{\Gamma_h^{D}, h^{\perp}}^2\\
    &\,+\,\left\|\left(h^{-1/2}\tfrac{\epsilon}{2}+\tau^{-1/2}\right)(G-G_h)\bm n\circ\bm\phi\right\|_{\Gamma_h^N}^2 
    +\,\left\|\tfrac{1}{2\epsilon}C_{tr}\boldsymbol{\varepsilon}_{\bm V} \right\|_{D_h}^2\,.
\end{align*} 
Finally, by \eqref{eq:Rbound}, we have that $R^3(C^{ext}_e C^{inv}_e)^2<1/2$, and rearranging terms, \eqref{eq:key_ineq_V} holds.
\qed\end{proof}

To derive the estimates for $\norm{\bm \varepsilon_{\bm\sigma}}_{D_h}$ we need to control the term $\norm{(G-G_h)\bm n\circ\bm\phi}_{\Gamma_h^N}$. We achieve that in the following lemma.

\begin{lemma}\label{lemma:rate_normGs}
Let $\boldsymbol{p}$ and $\boldsymbol{r}$ be the solutions of the state problem \eqref{HDG_mixed_formualtion_state} and adjoint problem \eqref{HDG_mixed_formualtion_adjoint}, respectively. There exists a positive constant $C_G$ such that
    \begin{equation}\label{ineq:GGh}
    \begin{array}{l}
        \displaystyle\norm{(G(\Gamma)- G_h(\Gamma))\bm n\circ\bm\phi}_{\Gamma_h^{N}}
        \leq C_G \big(
        \norm{(\bm r-\bm r_h)\cdot\bm n}_{\Gamma_{N}} \norm{(\bm p-\bm p_h)\cdot\bm n}_{\Gamma_{N}}\\[2ex]
        \displaystyle\qquad
        + \norm{\bm r}_{\bm H^{1}(\Omega)}\,\norm{(\bm p-\bm p_h)\cdot\bm n}_{\Gamma_{N}}
        + \norm{\bm p}_{\bm H^1(\Omega)}\,\norm{(\bm r -\bm r_h)\cdot\bm n}_{\Gamma_{N}}
        + \norm{(\bm r -\bm r_h)\cdot\bm n}_{\Gamma_{N}} \big)\,,
    \end{array}
    \end{equation}
    Furthermore, for $\boldsymbol{q} \in\{\boldsymbol{p}, \boldsymbol{r}\}$, there holds
    \begin{align}\label{eq:rate_normGs-2}
        & \norm{(\bm q-\bm q_h)\cdot\bm n}_{\Gamma_N}
        \,\lesssim \,
        h^{m + 1/2} (\norm{\boldsymbol{q}}_{\boldsymbol{H}^{m+1}(\Omega)} + |y|_{H^{k+1}(\Omega)})
        + h^{-1/2} \norm{\boldsymbol{I}_{\boldsymbol{q}}}_{D_h} 
        + R h^{-1/2} \norm{\boldsymbol{\varepsilon}_{\boldsymbol{q}}}_{D_h}\,.
    \end{align}
\end{lemma}
\begin{proof}
    Let us note the following
    \begin{align*}
        G(\Gamma)\,-\,G_h(\Gamma)
        &\,=\, a^{-1}\,(\bm p\cdot\bm n\,\bm r\cdot\bm n-\bm p_h\cdot\bm n\,\bm r_h\cdot \bm n)
        \,+\,\partial_{\bm n}g\,(\bm r-\bm r_h)\cdot\bm n\,,
    \end{align*}
    hence, we have that
    \begin{align*}
        \norm{(G-G_h)\bm n\circ\bm\phi}_{\Gamma_h^N}
        \lesssim\,& \norm{(G  - G_h) \bm n}_{\Gamma_N}
        \lesssim \norm{G_h -G}_{\Gamma_N}\\
        \,\lesssim&\, \norm{\bm p\cdot\bm n\,\bm r\cdot\bm n-\bm p_h\cdot\bm n\,\bm r_h\cdot\bm n}_{\Gamma_N}
        \,+\,\norm{(\bm r-\bm r_h)\cdot\bm n}_{\Gamma_N}\\
        \,\lesssim & \, \norm{\bm p\cdot \bm n\,(\bm r-\bm r_h)\cdot\bm n}_{\Gamma_N}
        \,+\,\norm{\bm r_h\cdot\bm n\,(\bm p-\bm p_h)\cdot\bm n}_{\Gamma_N}
                \,+\,\norm{(\bm r-\bm r_h)\cdot\bm n}_{\Gamma_N}\\
         \,\lesssim&\, \norm{\bm p\cdot\bm n}_{\Gamma_N}\,\norm{(\bm r-\bm r_h)\cdot\bm n}_{\Gamma_N}
        \,\,+\norm{(\bm r-\bm r_h)\cdot\bm n}_{\Gamma_N}\,\norm{(\bm p\ -\bm p_h)\cdot\bm n}_{\Gamma_N}\\
        \,&+\,\norm{\bm r\cdot\bm n}_{\Gamma_N}\,\norm{(\bm p-\bm p_h)\cdot\bm n}_{\Gamma_N}
        \,+\,\norm{(\bm r-\bm r_h)\cdot\bm n}_{\Gamma_N}\,.
    \end{align*}
    Then, \eqref{ineq:GGh} follows by the continuous trace inequality. On the other hand, we will prove the bound of the statement only for $\norm{(\bm p-\bm p_h)\cdot\bm n}_{\Gamma_N}$, since the proof for the bound of $\norm{(\bm r-\bm r_h)\cdot\bm n}_{\Gamma_N}$ is analogous. 

Let $\Gamma_e\subset\Gamma_N$. By adding and subtracting $\boldsymbol{E}(\Pi_{\bm Z} \bm p)$ we have that    
    \begin{align*}
       h_e^{1/2} \norm{(\bm p-\bm p_h)\cdot\bm n}_{\Gamma_e}
        &\leq 
         h_e^{1/2} \norm{\bm p-\boldsymbol{E}(\Pi_{\bm Z} \bm p) }_{\Gamma_e}
        +
        h_e^{1/2}\norm{\boldsymbol{E}(\Pi_{\bm Z} \bm p)-\bm p_h}_{\Gamma_e}\,.
    \end{align*}
 We bound the first term of the right hand side of the above equation by \eqref{ineq:discrete-trace-h-2}, and by \eqref{ineq:discrete-trace-h}, we bound the second term of the above equation, thus obtaining
\[
       h_e^{1/2}\, \norm{(\bm p-\bm p_h)\cdot\bm n}_{\Gamma_e}
       \lesssim 
        \norm{{\bm I}_{\bm p}}_{K^{ext}_e \cup K_e}
        + h_e \norm{\nabla{\bm I}_{\bm p}}_{K^{ext}_e \cup K_e}
        + \norm{\bm\varepsilon_{\bm p}}_{K^{ext}_e}\,,
\]   
In turn we note that using the definition of $C^{ext}_e$ (see \cite[Lemma A.1]{CoQiuSo2014}), we have that $\norm{\boldsymbol{\varepsilon}_{\boldsymbol{p}}}_{K^{ext}_e}\lesssim r_e^{1/2}\norm{\boldsymbol{\varepsilon}_{\boldsymbol{p}}}_{K_e}$, thus we deduce
\begin{alignat*}{2}
    &h_e^{1/2} \norm{(\bm p-\bm p_h)\cdot\bm n}_{\Gamma_e}
    \lesssim 
    \norm{{\bm I}_{\bm p}}_{K^{ext}_e \cup K_e}
    + h_e \norm{\nabla{\bm I}_{\bm p}}_{K^{ext}_e \cup K_e}
    + r_e^{1/2} \norm{\bm\varepsilon_{\bm p}}_{K_e}\\
    &\quad \lesssim \norm{{\bm I}_{\bm p}}_{K^{ext}_e \cup K_e}
    + h_e^{m+1}|\boldsymbol{E}(\boldsymbol{p})|_{\boldsymbol{H}^{m+1}(B_e)}
    + r_e^{1/2} h_e^{m+1}\norm{\boldsymbol{E}(\boldsymbol{p})}_{\boldsymbol{H}^{m+1}(K_e)}\\
    &\quad\quad 
    + r_e^{1/2} \norm{\boldsymbol{I}_{\boldsymbol{p}}}_{K_e}
    + h_e \norm{\nabla \boldsymbol{I}_{\boldsymbol{p}}}_{K_e}
    + r_e^{1/2}\,\norm{\boldsymbol{\varepsilon}_{\boldsymbol{p}}}_{K_e}.
    \tag*{ \text{\footnotesize (by \eqref{ineq:nablaIpext})}}\\
    &\quad
    \lesssim (1 + r_e^{1/2}) h^{m+1} \norm{\boldsymbol{E}(\boldsymbol{p})}_{\boldsymbol{H}^{m+1}(B_e)}
    + (1 + r_e^{1/2}) \norm{\boldsymbol{I}_{\boldsymbol{p}}}_{K_e}
    + h_e \norm{\nabla \boldsymbol{I}_{\boldsymbol{p}}}_{K_e}
    + r_e^{1/2} \norm{\boldsymbol{\varepsilon}_{\boldsymbol{p}}}_{K_e}\,.\tag*{ \text{\footnotesize (by \eqref{ineq:Ipext})}}
\end{alignat*}
 Dividing by $h_e^{1/2}$, adding over all the elements and applying \eqref{eq:extension-op}, \eqref{eq:grad-estimate}, and bearing in mind the definition of $R$ and $h$, \eqref{eq:rate_normGs-2} holds. 
\end{proof}

We can note that the fourth term on the right hand side in \eqref{eq:key_ineq_V} is controlled by \cite[Lemma 3.8]{CoQiuSo2014}.
Then, the unique term that is not controlled is
$\norm{\bm\varepsilon_{\bm V}}_{D_h}$. In order to have an estimate for the latter we have to present the error estimates for $ \bm V\,-\,\bm V_h$, which is presented in the next section. Once we have performed this analysis, we will be in position to establish the error estimates for the deformation field equation.

\subsection{\texorpdfstring{Error estimates for $\bm V\,-\,\bm V_h$ }{Error estimates for EhV }}
In this section we develop the error estimates for $ \bm V\,-\,\bm V_h$, for this purpose we will follow the same strategy done for the estimates of $ e^{y}$ and $e^{z}$, i.e., we will use a dual problem to find the estimates. For any given $\bm U \in [L^2(\Omega)]^d$, let be $(\bm\gamma,\bm u)$ solution of
\begin{subequations}
\label{eq:dual_problem_V}
\begin{alignat}{6}
    \label{eq:dual_problem_V_a}
    \bm \gamma\, +\, \nabla\bm u &\,=\, 0 && \quad \mbox{in }D_h\,,\\
    \label{eq:dual_problem_V_b}
    \bdiv(\bm\gamma) &\,=\, \bm U && \quad \mbox{in }D_h\,,\\
    \label{eq:dual_problem_V_c}
    \bm u &\,=\, \bm 0 && \quad\mbox{on }\Gamma_h^D\,,\\
    \label{eq:dual_problem_V_d}
    \bm\gamma\,\bm n_h&\,=\, \bm 0 && \quad\mbox{on }\Gamma_h^N\,.
\end{alignat}
\end{subequations}
We assume that the solution to this dual problem satisfies the elliptic regularity
\begin{equation}
    \label{eliptic_reg^DP_V}
    \norm{\bm u}_{s + 1, D_h}\,+\,\norm{\bm\gamma}_{s,D_h}
    \,\leq\, C\,\norm{\bm U}_{D_h}\,,
\end{equation}
where $s \geq 0$, and $C>0 $ depend on the domain $D_h$. 

\begin{remark}
The elliptic regularity holds with $s=1$, for example, when the domain is a convex polyhedral or has a $\mathcal{C}^2$ boundary as in \cite{CoQiuSo2014}. This $H^2$-regularity is usually used to prove superconvergence properties of HDG schemes. However, since in our context the domain $D_h$ might be a nonconvex polyhedron, the purpose of this duality argument is not to show suerconvergence but rather to to bound $\norm{\bm\varepsilon_{\bm V}}_{D_h}$ and obtain error estimates even for the case $s= 0$.
\end{remark}

We have the following identity.

\begin{lemma}The projection of the errors and the interpolation errors satisfy the following identity
    \label{lemma:identity_EhV_1}
    \begin{equation}
        \label{eq:identity_EhV_1}
        (\bm\varepsilon_{\bm V},\bm U)_{D_h}
        \,=\, (\boldsymbol{I}_{\bm \sigma},\bm \Pi_{\mathbb{Z}}\bm \gamma)_{D_h}
        \,-\,(\bm\varepsilon_{\bm\sigma},\bm\gamma -\bm \Pi_{\mathbb{Z}}\bm\gamma)_{D_h}
        \,+\,\mathbb{T}_{\bm V,h}\,,
    \end{equation}
    where 
    \begin{equation}
        \label{def:TVh}
        \mathbb{T}_{\bm V, h}
        \,:=\,\langle \bm g^{D}-\bm g_h^D,\bm\gamma\,\bm n_h\rangle_{\Gamma_h^D}
        \,-\,\langle(G-G_h)\bm n\circ\bm\phi,\bm u\rangle_{\Gamma_h^N}\,.
    \end{equation}
\end{lemma}
\begin{proof}
    Let us note that bearing in mind 
    \eqref{eq:dual_problem_V_a}, \eqref{eq:dual_problem_V_b}, we get
    \begin{align*}
        (\bm\varepsilon_{\bm V},\bm U)_{D_h}
        \,=\,& (\bm \varepsilon_{\bm V},\bdiv(\bm\gamma))_{D_h}
        \,=\, (\bm\varepsilon_{\bm V},\bdiv(\bm\gamma))_{D_h}
        \,-\,(\bm\varepsilon_{\bm \sigma},\bm\gamma)_{D_h}
        \,-\,(\bm\varepsilon_{\bm\sigma},\nabla \bm u)_{D_h}\\
        \,=\,& (\bm\varepsilon_{\bm V},\bdiv(\bm \Pi_{\mathbb{Z}}\bm\gamma))_{D_h}
        \,+\,(\bm\varepsilon_{\bm V},\bdiv(\bm\gamma -\bm \Pi_{\mathbb{Z}}\bm\gamma))_{D_h}
        \,-\,(\bm\varepsilon_{\bm \sigma},\bm\Pi_{\mathbb{Z}}\bm\gamma)_{D_h}
        \,-\,(\bm\varepsilon_{\bm \sigma},\bm\gamma-\bm\Pi_{\mathbb{Z}}\bm\gamma)_{D_h}\\
        &
        \,-\,(\bm\varepsilon_{\bm\sigma},\nabla \bm\Pi_{\bm W}\bm u)_{D_h}
        \,-\,(\bm\varepsilon_{\bm\sigma},\nabla (\bm u-\bm\Pi_{\bm W}\bm u))_{D_h}\,.
    \end{align*}
 Setting $\bm \psi\,=\, \bm\Pi_{\mathbb{Z}}\bm\gamma$ in \eqref{lemma:energy_arg_V_a} and $\bm w \,=\, \bm\Pi_{\bm W}\bm u$ in \eqref{lemma:energy_arg_V_b}, implies that 
    \begin{align*}
        (\bm\varepsilon_{\bm V},\bm U)_{D_h}
              &\,=\, (\boldsymbol{I}_{\bm\sigma},\bm\Pi_{\mathbb{Z}}\bm\gamma)_{D_h}
        \,-\,(\bm\varepsilon_{\bm\sigma},\bm\gamma-\Pi_{\mathbb{Z}}\bm\gamma)_{D_h}
        \,+\,\mathbb{T}_{\bm V,h}\,,
    \end{align*}
    where $\mathbb{T}_{\bm V, h}$ is defined as
    \begin{equation*}
        \mathbb{T}_{\bm V, h}
        \,:=\, \langle\bm\varepsilon_{\bm V},\bm\Pi_{\mathbb{Z}}\bm\gamma\,\bm n_h\rangle_{\mathcal{E}_h}
        \,-\,\langle\bm\varepsilon_{\widehat{\bm\sigma}}\,\bm n_h,\bm\Pi_{\bm W}\bm u\rangle_{\mathcal{E}_h}
        \,+\,(\bm\varepsilon_{\bm V},\bdiv(\bm \gamma-\bm\Pi_{\mathbb{Z}}\bm\gamma))_{D_h}
        \,-\,(\bm\varepsilon_{\bm\sigma},\nabla(\bm u-\bm\Pi_{\bm W}\bm u))_{D_h}\,.
    \end{equation*}
    Integrating by parts we can deduce
    \begin{align*}
        \mathbb{T}_{\bm V, h}
        &\,=\, 
        \langle\boldsymbol{\varepsilon}_{\boldsymbol{V}}, \boldsymbol{\Pi}_{\mathbb{Z}}\boldsymbol{\gamma}\boldsymbol{n}_h\rangle_{\mathcal{E}_h}
        \,-\, \langle\boldsymbol{\varepsilon}_{\widehat{\boldsymbol{\sigma}}}\boldsymbol{n}_h, \boldsymbol{\Pi}_{\boldsymbol{W}}\boldsymbol{u}\rangle_{\mathcal{E}_h}
        \,-\, (\nabla \boldsymbol{\varepsilon}_{\boldsymbol{V}}, \boldsymbol{\gamma} - \boldsymbol{\Pi}_{\mathbb{Z}}\boldsymbol{\gamma})_{D_h}
        \,+\,\langle\boldsymbol{\varepsilon}_{\boldsymbol{V}}, \boldsymbol{\gamma} - \boldsymbol{\Pi}_{\mathbb{Z}}\boldsymbol{\gamma}\rangle_{\mathcal{E}_h}\\
        &\quad\,+\,(\bdiv(\boldsymbol{\varepsilon}_{\boldsymbol{\sigma}}), \boldsymbol{u} - \boldsymbol{\Pi}_{\boldsymbol{W}}\boldsymbol{u})_{D_h}
        \,-\,\langle\boldsymbol{\varepsilon}_{\boldsymbol{\sigma}}\boldsymbol{n}_h,\boldsymbol{u} - \boldsymbol{\Pi}_{\boldsymbol{W}}\boldsymbol{u}\rangle_{\mathcal{E}_h}\\
        &\,=\,\langle\boldsymbol{\varepsilon}_{\boldsymbol{V}}, \boldsymbol{\Pi}_{\mathbb{Z}}\boldsymbol{\gamma}\boldsymbol{n}_h\rangle_{\mathcal{E}_h}
        \,-\, \langle\boldsymbol{\varepsilon}_{\widehat{\boldsymbol{\sigma}}}\boldsymbol{n}_h, \boldsymbol{\Pi}_{\boldsymbol{W}}\boldsymbol{u}\rangle_{\mathcal{E}_h}
        \,+\,\langle\boldsymbol{\varepsilon}_{\boldsymbol{V}}, \boldsymbol{\gamma} - \boldsymbol{\Pi}_{\mathbb{Z}}\boldsymbol{\gamma}\rangle_{\mathcal{E}_h}\\
        &\quad\,-\, \langle\boldsymbol{\varepsilon}_{\boldsymbol{\sigma}}\boldsymbol{n}_h,\boldsymbol{u} - \boldsymbol{\Pi}_{\boldsymbol{W}}\boldsymbol{u}\rangle_{\mathcal{E}_h}  \tag*{ \text{\footnotesize (by \eqref{HDG_proyection2_a} and \eqref{HDG_proyection2_b})}}\\
        &\,=\,\langle\bm\varepsilon_{\widehat{\bm V}}-\bm\varepsilon_{\bm V},(\bm\Pi_{\mathbb{Z}}\bm \gamma-\bm \gamma)\,\bm
        n_h\rangle_{\mathcal{E}_h}
        -\langle(\bm\varepsilon_{\widehat{\bm \sigma}} -\bm\varepsilon_{\bm\sigma})\,\bm n_h,\bm\Pi_{\bm W}\bm u-\bm u\rangle_{\mathcal{E}_h}
        +\langle\bm\varepsilon_{\widehat{\bm V}},\bm\gamma\,\bm n_h\rangle_{\mathcal{E}_h}
        -\langle\bm\varepsilon_{\widehat{\bm\sigma}}\,\bm n_h,\bm u\rangle_{\mathcal{E}_h}.
    \end{align*}
   By \eqref{lemma:energy_arg_V_c} we have that $\langle\bm \varepsilon_{\widehat{\bm\sigma}}\,\bm n_h,\bm u\rangle_{\mathcal{E}_h\backslash\Gamma_h}\,=\,0$. Moreover, since $\bm\varepsilon_{\widehat{\bm V}}$ is single-valued in $\mathcal{E}_h$ and $\bm\gamma \in H(\bdiv;\Omega)$, it follows that
    $
        \langle\bm\varepsilon_{\widehat{\bm V}},\bm\gamma\,\bm n_h\rangle_{\mathcal{E}_h} 
        \,=\,
       \langle\bm\varepsilon_{\widehat{\bm V }},\bm\gamma\,\bm n_h\rangle_{\Gamma_h}
$.
 In turn, by \eqref{lemma:energy_arg_V_f} we obtain
    \begin{align*}
        \mathbb{T}_{\bm V,h}
        \,=\,&
        \langle \bm\varepsilon_{\widehat{\bm V}}-\bm\varepsilon_{\bm V}, (\bm \Pi_{\mathbb{Z}}\bm \gamma-\bm\gamma)\,\bm n_h+\tau\,(\bm\Pi_{\bm W}\bm u-\bm u)\rangle_{\mathcal{E}_h}
        \,+\,\langle\bm\varepsilon_{\widehat{\bm V}},\bm \gamma\,\bm n_h\rangle_{\Gamma_h}
        \,-\,\langle\bm\varepsilon_{\widehat{\bm\sigma}}\,\bm n_h,\bm u\rangle_{\Gamma_h}\,.
    \end{align*}
    We note by \eqref{HDG_proyection2_c} that $\langle \bm\varepsilon_{\widehat{\bm V}}-\bm\varepsilon_{\bm V}, (\bm \Pi_{\mathbb{Z}}\bm \gamma-\bm\gamma)\,\bm n_h+\tau\,(\bm\Pi_{\bm W}\bm u-\bm u)\rangle_{\mathcal{E}_h} \,=\, 0$. Finally by \eqref{lemma:energy_arg_V_d}, \eqref{lemma:energy_arg_V_e}, \eqref{eq:dual_problem_V_c}, and \eqref{eq:dual_problem_V_d} we get
    \begin{align*}
        \mathbb{T}_{\bm V, h} 
        &\,=\,\langle \bm g^{D}-\bm g_h^D,\bm\gamma\,\bm n_h\rangle_{\Gamma_h^D}
        \,-\,\langle(G-G_h)\bm n\circ\bm\phi,\bm u\rangle_{\Gamma_h^N}\,.
    \end{align*}
\qed\end{proof}
We can now establish one of the most important results of this section. In fact, thanks to this lemma, we will be able to deduce the convergence rate of the scheme. 

\begin{lemma} \label{lemma:bound_E_h^V}
Suppose that $\mathcal T_h$ is an admissible triangulation and assume the local proximity condition \eqref{eq:localProximity} is 
    satisfied for $n\geq 1$ and $ 0 < \delta <1$ and $ s \geq 0 $ in \eqref{eliptic_reg^DP_V}. 
    Then,
    \begin{align}\nonumber
        \norm{\bm\varepsilon_{\bm V}}_{D_h}
        &\lesssim 
        \big(H_{\bm V}\norm{\boldsymbol{I}_{\bm\sigma}}_{D_h} 
        + h^{n+\delta}H_{\bm V}\norm{\partial_{\bm n_h}(\boldsymbol{I}_{\bm \sigma} \bm n_h)}_{D_h^{c},(h^{\perp})^2}
        + h^{n/2+\delta/2}H_{\bm V}  \norm{\boldsymbol{I}_{\bm\sigma} \bm n_{h}}_{\Gamma_h^{D},h^{\perp}}
        \\
        \label{eq:bound_E_h^V} 
        & \qquad +( H_{\bm V}^2h^{-1/2}+1) \norm{(G-G_h)\bm n\circ\bm\phi}_{\Gamma_h^{N}}\big),
    \end{align}
    where $ H_{\bm V}:=h^s+h^{n/2+\delta/2-1/2}$.
\end{lemma}

\begin{proof}
    From Lemma \ref{lemma:identity_EhV_1}, we find
    \begin{align*}
    \nonumber
        \mathbb{T}_{\bm V, h}
        =& \langle \bm g^{D}-\bm g_h^D,\bm\gamma\bm n_h\rangle_{\Gamma_h^D}
        - \langle(G-G_h)\bm n\circ\bm\phi,\bm u\rangle_{\Gamma_h^N}\\
        \leq& \norm{l^{-1/2}(\bm g^D -\bm g_h^D)}_{\Gamma_h^D} \norm{l^{1/2}\boldsymbol{\gamma} \boldsymbol{n}_h}_{\Gamma_h^D}
        + \norm{ (G - G_h) \boldsymbol{n} \circ \boldsymbol{\phi}}_{\Gamma_h^N} \norm{ \boldsymbol{u}}_{\Gamma_h^N}.
    \end{align*}
Adding and subtracting $\boldsymbol{\Pi}_{\mathbb{Z}} \boldsymbol{\gamma}$,  using \eqref{eq:I_r-Gamma-estimate} and recalling  that $l({\bm x}):=|\bm x - \bar{\bm x}|\lesssim h^{n+\delta}$, we obtain
        \begin{align}
        \mathbb{T}_{\bm V, h}
        \leq& \norm{l^{-1/2}(\bm g^D -\bm g_h^D)}_{\Gamma_h^D} \norm{l^{1/2}(\boldsymbol{\Pi}_{\mathbb{Z}} \boldsymbol{\gamma}-\boldsymbol{\gamma}) }_{\Gamma_h^D}
        +\norm{l^{-1/2}(\bm g^D -\bm g_h^D)}_{\Gamma_h^D} \norm{l^{1/2}\boldsymbol{\Pi}_{\mathbb{Z}} \boldsymbol{\gamma}}_{\Gamma_h^D}\nonumber\\
        &+
         \norm{(G - G_h) \boldsymbol{n} \circ \boldsymbol{\phi}}_{\Gamma_h^N} \norm{\boldsymbol{u}}_{\Gamma_h^N}\nonumber\\
         \lesssim& h^{n/2+\delta/2+s-1/2}\norm{l^{-1/2}(\bm g^D -\bm g_h^D)}_{\Gamma_h^D} \|\boldsymbol{U}\|_{D_h}
        +h^{n/2+\delta/2-1/2}\norm{l^{-1/2}(\bm g^D -\bm g_h^D)}_{\Gamma_h^D} \|\boldsymbol{U}\|_{D_h}\nonumber \\
        &+
         \norm{(G - G_h) \boldsymbol{n} \circ \boldsymbol{\phi}}_{\Gamma_h^N} \|\boldsymbol{U}\|_{D_h}\nonumber\\
          \lesssim& h^{n/2+\delta/2-1/2}\norm{l^{-1/2}(\bm g^D -\bm g_h^D)}_{\Gamma_h^D} \|\boldsymbol{U}\|_{D_h}+
         \norm{(G - G_h) \boldsymbol{n} \circ \boldsymbol{\phi}}_{\Gamma_h^N} \|\boldsymbol{U}\|_{D_h}\label{eq:eq:bound_E_h^V-aux1}.
    \end{align}
    On the other hand, mimicking \cite[Step 4 of the proof of Lemma 3.5]{CoQiuSo2014}, we deduce
    \begin{align}\label{eq:eq:bound_E_h^V-aux2}
    (\boldsymbol{I}_{\boldsymbol{\sigma}},\boldsymbol{\Pi}_{\mathbb{Z} }\boldsymbol{\gamma})_{D_h}
        + (\boldsymbol{\varepsilon}_{\boldsymbol{\sigma}}, \boldsymbol{\Pi}_{\mathbb{Z}}\boldsymbol{\gamma} - \boldsymbol{\gamma} )_{D_h}
        \lesssim h^s (\norm{\boldsymbol{I}_{\boldsymbol{\sigma}}}_{D_h}
        + \norm{\boldsymbol{\varepsilon}_{\boldsymbol{\sigma}}}_{D_h}) \norm{\boldsymbol{U}}_{D_h}.
    \end{align}
    Combining \eqref{eq:eq:bound_E_h^V-aux1} and 
    \eqref{eq:eq:bound_E_h^V-aux2} with \eqref{eq:identity_EhV_1}, we get
    \begin{align*}
    \nonumber
        (\boldsymbol{\varepsilon}_{\boldsymbol{V}}, \boldsymbol{U})_{D_h}
        \lesssim& h^s (\norm{\boldsymbol{I}_{\boldsymbol{\sigma}}}_{D_h}
        + \norm{\boldsymbol{\varepsilon}_{\boldsymbol{\sigma}}}_{D_h}) \norm{\boldsymbol{U}}_{D_h}\\
        &+h^{n/2+\delta/2-1/2}\norm{l^{-1/2}(\bm g^D -\bm g_h^D)}_{\Gamma_h^D} \|\boldsymbol{U}\|_{D_h}+
         \norm{(G - G_h) \boldsymbol{n} \circ \boldsymbol{\phi}}_{\Gamma_h^N} \|\boldsymbol{U}\|_{D_h}.
    \end{align*}
   Then,
   \begin{align*}
(\boldsymbol{\varepsilon}_{\boldsymbol{V}}, \boldsymbol{U})_{D_h}
        \lesssim& (h^s+h^{n/2+\delta/2-1/2}) (\norm{\boldsymbol{I}_{\boldsymbol{\sigma}}}_{D_h}
        + \norm{\boldsymbol{\varepsilon}_{\boldsymbol{\sigma}}}_{D_h}+\norm{l^{-1/2}(\bm g^D -\bm g_h^D)}_{\Gamma_h^D} ) \norm{\boldsymbol{U}}_{D_h}\\
        &+
         \norm{(G - G_h) \boldsymbol{n} \circ \boldsymbol{\phi}}_{\Gamma_h^N} \|\boldsymbol{U}\|_{D_h}.
    \end{align*}
    Setting $\boldsymbol{U} = \boldsymbol{\varepsilon}_{\boldsymbol{V}}$ in 
    $D_h$, using \eqref{eq:key_ineq_V} and defining $ H_{\bm V}:=h^s+h^{n/2+\delta/2-1/2}$.
    \begin{align*}
        \norm{\boldsymbol{\varepsilon}_{\boldsymbol{V}}}_{D_h}
        \lesssim& H_{\bm V}\norm{\boldsymbol{I}_{\boldsymbol{\sigma}}}_{D_h}
        + R^{1/2} H_{\bm V}\norm{\boldsymbol{I}_{\boldsymbol{\sigma}} \boldsymbol{n}_h}_{\Gamma_h^D, h^{\perp}} +  H_{\bm V} \left\|\left(h^{-1/2}\tfrac{\epsilon}{2}+\tau^{-1/2}\right)(G-G_h)\bm n\circ\bm\phi\right\|_{\Gamma_h^N}\\
        &
        + H_{\bm V} R  \norm{\partial_{\boldsymbol{n}_h} (\boldsymbol{I}_{\boldsymbol{\sigma}} \boldsymbol{n}_h)}_{D_h^c, (h^{\perp})^2}+ H_{\bm V}\|\tfrac{1}{2\epsilon}C_{tr}\boldsymbol{\varepsilon}_{\bm V} \|_{D_h}+
         \norm{(G - G_h) \boldsymbol{n} \circ \boldsymbol{\phi}}_{\Gamma_h^N} .
    \end{align*}
    In addition, considering that $R \lesssim h^{n+\delta}$ (cf. 
    \eqref{eq:localProximity}), we deduce that there exists a constant $C>0$, independent of $h$, such that
    \begin{align*}
        \norm{\boldsymbol{\varepsilon}_{\boldsymbol{V}}}_{D_h}
         \leq& C\bigg( H_{\bm V} \norm{\boldsymbol{I}_{\boldsymbol{\sigma}}}_{D_h}
        + h^{n/2+\delta/2}H_{\bm V}\norm{\boldsymbol{I}_{\boldsymbol{\sigma}} \boldsymbol{n}_h}_{\Gamma_h^D, h^{\perp}}+  H_{\bm V} \left\|\left(h^{-1/2}\tfrac{\epsilon}{2}+\tau^{-1/2}\right)(G-G_h)\bm n\circ\bm\phi\right\|_{\Gamma_h^N} \\
        &
        + h^{n+\delta}H_{\bm V} \norm{\partial_{\boldsymbol{n}_h} (\boldsymbol{I}_{\boldsymbol{\sigma}} \boldsymbol{n}_h)}_{D_h^c, (h^{\perp})^2}+ H_{\bm V} \|\tfrac{1}{2\epsilon}C_{tr}\boldsymbol{\varepsilon}_{\bm V} \|_{D_h}+
         \norm{(G - G_h) \boldsymbol{n} \circ \boldsymbol{\phi}}_{\Gamma_h^N} \bigg).
    \end{align*}
Finally, choosing $\epsilon=CC_{tr} H_{\bm V} $, we have
\begin{align*}
        \norm{\boldsymbol{\varepsilon}_{\boldsymbol{V}}}_{D_h}
        & \lesssim H_{\bm V} \norm{\boldsymbol{I}_{\boldsymbol{\sigma}}}_{D_h}
        + h^{n/2+\delta/2}H_{\bm V}\norm{\boldsymbol{I}_{\boldsymbol{\sigma}} \boldsymbol{n}_h}_{\Gamma_h^D, h^{\perp}} +  H_{\bm V}(H_{\bm V}h^{-1/2}+1) \|(G-G_h)\bm n\circ\bm\phi\|_{\Gamma_h^N} \\
        &\quad  +  \norm{(G - G_h) \boldsymbol{n} \circ \boldsymbol{\phi}}_{\Gamma_h^N}
        + h^{n+\delta}H_{\bm V}\norm{\partial_{\boldsymbol{n}_h} (\boldsymbol{I}_{\boldsymbol{\sigma}} \boldsymbol{n}_h)}_{D_h^c, (h^{\perp})^2}\,,
    \end{align*}
and  
\eqref{eq:bound_E_h^V}  after noticing that $s\geq0$ and $n/2+\delta/2-1/2\geq0$.
\end{proof}

From the above lemma we can state the following theorem which gives us the 
convergence rates of the approximations.
\begin{theorem}
\label{thm:conv_rate_E_h^V}
Let us assume that ${\bm V}\in [H^{k+1}(\Omega)]^d$ and ${\bm \sigma}\in [H^{k+1}(\Omega)]^{d\times d}$. If the local proximity condition \eqref{eq:localProximity} is satisfied for $n\geq 1$ and $0 < \delta <1$ and $ s \geq 0 $ in \eqref{eliptic_reg^DP_V}, then
\begin{subequations}
\begin{alignat}{6}\norm{\boldsymbol{\varepsilon}_{\boldsymbol{V}}}_{D_h} 
    \lesssim&\,  h^{k + 1}H_{\bm V} + C_G h^{k+1/2}( H_{\bm V}^2h^{-1/2}+1),\label{eq:conv_rate_E_h^V-a1}\\[2ex]
\norm{\boldsymbol{\varepsilon}_{\boldsymbol{\sigma}}}_{D_h} 
    \lesssim&\, h^{k+1}  +C_G h^{k}+\norm{\boldsymbol{\varepsilon}_{\boldsymbol{V}}}_{D_h},\label{eq:conv_rate_E_h^V-a2}
\end{alignat}

and 
\begin{equation}\label{eq:conv_rate_E_h^V-b}
    \norm{\boldsymbol{V} - \boldsymbol{V}_h}_{\Omega} 
    \lesssim h^{k+1} +\norm{\boldsymbol{\varepsilon}_{\boldsymbol{V}}}_{D_h} , \qquad 
    \norm{\boldsymbol{\sigma} - \boldsymbol{\sigma}_h}_{\Omega} 
    \lesssim h^{k+1} + C_G h^k+\norm{\boldsymbol{\varepsilon}_{\boldsymbol{V}}}_{D_h}.
\end{equation}
Here, we recall that where $ H_{\bm V}:=h^s+h^{n/2+\delta/2-1/2}$.
\end{subequations}
\end{theorem}
\begin{proof}
    By \eqref{theorem:proyHDG_a}, \cite[Lemma 3.8]{CoQiuSo2014} and \eqref{eq:I_r-Gamma-estimate}, we note that

\[
H_{\bm V}\norm{\boldsymbol{I}_{\bm\sigma}}_{D_h} 
        + h^{n+\delta}H_{\bm V}\norm{\partial_{\bm n_h}(\boldsymbol{I}_{\bm \sigma} \bm n_h)}_{D_h^{c},(h^{\perp})^2}
        + h^{n/2+\delta/2}H_{\bm V}  \norm{\boldsymbol{I}_{\bm\sigma} \bm n_{h}}_{\Gamma_h^{D},h^{\perp}}\lesssim H_{\bm V} h^{k+1}. 
\]

    In turn, by Lemma \ref{lemma:rate_normGs}, we can deduce that
    \[
    ( H_{\bm V}^2h^{-1/2}+1)  \norm{(G -G_h)\boldsymbol{n} \circ \boldsymbol{\phi}}_{\Gamma_h^N} \lesssim C_G h^{k+1/2}( H_{\bm V}^2h^{-1/2}+1) .
    \]
    Therefore, combining all the above estimates, we prove \eqref{eq:conv_rate_E_h^V-a1}. On the other hand, choosing $\epsilon=1$ in \eqref{eq:key_ineq_V}, we can deduce
    \begin{align*}
        \norm{\boldsymbol{\varepsilon}_{\boldsymbol{\sigma}}}_{D_h}
        \lesssim& 
        \norm{\boldsymbol{I}_{\boldsymbol{\sigma}}}_{D_h} 
        + h^{n/2+\delta/2} \norm{\boldsymbol{I}_{\boldsymbol{\sigma}}}_{\Gamma_h^D, h^{\perp}}
        + \norm{h^{-1/2}(G -G_h)\boldsymbol{n}\circ \boldsymbol{\phi}}_{\Gamma_h^N}\\
        &+ h^{n+\delta} \norm{\partial_{\boldsymbol{n}_h}(\boldsymbol{I}_{\boldsymbol{\sigma}}\boldsymbol{n}_h)}_{D_h^c, (h^{\perp})^2} 
        + \norm{\boldsymbol{\varepsilon}_{\boldsymbol{V}}}_{D_h}.
    \end{align*}
    Then, by following steps analogous to those used to derive \eqref{eq:conv_rate_E_h^V-a1}, combining them with the estimate in Lemma \ref{lemma:rate_normGs} and bounding $\norm{\boldsymbol{\varepsilon}_{\boldsymbol{V}}}_{D_h}$ by \eqref{eq:conv_rate_E_h^V-a1}, we obtain \eqref{eq:conv_rate_E_h^V-a2}. Now, note that
    \begin{align*}
        \norm{\boldsymbol{V} - \boldsymbol{V}_h}_{\Omega} 
        \leq 
        \norm{\boldsymbol{V} - \boldsymbol{V}_h}_{D_h}
        + \norm{\boldsymbol{V} - \boldsymbol{V}_h}_{D_h^c} 
        \leq \norm{\boldsymbol{I}_{\boldsymbol{V}}}_{D_h}
        + \norm{\boldsymbol{\varepsilon}_{\boldsymbol{V}}}_{D_h}
        + \norm{\boldsymbol{V} - \boldsymbol{V}_h}_{D_h^c}\,.
    \end{align*}
    Then, applying \cite[Lemma 3.7]{CoQiuSo2014}, we get
    \begin{equation*}
        \norm{\boldsymbol{V} - \boldsymbol{V}_h}_{\Omega} 
        \leq 
        \norm{\boldsymbol{I}_{\boldsymbol{V}}}_{D_h}
        + \norm{\boldsymbol{\varepsilon}_{\boldsymbol{V}}}_{D_h}
        + h\norm{\boldsymbol{I}_{\boldsymbol{\sigma}}}_{D_h^c} 
        + h^{\delta/2+1} \norm{\boldsymbol{\varepsilon}_{\boldsymbol{
        V}}}_{D_h},
    \end{equation*}
    which implies the first estimate of \eqref{eq:conv_rate_E_h^V-b}.

    From \cite[Lemma 3.8]{CoQiuSo2014} we have that $\norm{\boldsymbol{I}_{\boldsymbol{\sigma}}}_{D_h^c} \lesssim h^{k+1}$. In turn, 
    By \eqref{theorem:proyHDG_a} and \eqref{eq:conv_rate_E_h^V-a1}, we find that the first estimate of  \eqref{eq:conv_rate_E_h^V-b} holds. On the other hand, analogue as above, we have 
    \begin{equation*}
        \norm{\boldsymbol{\sigma} - \boldsymbol{\sigma}_h}_{\Omega}
        \leq \norm{\boldsymbol{I}_{\boldsymbol{\sigma}}}_{D_h}
        + \norm{\boldsymbol{\varepsilon}_{\boldsymbol{\sigma}}}_{D_h}
        + \norm{\boldsymbol{\sigma} - \boldsymbol{\sigma}_h}_{D_h^c}\,.
    \end{equation*}
    By \cite[Lemma 3.7]{CoQiuSo2014}, we get
    \begin{equation*}
        \norm{\boldsymbol{\sigma} - \boldsymbol{\sigma}_h}_{\Omega}
        \leq 
        \norm{\boldsymbol{I}_{\boldsymbol{\sigma}}}_{D_h}
        + \norm{\boldsymbol{\varepsilon}_{\boldsymbol{\sigma}}}_{D_h}
        + \norm{\boldsymbol{I}_{\boldsymbol{\sigma}}}_{D_h^c}
        + h^{\delta/2} \norm{\boldsymbol{\varepsilon}_{\boldsymbol{\sigma}}}_{D_h}\,.
    \end{equation*}
    Finally, applying \eqref{theorem:proyHDG_a}, \eqref{eq:conv_rate_E_h^V-a2}, and \cite[Lemma 3.8]{CoQiuSo2014}, then the second estimate of \eqref{eq:conv_rate_E_h^V-b} holds.
\end{proof}

\section{Numerical experiments}\label{sec:numerical_experiments}

We consider two numerical examples in two dimensions and set the stabilization parameter equal to one in all of them. The first one consists of a manufactured solution in a fixed domain, with the aim to compute errors and order of convergence of the schemes, while the second is intended to observe how a domain evolves when minimizing the energy functional. 

\subsection{Computational grids}\label{sec:ComputationalGrids}
The domain $\mathcal{M}$ is triangulated by a family (background triangulation) $\{T_h\}_{h>0}$ that satisfies the first three properties in the definition of {\it admissible triangulation} from Section \ref{sec:Geometry}. This ensures that the computational grids $\{\mathcal{T}_h\}_{h>0}$ obtained through \eqref{eq:triangulation} also satisfy these three conditions, as we can observe in Fig. \ref{fig:optimal-shape}. Regarding the fourth condition, computations require only that the mapping \eqref{eq:Phi} is bijective at the vertices and the quadrature points of boundary edges. This---along with condition 4(a)---can be guaranteed by constructing the transfer paths according to the algorithm in \cite[Section 2.4]{CoSo2012}. Said algorithm also provides the local proximity condition \eqref{eq:localProximity} with $n=1$ and $\delta=0$. In practice, numerical experiments show that the method performs optimally by picking the computational mesh from a family of background triangulations satisfying only these conditions.

\subsection{Experiment 1: Convergence of the HDG scheme}\label{sec:convergence}
In order to verify the orders of convergence of the HDG solvers, we consider a fixed domain \[
\Omega:=\{\boldsymbol x=(x_1,x_2)\in \mathbb{R}^2: (0.05)^2\le |\boldsymbol{x}|^2<(0.2)^2\},
\]
the target function $\widetilde{y}(x_1,x_2)=-\sin(x_1)\sin(x_2)$ and $a=1$. The data $f$ and $g$, and a non-homogeneous boundary condition in \eqref{adjoint_eq_mp} are chosen such that $y(x_1,x_2)=\sin(x_1)\sin(x_2)$ is the solution to \eqref{pde_1} and $z(x_1,x_2)=\sin(x_1)\sin(x_2)$ is the solution to \eqref{adjoint_eq_mp}, leading to $G(\Gamma)=0$ for the optimal domain. Then, to test the HDG approximation to the deformation field, we consider artificial right-hand sides in \eqref{pde_V} so that $\boldsymbol{V}(x_1,x_2)=e^{|\boldsymbol{x}|^2-0.05^2}(1,1)^t$ is the solution to \eqref{pde_V}.

Numerical experiments for polynomial bases of degrees $k=1,2,3$ were performed and the convergence history of the scheme, as a function of the number of mesh elements $N$, is shown in Figure \ref{fig:StateConvergence}. As a reference, the black dashed line indicates the slope $h^{k+1}$, which coincides with the experimental order of convergence, in the $L^2$--norm, of the volume variables (plotted in blue and red lines in the figure). This behavior is better than the one predicted by Theorem \ref{thm:conv_rate_E_h^V} which, for $n=1$ and $\delta=0$ as in our case, predicts that the order of convergence for all the errors is $h^{k}$. Note that the error in the traces (plotted in yellow) measured with respect to the norm
\[
\vertiii{\cdot}_{h}:=\left(\displaystyle\sum_{K\in D_h} h_K\|\cdot\|_{\partial K}^2\right)^{1/2},
\] 
superconverges with order $h^{k+2}$ as proven for Dirichlet boundary value problems in \cite{CoQiuSo2014,CoSo2012}.

\begin{figure}
\begin{tabular}{cccc} 
& {\large Degree $k=1$} & {\large Degree $k=2$} & {\large Degree $k=3$} \\
\rotatebox[origin=l]{90}{\large \quad State variables} & \includegraphics[width=0.3\linewidth]{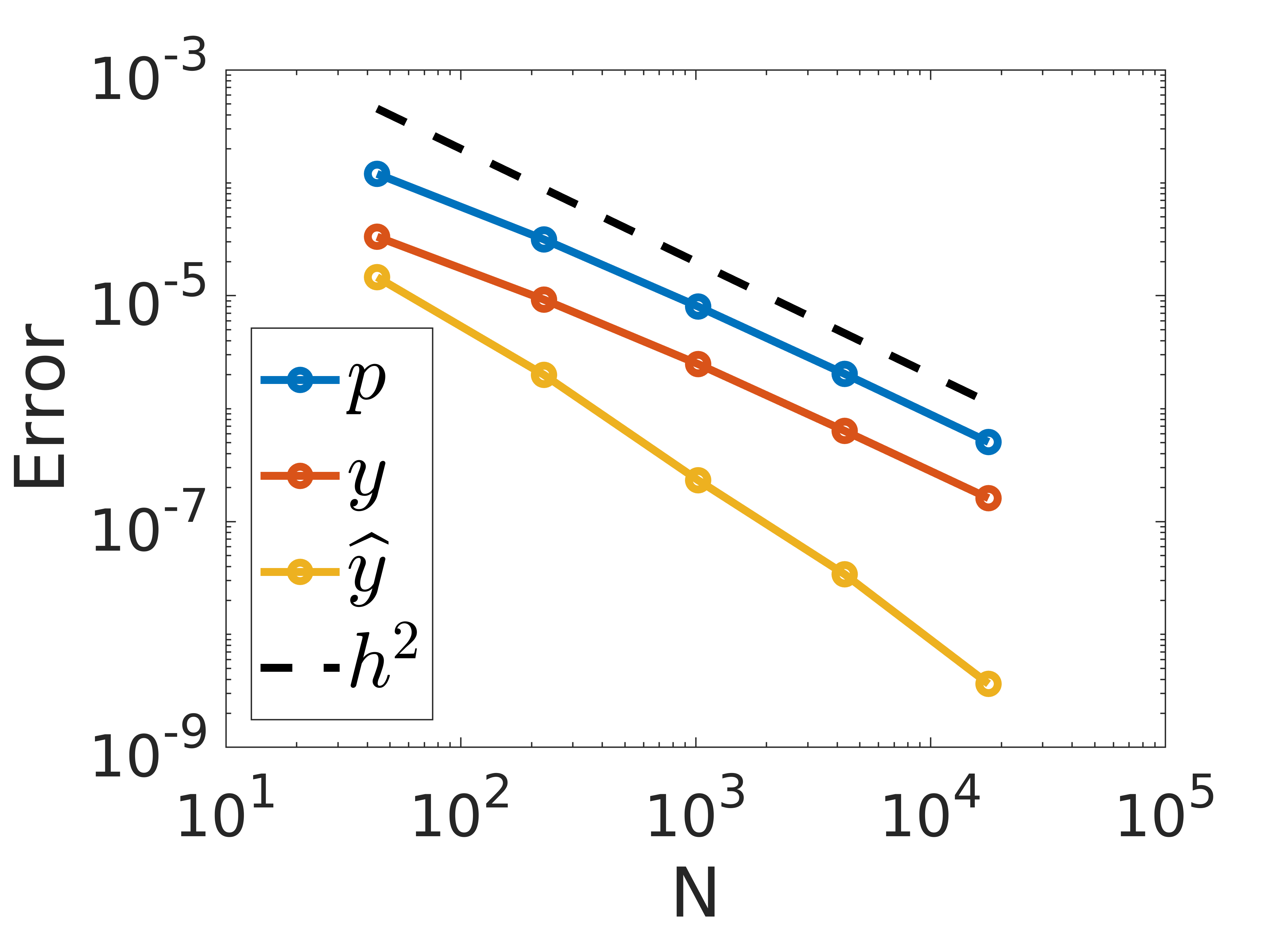} & \includegraphics[width=0.3\linewidth]{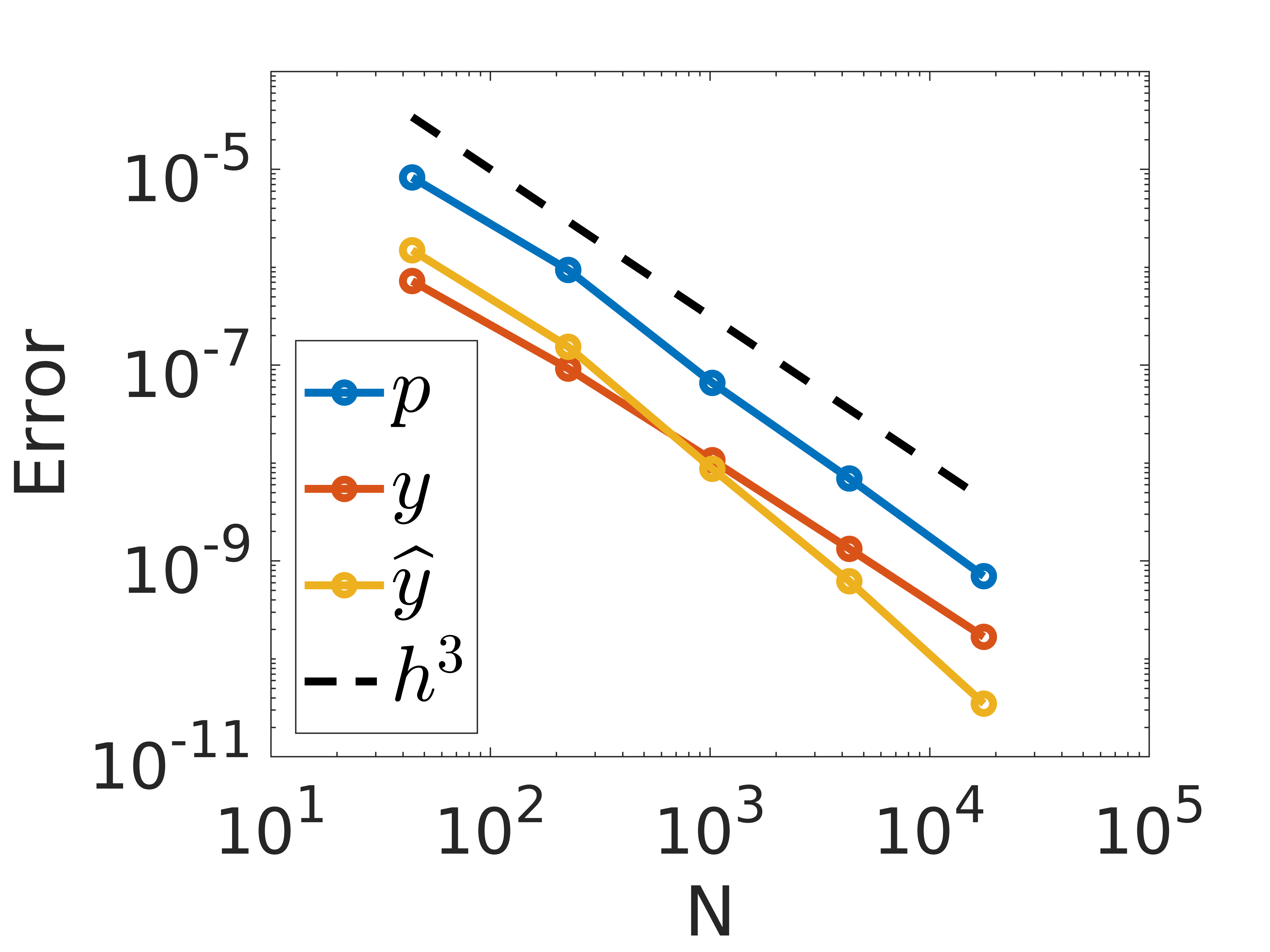} & \includegraphics[width=0.3\linewidth]{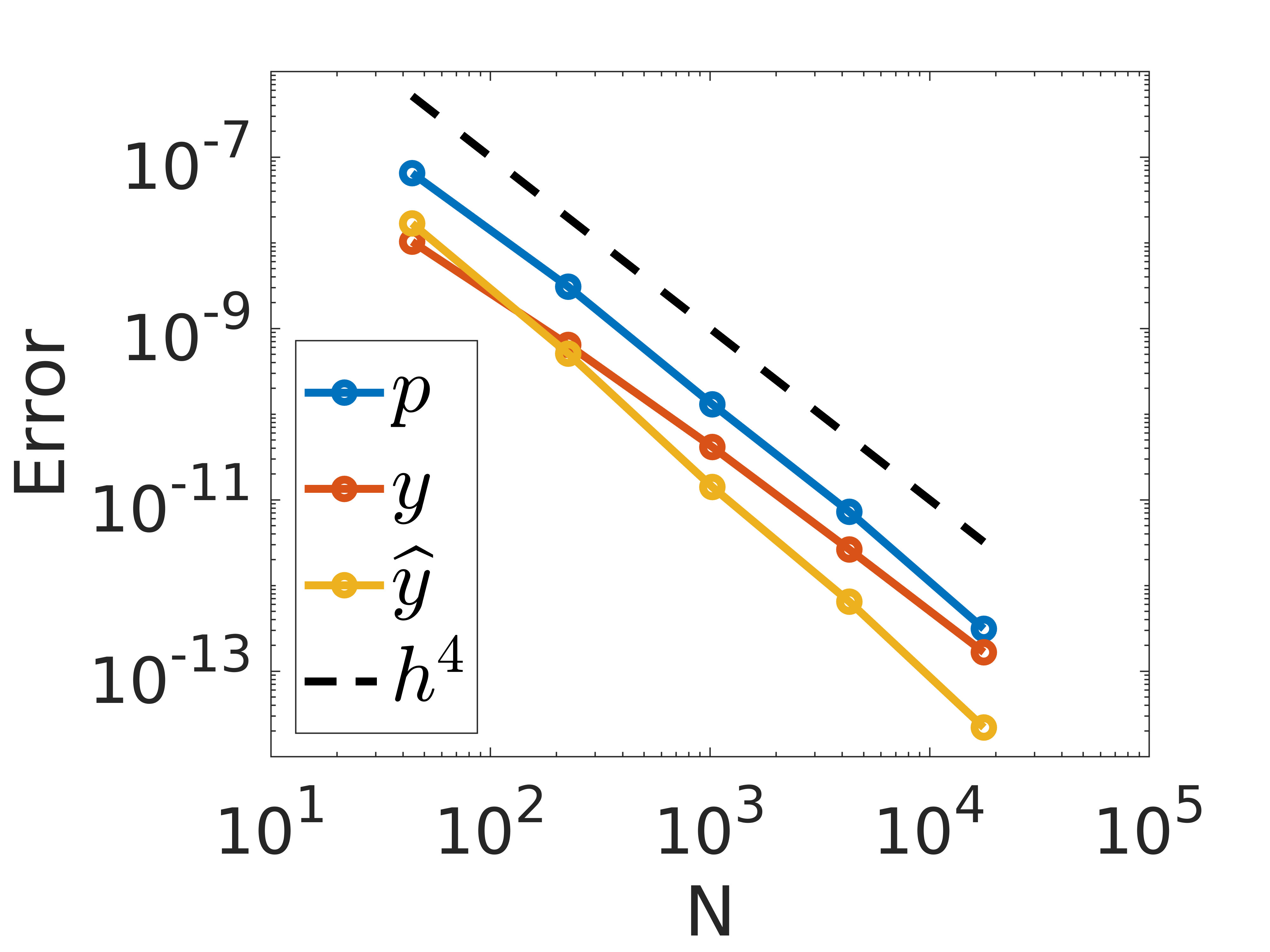} \\
\rotatebox[origin=l]{90}{\large \quad Adjoint variables} & \includegraphics[width=0.3\linewidth]{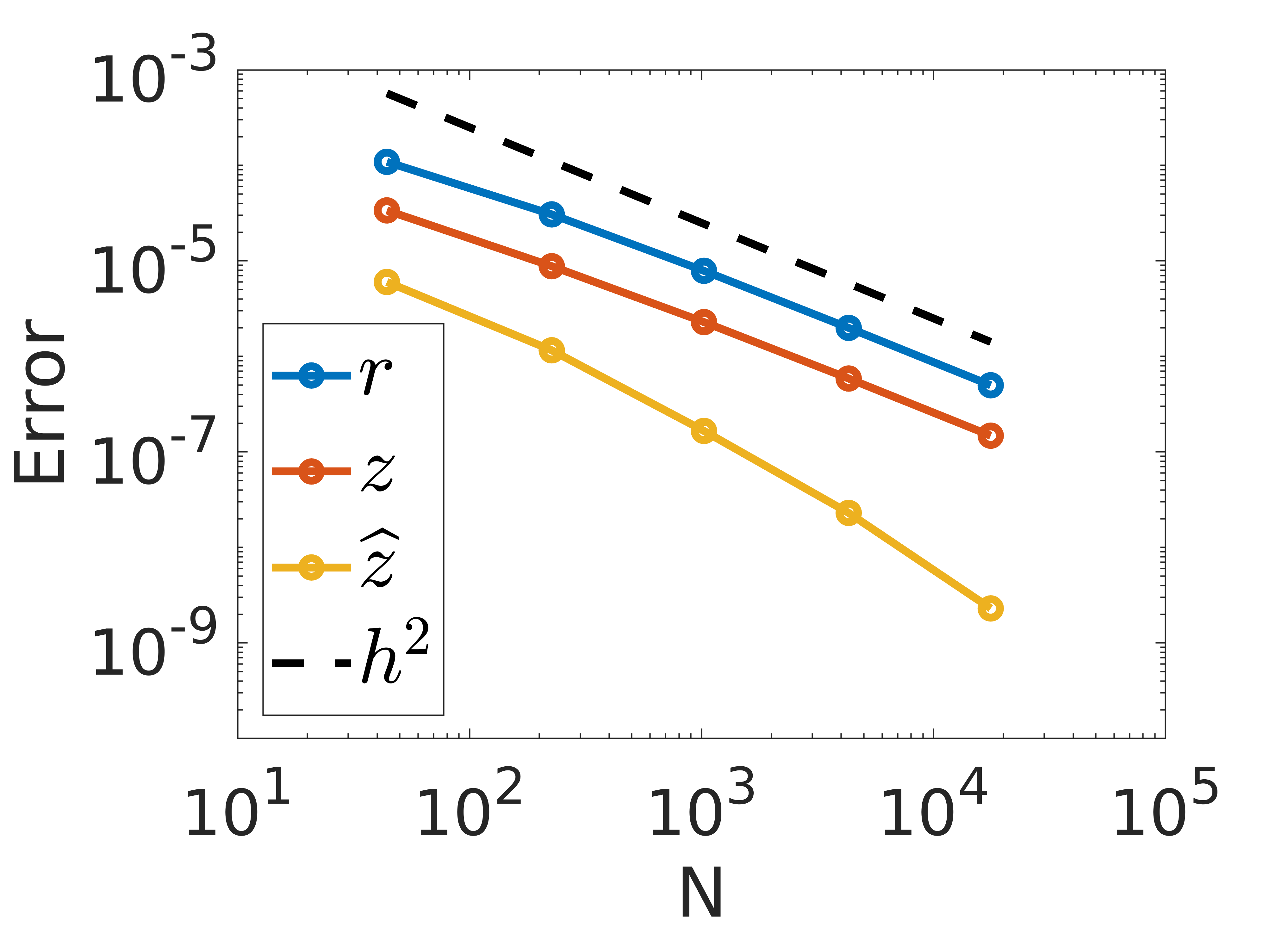} & \includegraphics[width=0.3\linewidth]{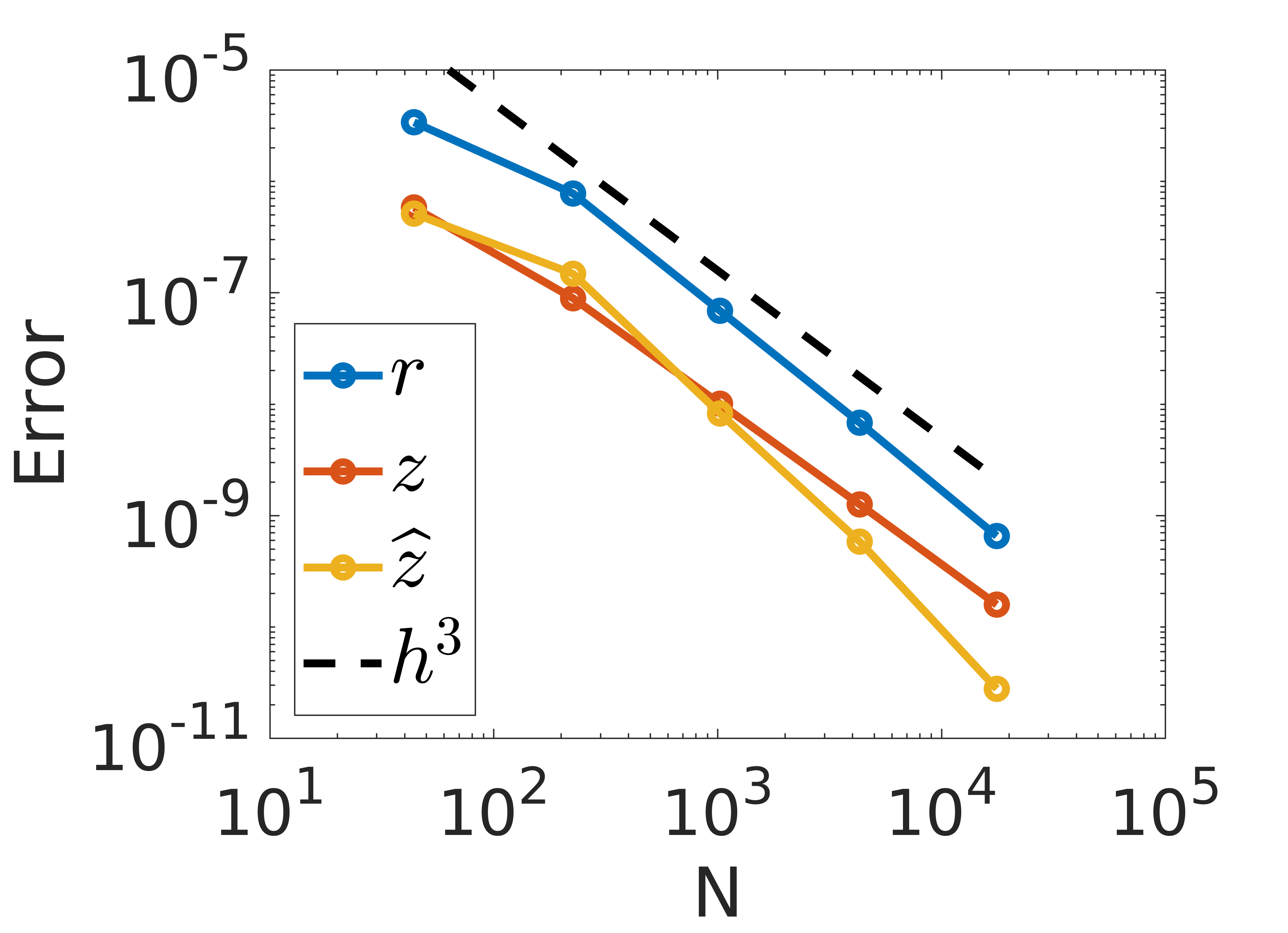} & \includegraphics[width=0.3\linewidth]{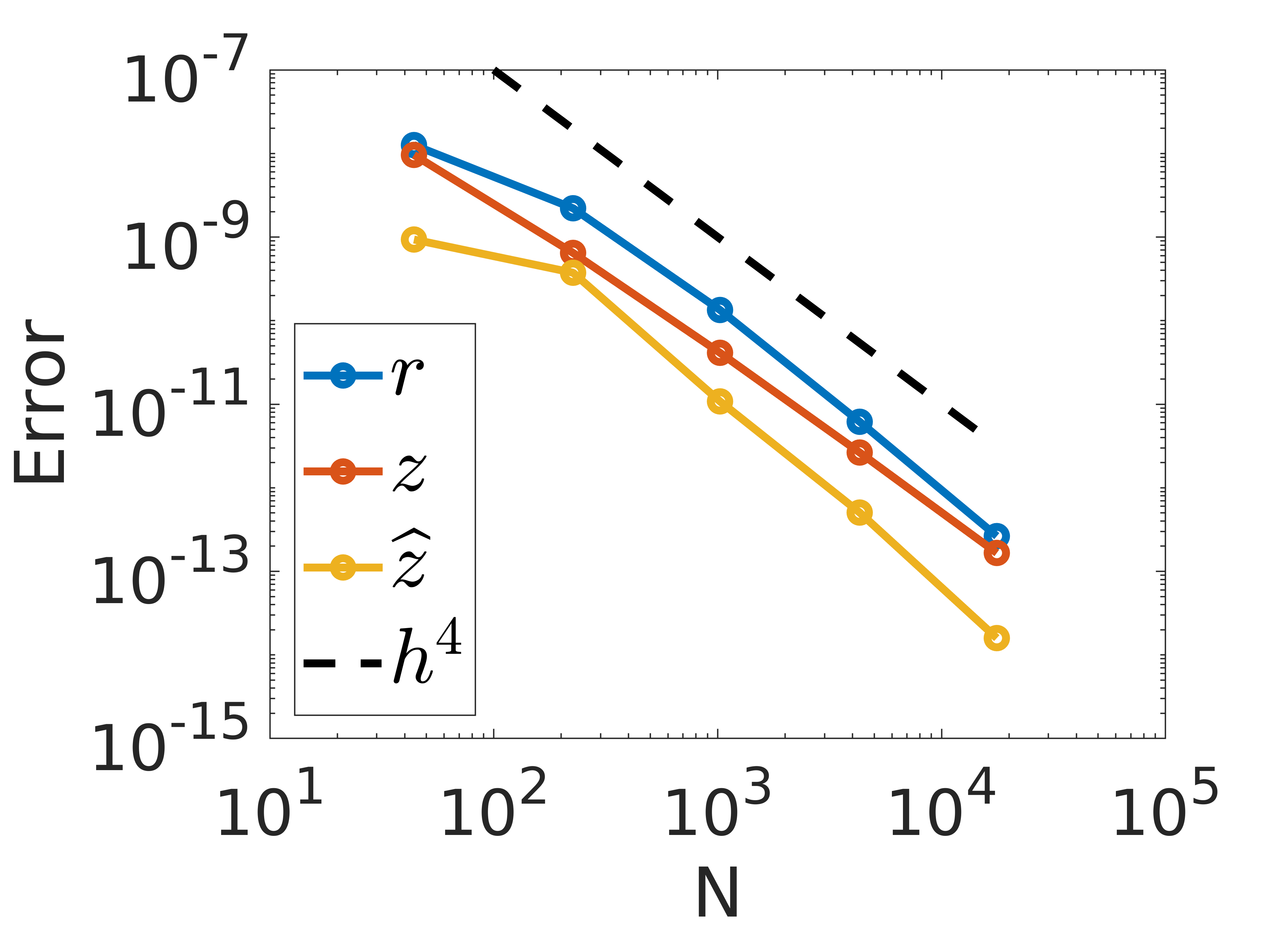} \\
\rotatebox[origin=l]{90}{\large \quad Deformation field} & \includegraphics[width=0.3\linewidth]{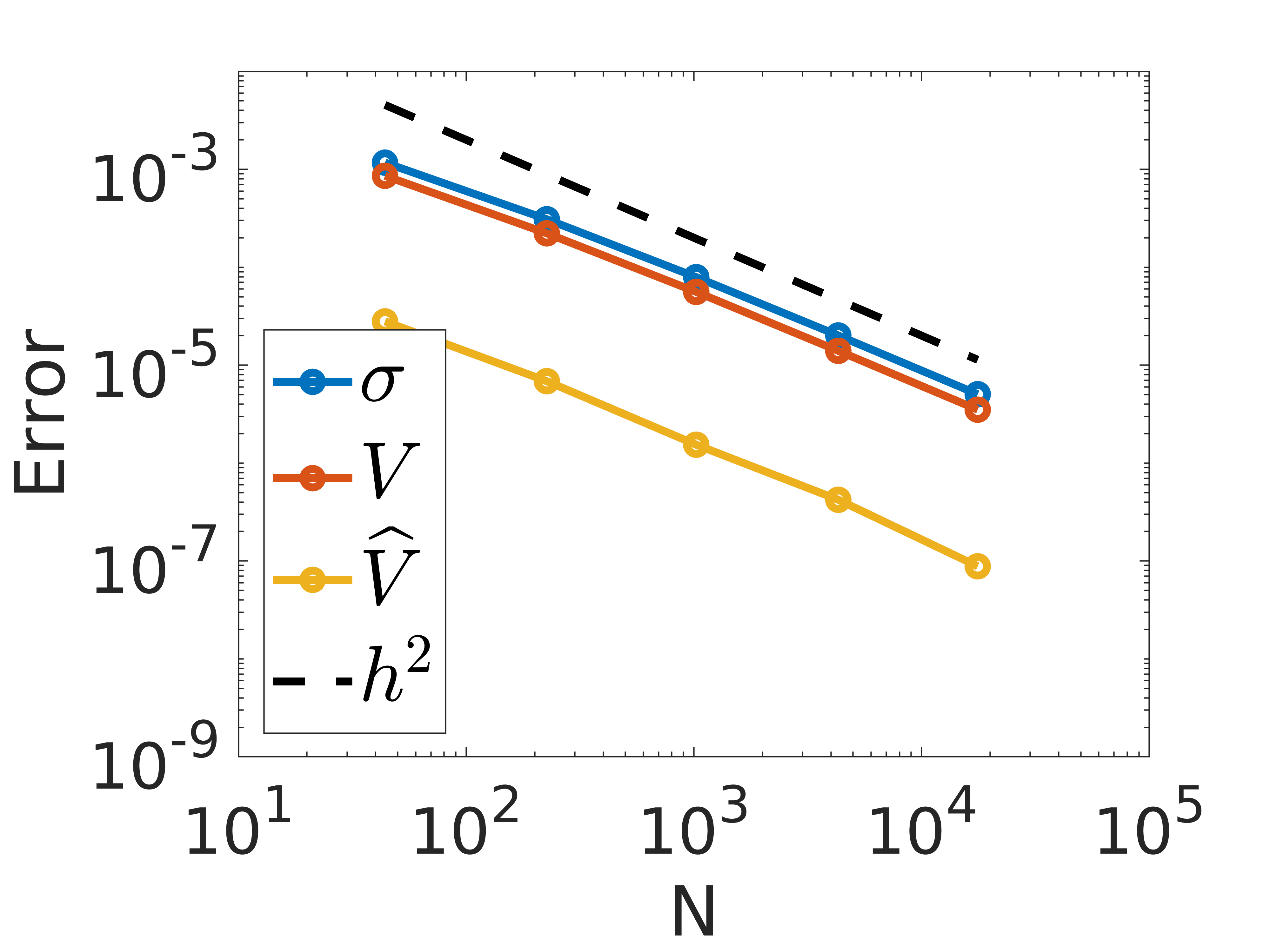} & \includegraphics[width=0.3\linewidth]{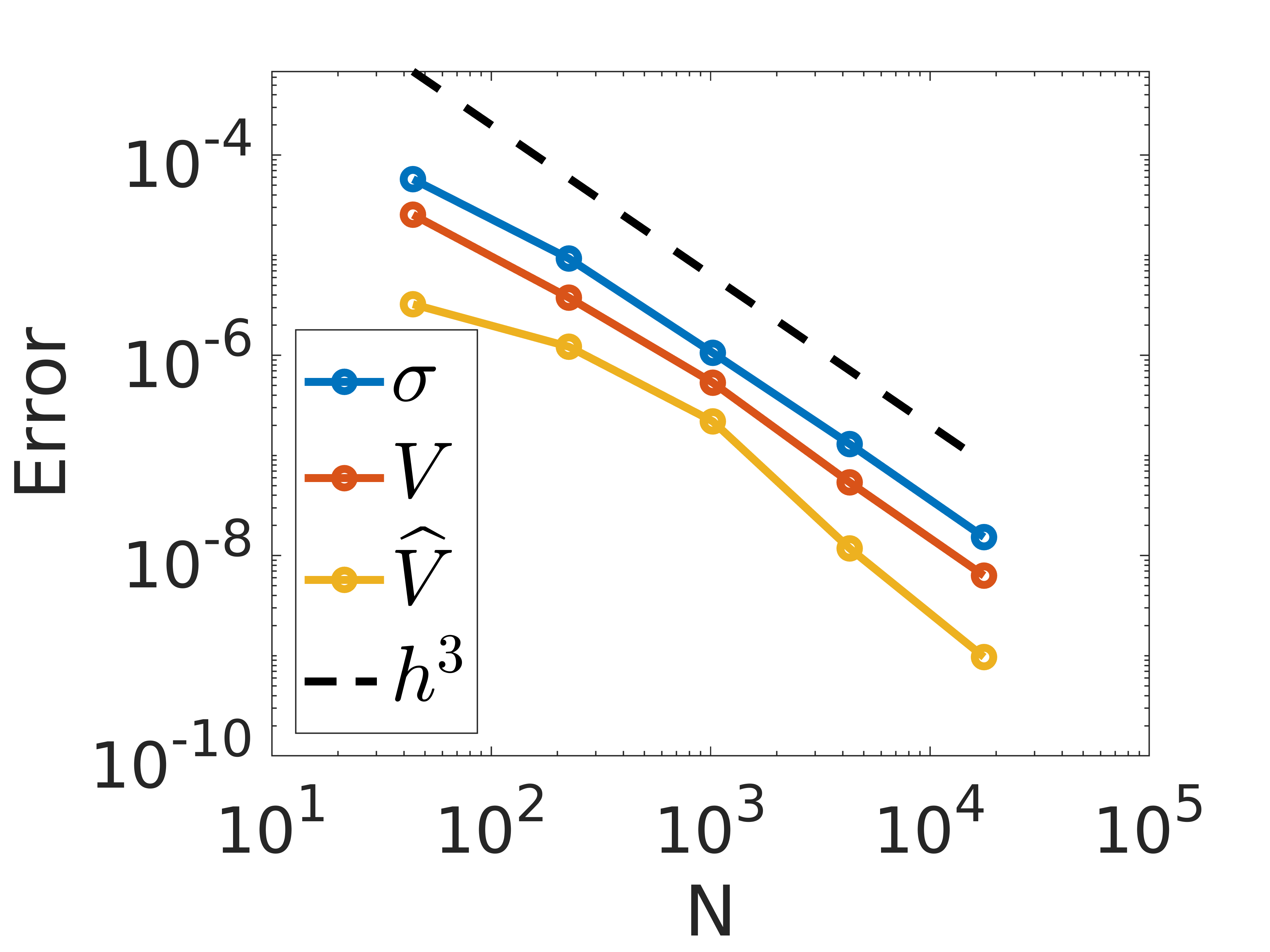} & \includegraphics[width=0.3\linewidth]{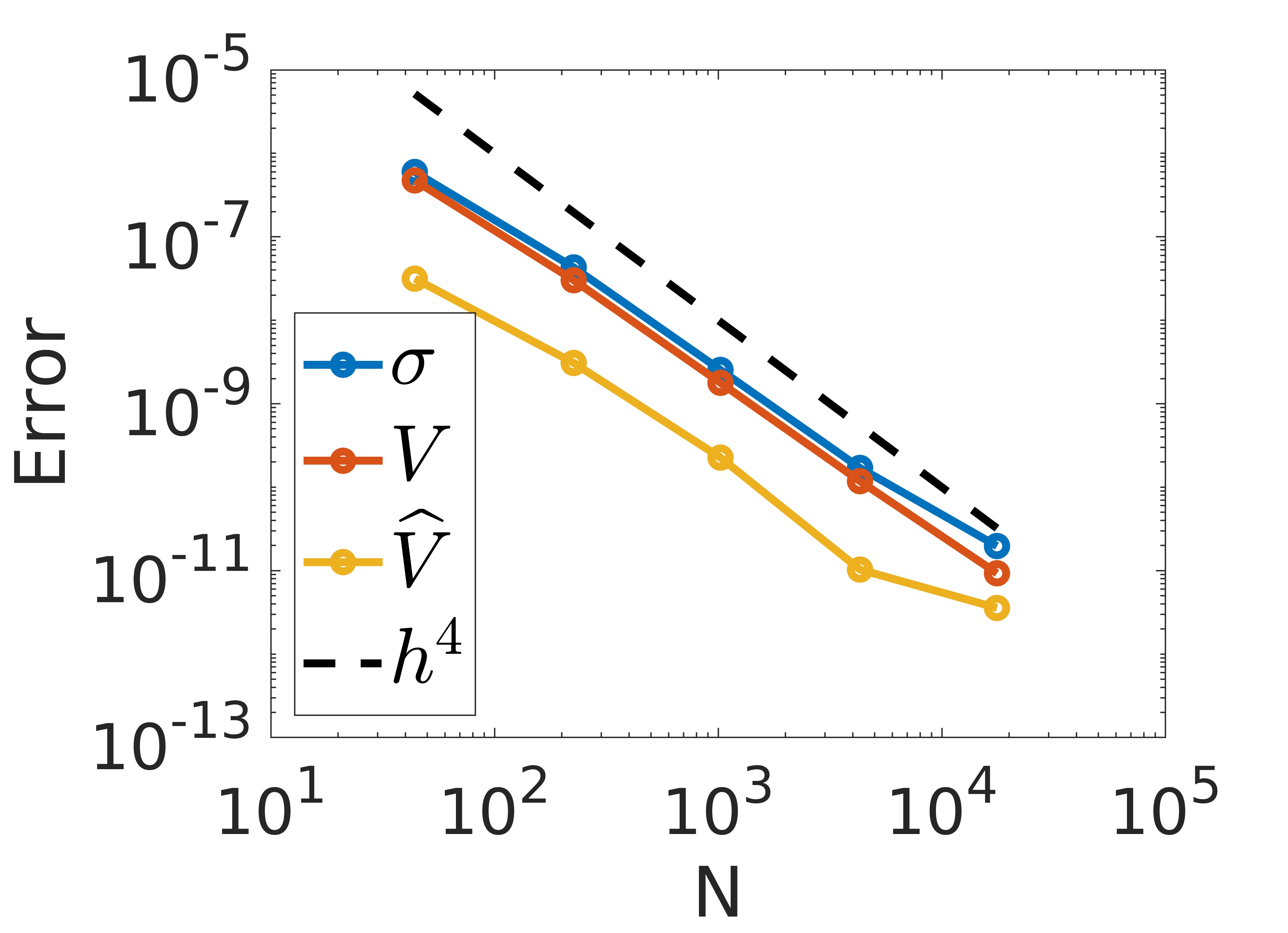} \\
\end{tabular}
\caption{Convergence history for the variables associated to the three systems--- state, adjoint and deformation---for different orders of polynomial bases. The trace variables (denoted with a ``\,hat \;$\widehat\cdot$\;\,'' and plotted in yellow) all converge with an additional power of $h$.}\label{fig:StateConvergence}
\end{figure}

\subsection{Experiment 2: Shape optimization with a manufactured example.}\label{sec:Manufactured}
Consider the target function
\[
\widetilde{y}(x_1,x_2):=\left(|\boldsymbol{x}|^2-\dfrac{1}{2\pi}\right)\left(|\boldsymbol{x}|^2-0.05^2\right),
\]
along with the problem data
\[
f(x_1,x_2)=\left(4|\boldsymbol{x}|^2-\dfrac{1}{2\pi}-0.05^2\right)\,, \qquad  g=0\,, \qquad \text{ and } \qquad a=1.
\]
With these parameters, the solution to the state equation \eqref{pde_1} is $y(x_1,x_2)=\tfrac{1}{4}\,\widetilde{y}(x_1,x_2)$. 

Let $\mathcal U:=\{\boldsymbol{x}=(x_1,x_2)\in (-1,1)^2: |\boldsymbol{x}|>0.05\} $ and consider the set of admissible domains

\[
\mathcal O := \left\{ \Omega: \Omega\subset \mathcal U \;\text{ and }\; \mu(\Omega) = m_0:=\pi\left(\dfrac{1}{2\pi}-0.05^2\right) \; \text{with} \; \{|\boldsymbol x| = 0.05\}\subset \partial \Omega\right\}\,.
\]
In this case, we expect the optimal domain to be
 \[
 \widehat{\Omega}=\left\{\boldsymbol{x}=(x_1,x_2)\in \mathcal{U}: |\boldsymbol{x}|^2<\dfrac{1}{2\pi}\right\}\,,
 \]
and the optimal energy to be given by 
\[
J(\widehat{\Omega};y)=\dfrac{9}{8} \int_{\widehat{\Omega}}\widetilde{y}^2(x_1,x_2) \,\,d(x_1,x_2) = \dfrac{9\pi}{16}\int_{0.05}^{1/\sqrt{2\pi}} r \left(r^2-\dfrac{1}{2\pi}\right)^2 (r^2-0.05^2)^2\,dr\approx 6.83 \times 10^{-8}.
\]
Algorithm \ref{alg2_shp_opt} was employed using the value $\epsilon = 10^{-4}$ (see eq. \eqref{LM:update}) and starting from the initial guess for $\Omega$ displayed in red on the leftmost panel of Figure \ref{fig:optimal-shape}. To enforce the condition that $|\boldsymbol x| = 0.05$ remains a subset of $\partial\Omega$, the initial guess includes this segment in its boundary and the Dirichlet condition $\boldsymbol V = 0$ on $\{|\boldsymbol x| = 0.05\}$ is included in the BVP for the deformation field.

The domain will be approximated by an interpolatory polygonal approximation stored in an ordered vector of points. Every point $p_i$ in the vector represents a vertex of the polygon and points with consecutive indices are connected through a straight segment. The approximated domain will be updated  as prescribed by Algorithm \ref{alg2_shp_opt} using the discrete approximations $y_h$ and $z_h$.  Each of the vertices $p_i$ is displaced according to the discrete deformation field as $p_i \,\mapsto \, p_i + s\boldsymbol V_h(p_i)$, where the value of the step size $s$ is updated in every step through an Armijo line search. As the shape of the approximated domain changes, the computational mesh (displayed in black on Figure \ref{fig:optimal-shape}) is updated without the need for remeshing by simply activating or deactivating background elements.

The accuracy of an approximate domain $\Omega_h$ is measured through its Hausdorff distance from the optimal domain $\widehat\Omega$, which is defined as
\[
d_H(\Omega_h,\widehat\Omega) :=\max\left\{ \sup_{x\in\Omega_h\phantom{^;}}\,\inf_{y\in\widehat\Omega\phantom{^T}}\|x-y\|\;,\;\sup_{y\in\widehat\Omega}\,\inf_{x\in\Omega_h\phantom{^{\widehat{\widehat T}}}}\|x-y\|\right\}\,,
\]
where the quantity $\|x-y\|$ is the usual Euclidean distance between the points $x$ and $y$. In the numerical experiments, we compute the Hausdorff distance using the implementation \cite{Danziger2026}.

The initial guess for the domain $\Omega$ consisted of a polygonal approximation with vertices at $N=2000$ points $p_i$ not located on the fixed inner circle $\{|\boldsymbol x = 0.05|\}$ and distributed as shown in the leftmost panel of Figure \ref{fig:optimal-shape}. At every step of the iterative process, the state, adjoint and deformation fields were approximated with linear elements on a shape-regular and uniform background triangulation with mesh parameter $h=3.125\times10^{-2}$ (depicted in Figure \ref{fig:optimal-shape}). The approximation after 60 iterations is displayed on the rightmost panel in Figure \ref{fig:optimal-shape}. The maximum length of a segment in our final polygonal approximation is $\ell = 1.8\times 10^{-3}$. Using a traditional approach reliant on an interpolatory uniform mesh, where the number of elements is proportional to $h^{-2}$, such a resolution would require approximately $3\times 10^{5}$ elements compared to the 1924 elements in our grid. This dramatic reduction (a factor of approximately $1.5\times10^{2}$) on the number of elements required underscores the efficiency of the proposed technique. Although the number of elements employed by an interpolatory approximation could be reduced through the use of a graded mesh, such a grid would have to be recomputed at every step of the iteration, increasing the complexity and computing time. The strength of the transfer path method in this context is precisely that no special treatment needs to be done around the boundary and no updating of the mesh is required over iterations. The method works even for a fixed, shape-regular triangulation that requires no updates.

\begin{figure}[tb]
\begin{tabular}{cccc}
\includegraphics[width=0.225\linewidth]{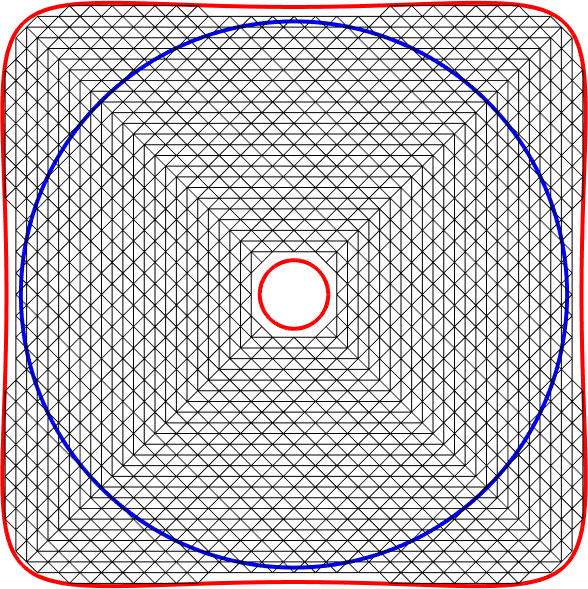} &
\includegraphics[width=0.225\linewidth]{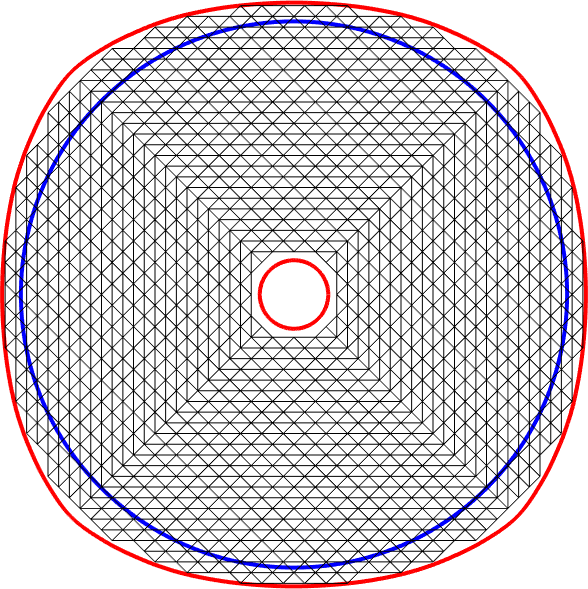} &
    \includegraphics[width=0.225\linewidth]{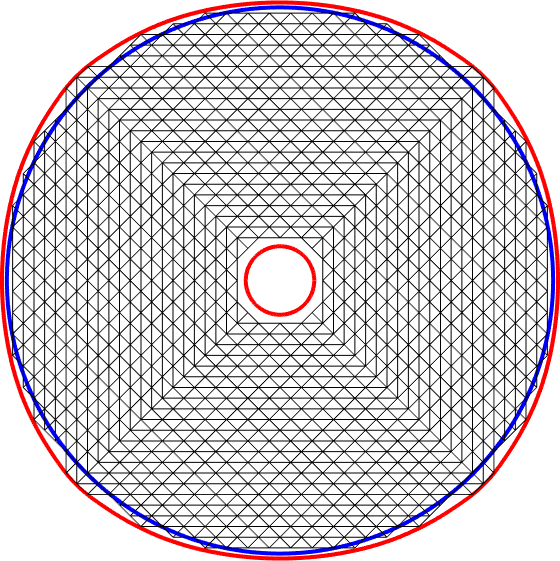} &
\includegraphics[width=0.225\linewidth]{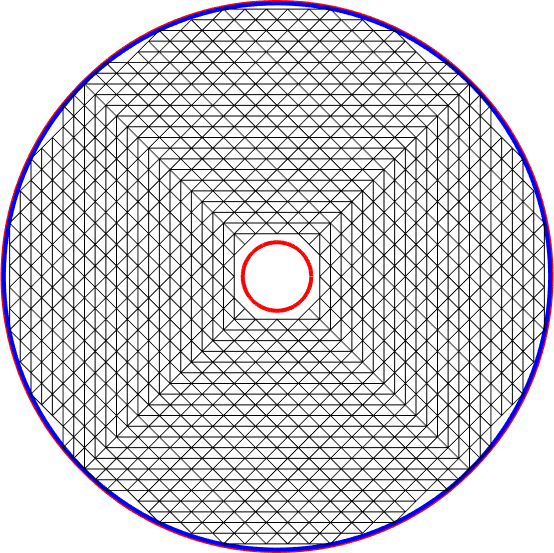} \\
Initial guess & 30 iterations & 50 iterations & 60 iterations
\end{tabular}
\caption{Numerical approximation to the optimal shape for iterations: 0, 30, 50, 60. The optimal shape is drawn in blue, while the numerical approximation is drawn in red. The computational mesh used for solving the state, adjoint and deformation equations is plotted in black. After 60 iterations the two curves are indistinguishable to the naked eye.}\label{fig:optimal-shape}
\end{figure}

\begin{figure}[tb]
\begin{tabular}{cccc}
Energy & Area difference  & Lagrange multiplier & Hausdorff distance \\
\hspace{-.32cm}\includegraphics[width=0.25\linewidth]{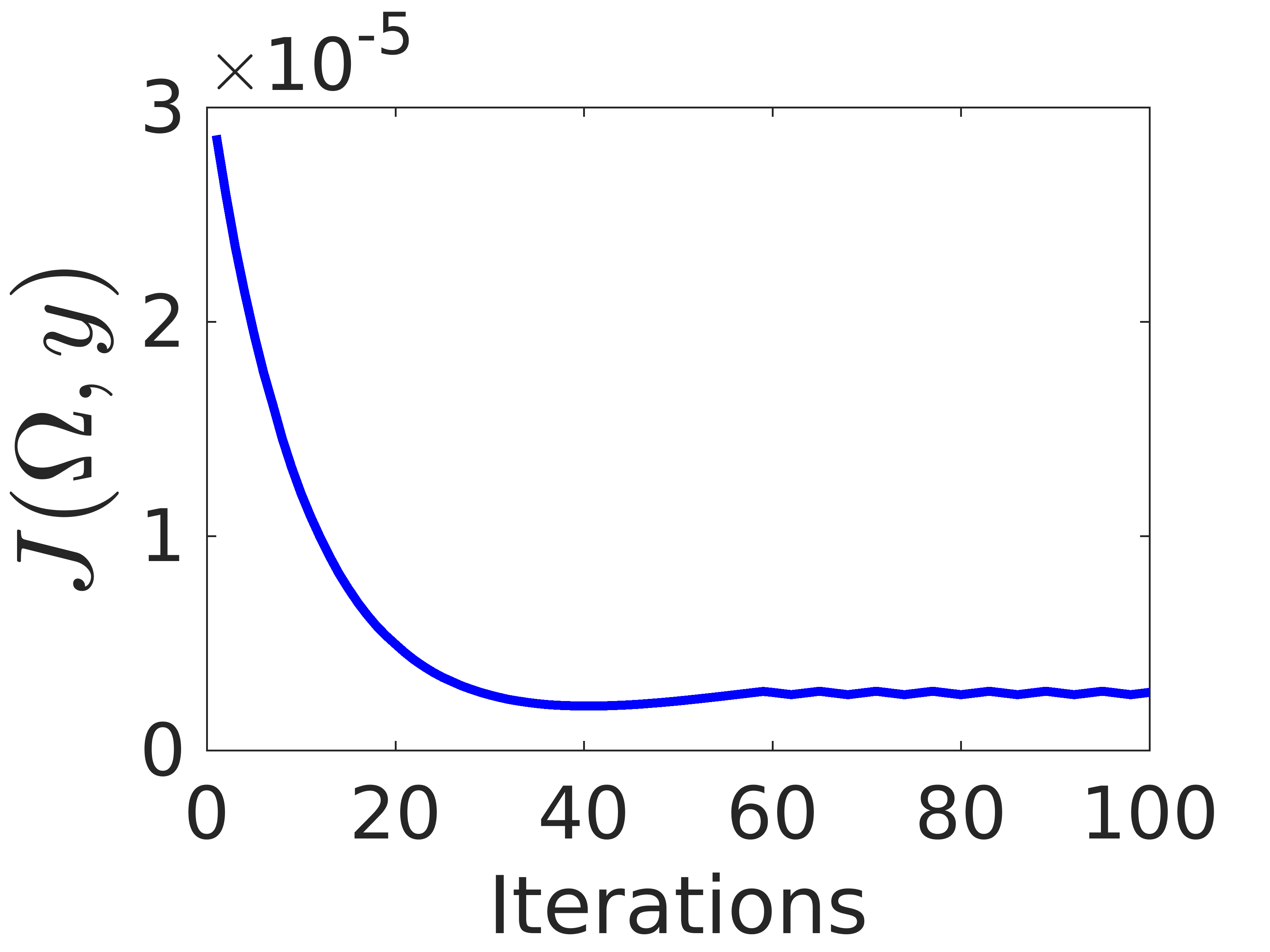} &
\hspace{-.35cm}\includegraphics[width=0.25\linewidth]{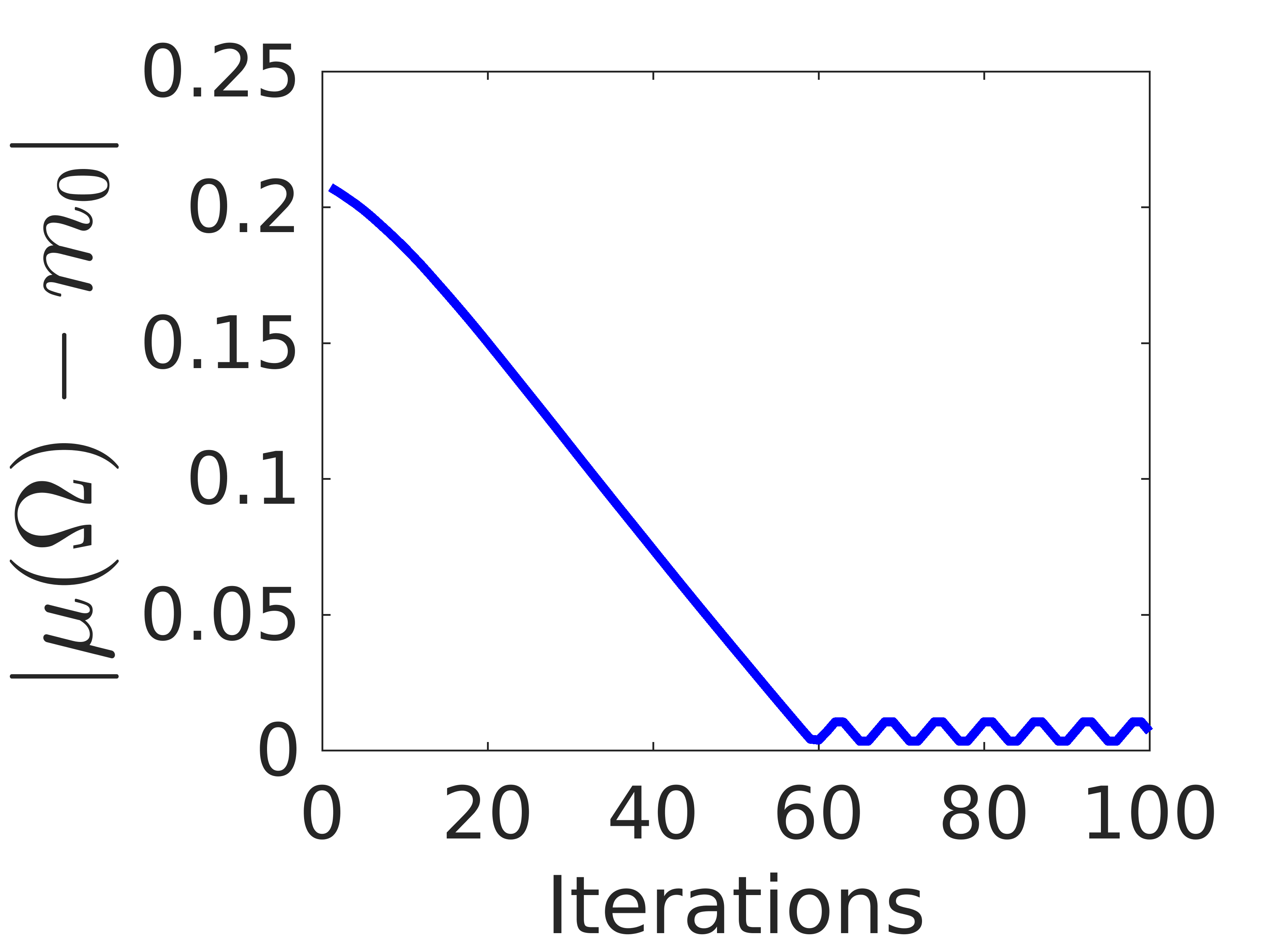} &
\hspace{-.35cm}\includegraphics[width=0.25\linewidth]{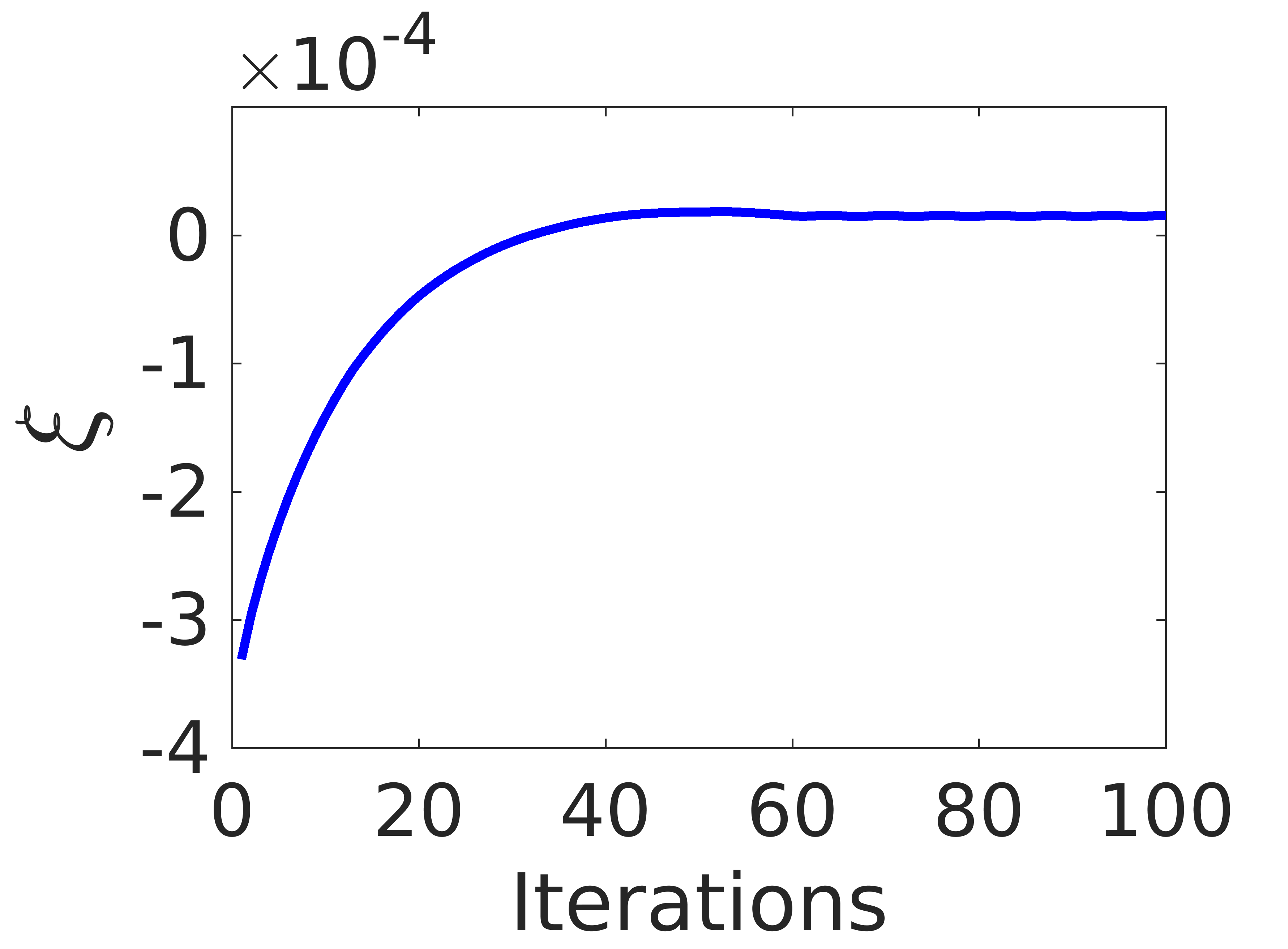} &
\hspace{-.35cm}\includegraphics[width=0.25\linewidth]{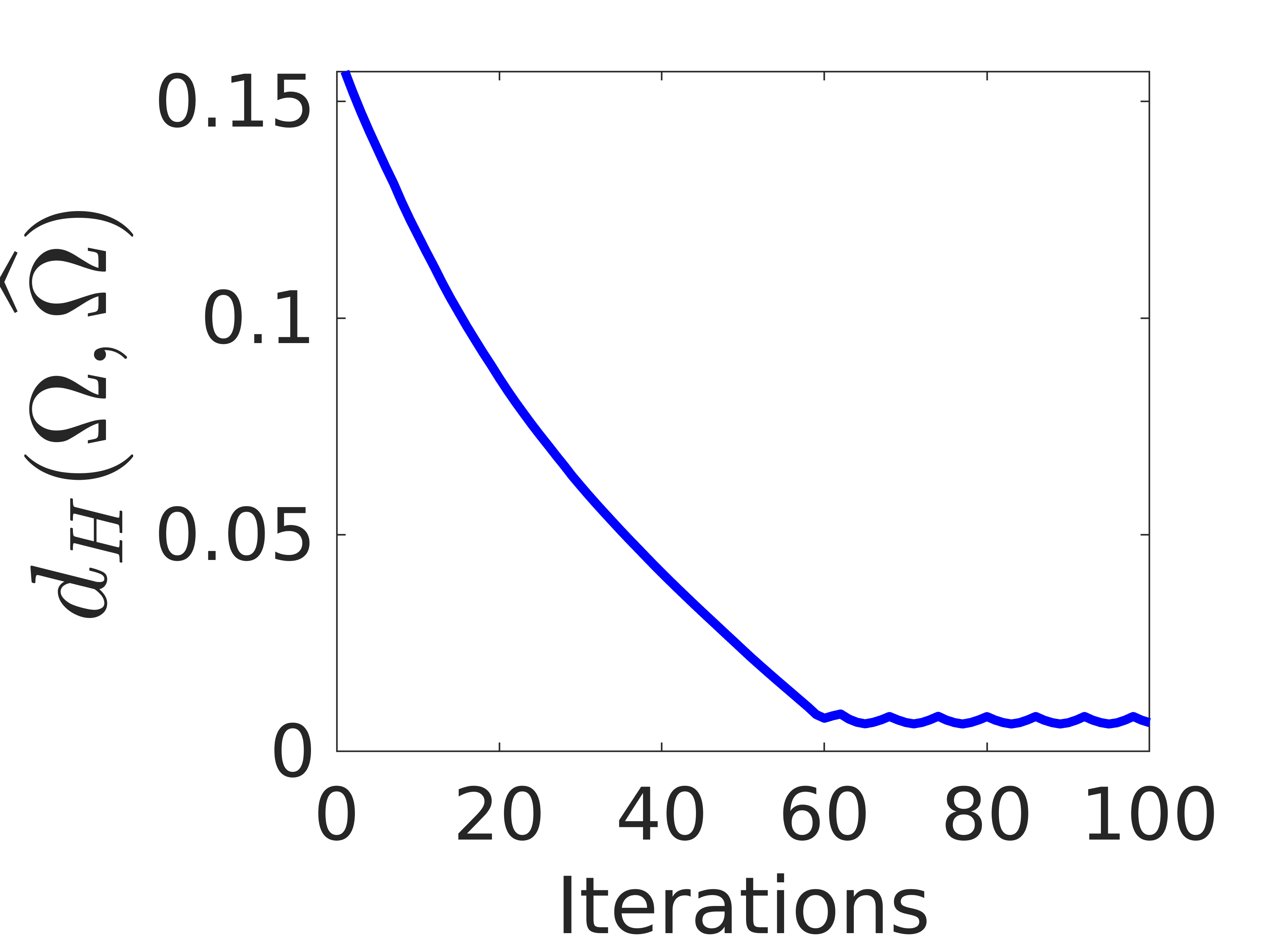}
\end{tabular}
\caption{Left: The target energy functional $J(\Omega;y)$ decreases as the number of iterations grows, settling after about 50 iterations. Center left: The violation of the area constraint decreases steadily and reaches a steady state after about 60 iterations. Center right: The value of the Lagrange multiplier grows with the number of iterations, switching from assigning more importance to the shape of the domain at the beginning, towards enforcement of the area constraint at the end. Right: The Hausdorff distance measures the proximity between the approximated domain and the optimum.}\label{fig:energy}
\end{figure}

The value of the Lagrange multiplier is updated according to equation \eqref{LM:update}. As shown in the right panel of Figure \ref{fig:energy}, initially the value of the Lagrange multiplier $\xi$ is small, penalizing mostly the energy term $J(\Omega;y)$ that determines the shape of the domain. As the algorithm progresses, the value of  $\xi$ increases, penalizing mostly the violation of the volume constraint. Accordingly, the value of the energy decreases sharply at the beginning (left panel of Figure \ref{fig:energy} and increases a little before settling down, while the value of the area difference decreases steadily until it oscillates around a steady value (center panel of Figure \ref{fig:energy}). The slight increase in the energy towards the end of the process reflects the fact that a more energetically efficient shape can be obtained (similar to the red curve in the center--left panel of Figure \ref{fig:optimal-shape}) if the volume constraint is dropped. As it can be seen in the leftmost panel of Figure \ref{fig:energy}, after 60 iterations the approximate value of $J$ oscillates between $2.6\times 10^{-6}$ and $2.7\times 10^{-6}$. The oscillatory behavior reflects the tension between the need for minimizing the value of the energy and satisfying the area constraint. The evolution of the state variable $y$ as the shape of the domain changes is depicted in Figure \ref{fig:solutions}.

\begin{figure}[bt]
\begin{tabular}{cccc}
\includegraphics[width=0.225\linewidth]{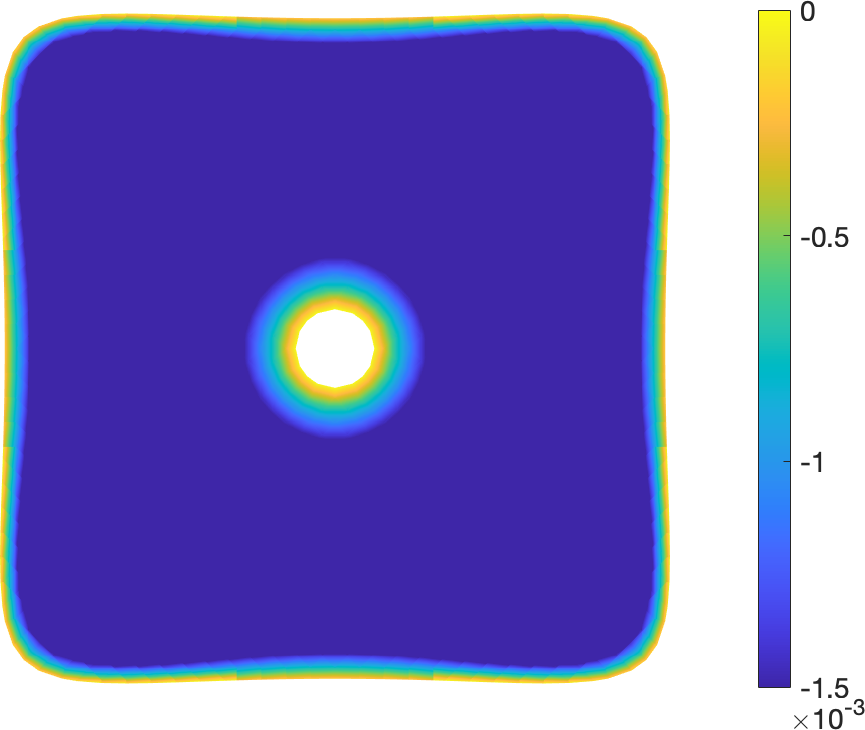} &
\includegraphics[width=0.225\linewidth]{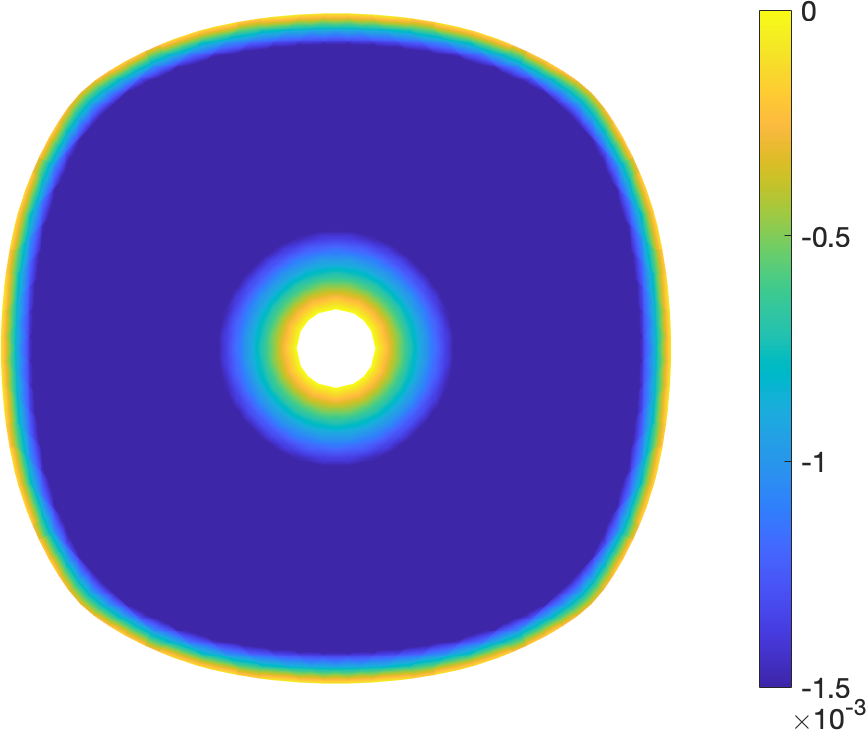} &
\includegraphics[width=0.225\linewidth]{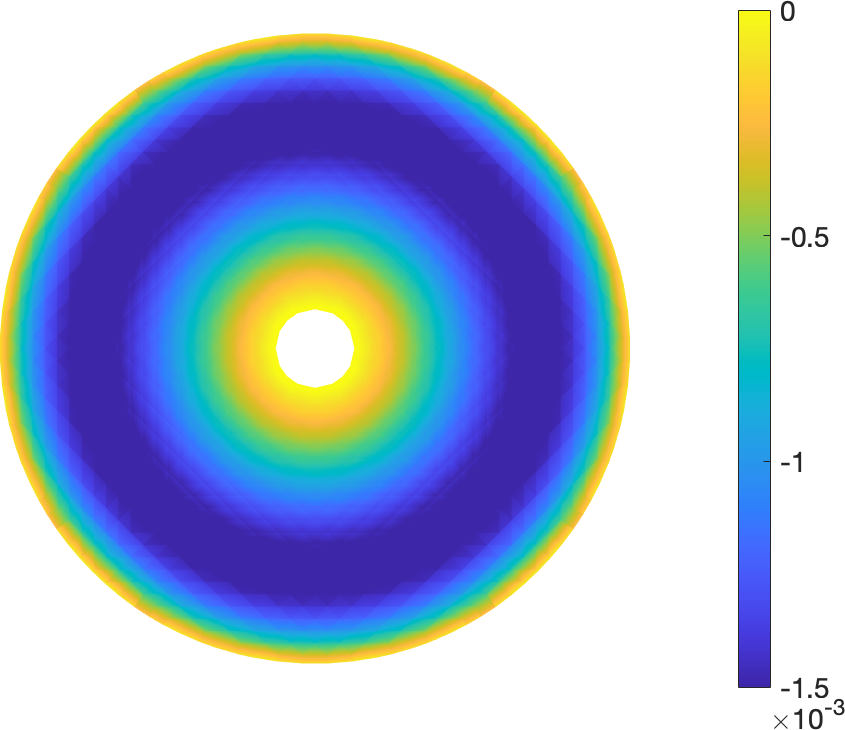} &
\includegraphics[width=0.225\linewidth]{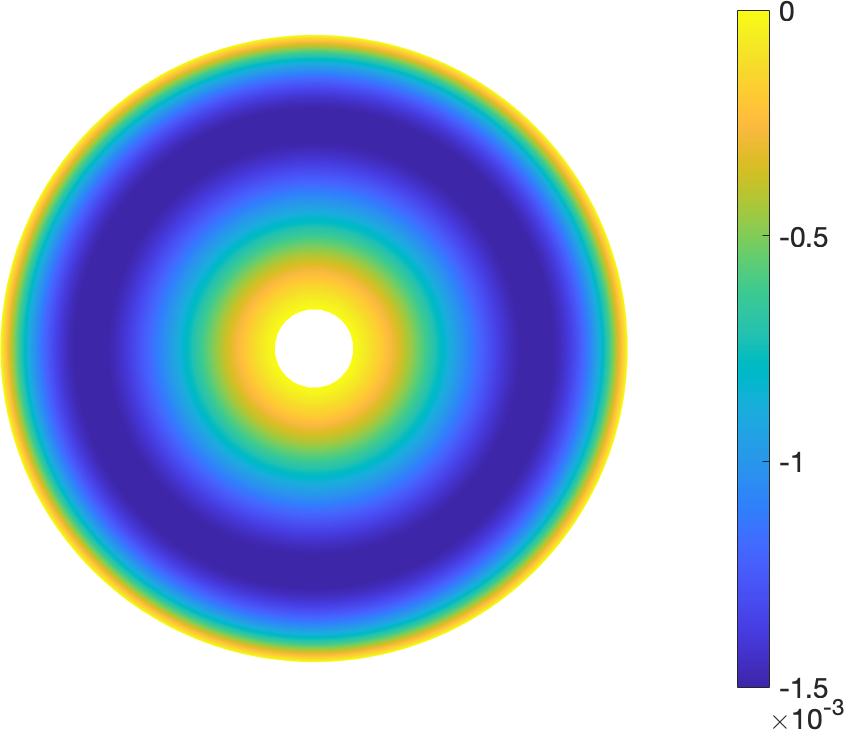} \\
Initial guess for $y$ & $y$ after 30 iterations & $y$ after 60 iterations & Optimal $y$  
\end{tabular}
    \caption{The three panels on the left depict the evolution of the state variable from the initial guess until 60 iterations. The rightmost panel shows the optimal state, which equals $\widetilde{y}$/4.} \label{fig:solutions}
\end{figure}

\subsection{On the impact of the quality of the boundary representation and the numerical discretization}\label{sec:GeometricApproximation}

There are two convergence processes happening in this problem: one has to do with the iterative minimization of the functional (optimization error), and another one with the numerical approximation of the optimal shape and the variables associated to the shape gradient (approximation error). If we denote by $(\widehat \Omega, \widehat y)$ the exact optimal domain and state variable, by $(\Omega^{k}, y^{k} )$ the exact domain and state variable updates at the $k$-th iteration, and by $(\Omega^{k}_h, y^{k}_h)$ the numerical approximations to the domain and state variable updates at the $k$-th iteration (these last ones are the quantities that we compute and can actually manipulate), it follows that
\[
  \underbrace{\left|J\left(\widehat \Omega, \widehat y\right) - J\left(\Omega^k_h, y^k_h\right)\right|}_\text{Computational error}\leq \underbrace{\left|J\left(\widehat \Omega,\widehat y\right) - J\left(\Omega^k,y^k\right)\right|}_\text{Optimization error} + \underbrace{\left|J\left(\Omega^k,y^k\right) - J\left(\Omega^k_h,y^k_h\right)\right|}_\text{Approximation error}\,,
\]
from which we see that these processes are independent of each other. The approximation error further contains contributions due to the accuracy of the geometric approximation to the domains ($\Omega^k_h\approx \Omega^k$) and to the state and adjoint variables at every iteration ($y^k_h\approx y^k$) both of which depend on factors such as the order of the numerical method used, the size of the mesh, and the order and quality of the discrete representation of $\Omega^k_h$. It is possible to further split it as
\[
\underbrace{\left|J\left(\Omega^k,y^k\right) - J\left(\Omega^k_h,y^k_h\right)\right|}_\text{Approximation error} \leq \underbrace{\left|J\left(\Omega^k,y^k\right) - J\left(\Omega^k_h,y^k\right)\right|}_\text{Geometric error} + \underbrace{\left|J\left(\Omega^k_h,y^k\right)- J\left(\Omega^k_h,y^k_h\right)\right|}_\text{Discretization error}\,.
\] 
Since the approximate iterative search will conclude when the computational error falls below a predetermined tolerance, it follows that for a ``fine enough" discretization the leading error term (and therefore the number of iterations required) will depend only on the nuances of the particular optimization technique, the choice of initial guess, the prescribed tolerance, the properties of the functional and the constraints, etc. Our use of a gradient descent technique in the numerical experiment simply follows a choice that is popular in the shape optimization literature.

This can be seen in Figure \ref{fig:GeometryAndDiscretiztion}, where the results of variants of the experiment from Section \ref{sec:Manufactured} are presented. These experiments use the same computational triangulation of the PDE solves depicted in Figure \ref{fig:optimal-shape} but vary along two axes: the degree of the geometric approximation of the domain $\Omega$ (piecewise linear segments and cubic splines), and the degree of the HDG discretization of the PDEs ($k=1$ and $k=3$) for an increasing number of points representing the boundary ($N=40,80,160,320$). As the figure shows, the convergence curves of the Hausdorff distance vs. iteration count follow almost undistinguishable trajectories. Moreover, all reach the ``computational minimum" after 60 iterations regardless of the number of points of the representation, the use of piecewise linear or spline representations, or the degree of the HDG basis functions.

\begin{figure}[b]
\begin{tabular}{cccc}
Geometry: Polygonal & Geometry: Polygonal & Geometry: Cubic spline & Geometry: Cubic spline\\
HDG Discretization: $k=1$ & HDG  Discretization: $k=3$ & HDG  Discretization: $k=1$ & HDG  Discretization: $k=3$\\
\hspace{-.22cm}\includegraphics[width=0.25\linewidth]{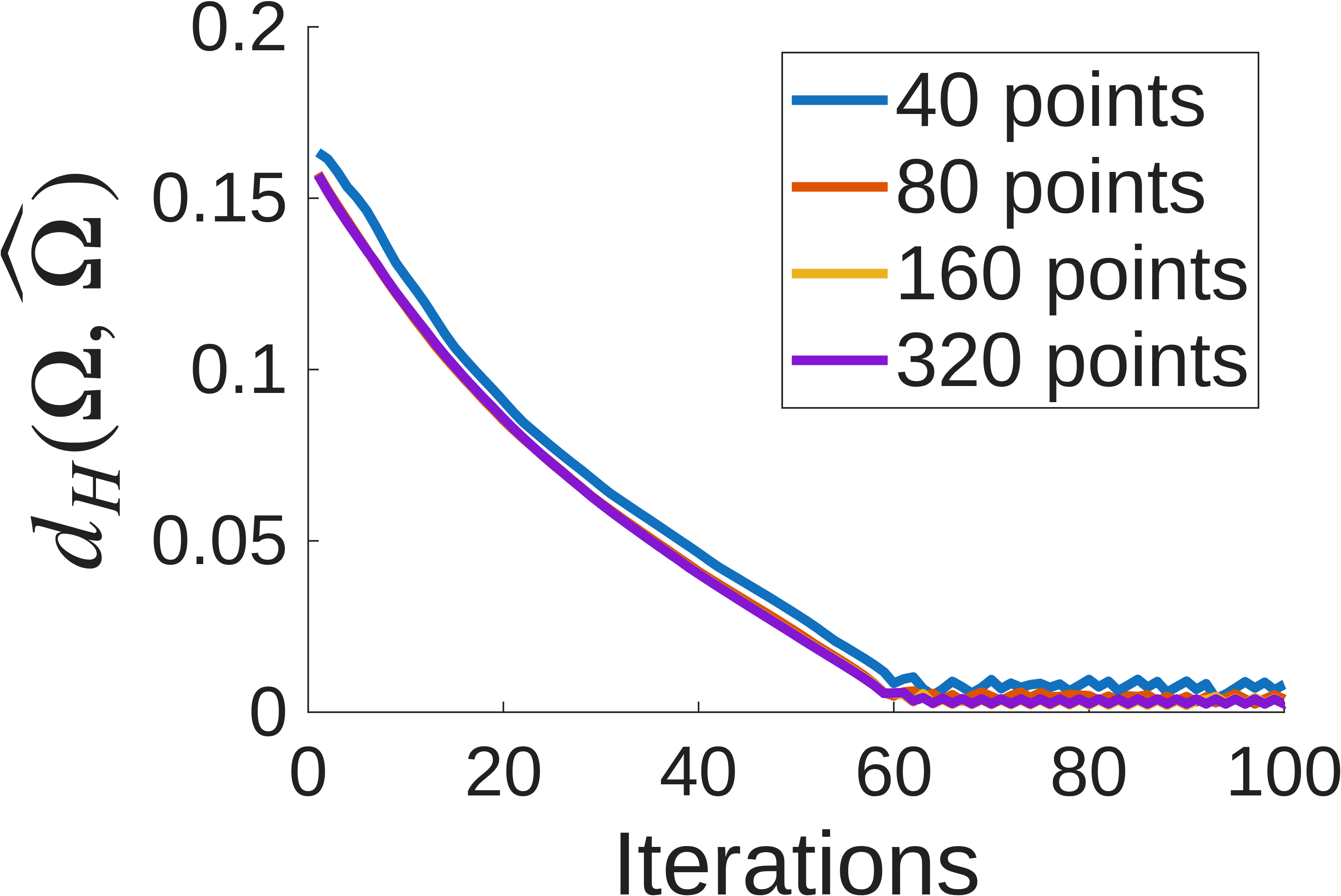} & \hspace{-.52cm}
\includegraphics[width=0.25\linewidth]{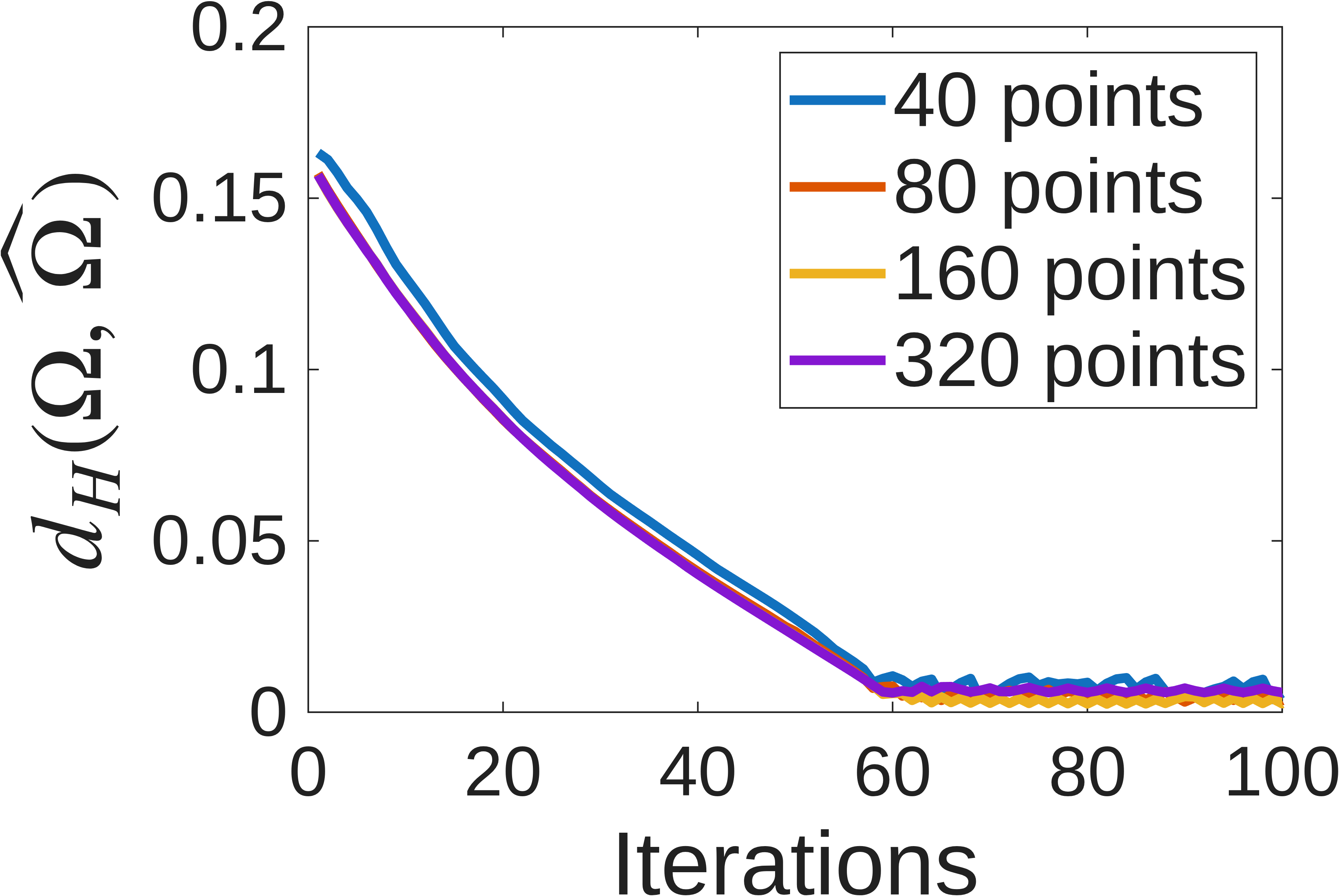} & \hspace{-.525cm}
\includegraphics[width=0.25\linewidth]{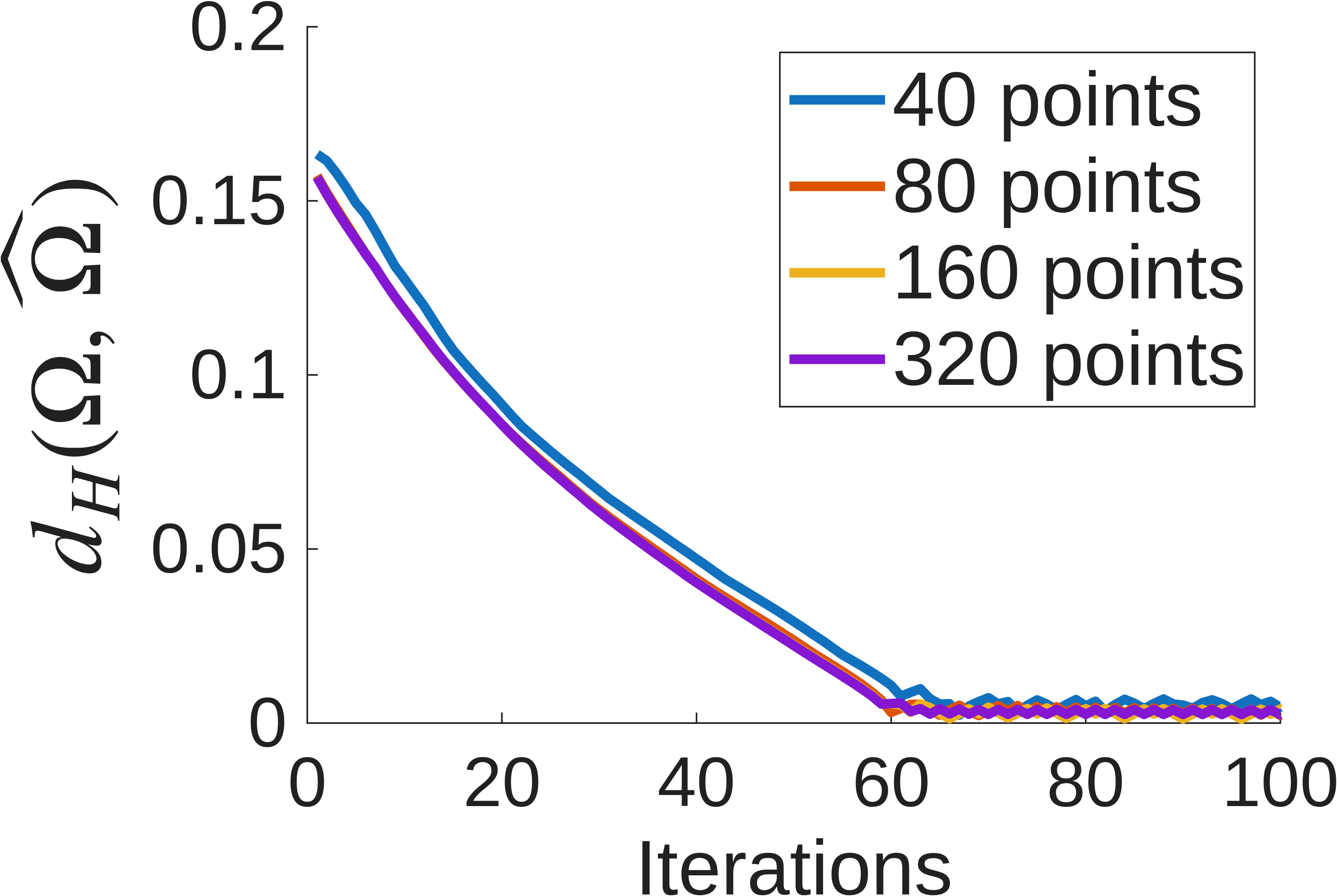} & \hspace{-.525cm}
\includegraphics[width=0.25\linewidth]{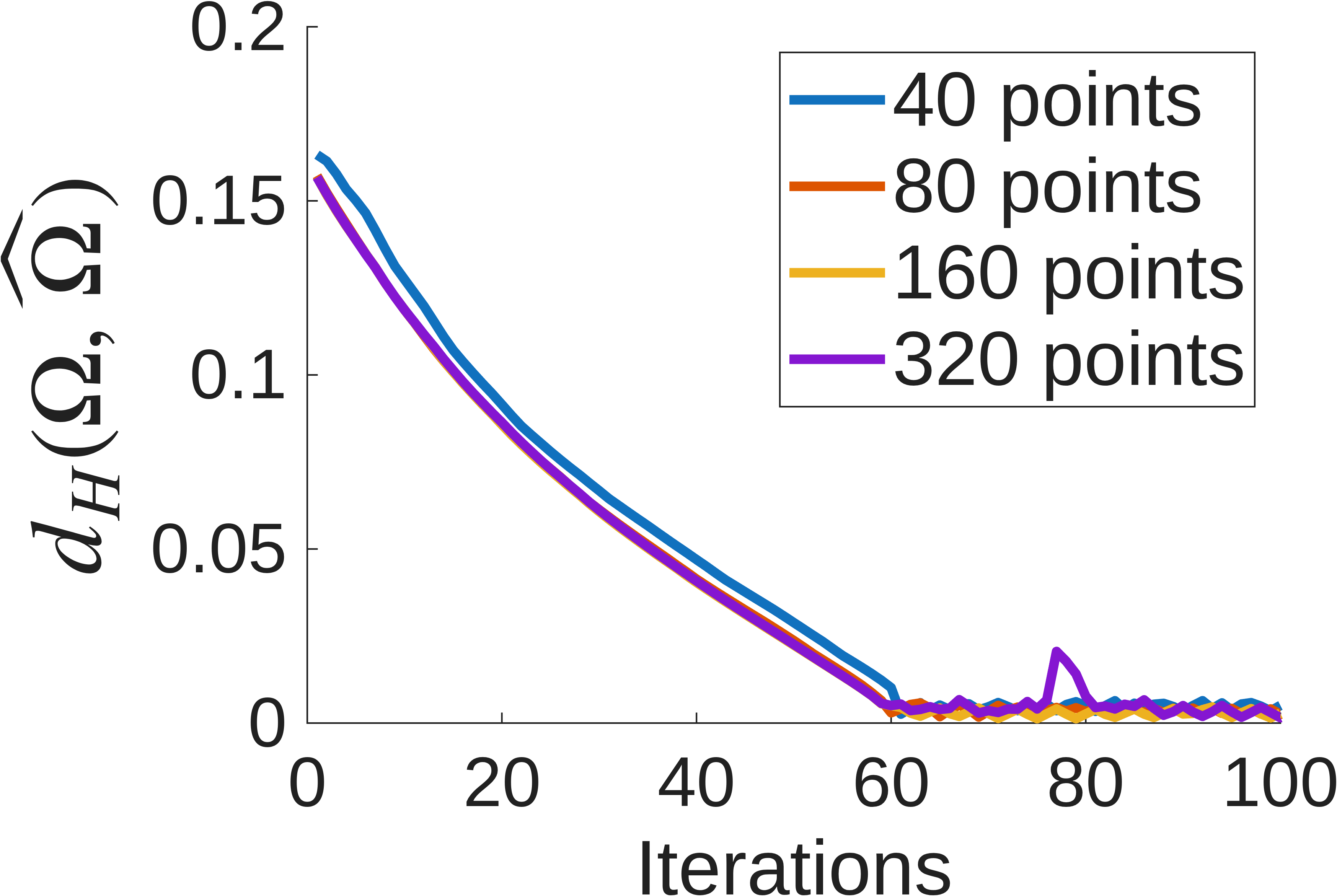} 
\end{tabular}
    \caption{The optimization error is the main driver of the convergence to the optimal solution. The behavior is almost identical regardless of the order of the geometric representation and the HDG solver.} \label{fig:GeometryAndDiscretiztion}
\end{figure}

\section{Concluding remarks}

This article was mostly devoted to the analysis of the unfitted HDG discretizations associated to the state, adjoint and deformation problems arising from a model problem in geometric shape optimization. The analysis faces the, now well--known, challenges associated to a Neumann type boundary condition stemming from condition \eqref{mixed_Dh_velocity_c} in the problem for the deformation field. The challenge manifests itself in the error bounds for $\boldsymbol V$ through the appearance of the term $\norm{(G-G_h)\bm n\circ\bm\phi}_{\Gamma_h^{N}}$, whose optimal bound remains elusive. Nevertheless, numerical experiments (as is the case of Experiment 1) repeatedly show that the optimal order may be feasible, even in the when employing triangulations that are not fully admissible in the sense defined in Section \ref{sec:Geometry}. Ensuring these bounds theoretically remains an open challenge.

However, an interesting and promising path arises from the findings of the second numerical experiment. The fact that the transfer path technique enables a very accurate and efficient description of the optimal boundary. Efficiency in this context refers to the fact that the maximum length of the sides of the polygonal approximation can be made at least as small as $h^2$ for a regular mesh of diameter $h$. This feature has the potential of greatly speeding up computations and, very importantly, is independent of the method chosen for the discretization of the associated boundary value problems, as long as they are cast in a mixed form. Moreover, the fact that the transfer path technique enables the use of a simple (essentially  structured) fixed, simple shape-regular triangulation that does not need to be updated after every iteration can not be understated. Not only the implementation and geometric treatments become simple, but sidestepping the need for constant remeshing can single-handedly provide significant time reductions. Given that these features are independent of HDG, their analysis and further explorations are the subject of ongoing work and will be shared in a separate communication \cite{HeSaSo2025}. 

\section*{Conflict of interest}
%
The authors do not have conflicts of interest to declare.
%
\section*{Funding statement}
%
Esteban Henr\'iquez and Manuel E. Solano were partially supported by ANID-Chile through Basal Project FB210005 and Fondecyt 1240183. Tonatiuh S\'anchez-Vizuet was partially supported by the   United States National Science Foundation through the grant NSF-DMS-2137305 ``LEAPS-MPS: Hybridizable discontinuous Galerkin methods for non-linear integro-differential boundary value problems in magnetic plasma confinement''.

\section*{Author contributions}
%
Because of the nature of mathematical research, authors are considered on equal footing and are always listed in alphabetical order. In other words, \textit{as opposed to other disciplines, there is no ``first author"}. Esteban Henr\'iquez, Tonatiuh S\'anchez-Vizuet, and Manuel Solano all contributed in equal proportion; the ordering has been determined following purely alphabetical considerations.

\section*{Acknowledgement}
%
The authors are thankful to ANID-Chile and the U. S. National Science Foundation for their financial support. 

\appendix
\section{Auxiliary estimates}

For any $e\in \mathcal{E}_h^{\partial}$, any point $\bm x$ lying on the face $e$ and any smooth enough function $\bm v$ defined in $K^{ext}_e$, we set
\begin{equation}
    \label{def_Lambda}
    \Lambda^{\bm v}(\bm x) \,:=\, \frac{1}{l(\bm x)}\,\int_{0}^{l(\bm x)}(\bm v(\bm x\,+\,s\,\bm n_h)\,-\,\bm v(\bm x))\cdot\bm n_h \, ds\,.
\end{equation}
For each $e\in \mathcal{E}_h^{\partial}$, and vector--valued functions $\boldsymbol v$ and $\boldsymbol n_h$ such that their scalar product $\boldsymbol v\cdot \boldsymbol u_h \in H^1(K^{ext}_e)$, it satisfies (cf. \cite[Lemma 5.2]{CoQiuSo2014})
\begin{equation}\label{est:Lambda1}
    \norm{\Lambda^{\bm v}}_{e,l}\,\leq\, \frac{1}{\sqrt{3}}\,r_{e}\,\norm{\partial_{\bm n_h}(\bm v\cdot\bm n_h)}_{K^{ext}_{e},(h^{\perp})^2}\,,
\end{equation}
where we have used the notation $\partial_{\boldsymbol n_h} u := \nabla u\cdot\boldsymbol n_h$. Moreover, if $\bm v\in [\mathbb{P}_k(K_{e})]^{d}$, we have that
\begin{equation}\label{est:Lambda2}
    \norm{\Lambda^{\bm v}}_{e,l}\,\leq\, \frac{1}{\sqrt{3}}\,r_e^{3/2}\,C^{ext}_e \,C^{inv}_{e}\,\norm{\bm v}_{K_{e}}\,.
\end{equation}
We also recall the discrete trace inequality, whose proof can be found in \cite[Lemma 1.52]{di2011mathematical}, for instance.
Let $K$ be an  element of $\mathcal{T}_h$ with diameter $h$ and $v\in \mathbb{P}_{k}(K)$. Then there exists $C_{tr}>0$, depending only on $k$ and mesh regularity, such that
\begin{equation}
    \label{trace_inequality}
    h^{1/2}\,\norm{v}_{e}\,\leq\, C_{tr}\,\norm{v}_{K}\,,
\end{equation}
where $e$ is a edge or face of $K$.

The following estimates are used in Section \ref{sec:ErrorEstimates} to find estimates for the term $\norm{(G-G_h)\bm n\circ\bm\phi}_{\Gamma_h^N}$. 

\begin{lemma}
For a polynomial $p$ defined over a boundary element $K_e$ associated with the edge $e$ and the extension patch $K_e^{ext}$, and for all $v \in H^1(K_e)$ the following estimates hold
\begin{align}
\label{ineq:discrete-trace-h} h_e^{1/2}  \,\|p\|_{\Gamma_e}\,\lesssim\,&     \|p\|_{K^{ext}_e\cup K_e},\\
\label{ineq:discrete-trace-h-2}
  h_e^{1/2} \,\|v\|_{\Gamma_e}\,\lesssim \,& \left( \|v\|_{K^{ext}_e\cup K_e}^2 \,+\, h_e^2\, \|\nabla v\|_{K^{ext}_e\cup K_e}^2\right)^{1/2}\,.
\end{align}
\end{lemma}
\begin{proof}
    Let us define the following reference patch
    \begin{equation*}
        \widehat{K} := \{ \widehat{\boldsymbol{y}}:\  \widehat{\boldsymbol{y}} = \boldsymbol{y} / (H_e+h_e)\,,\ \boldsymbol{y} \in K^{ext}_e\cup K_e\}\,.
    \end{equation*}
    Additionally, we define $\widehat{p}(\widehat{\boldsymbol{y}}) := p (\boldsymbol{y} / H_e)$ and $\widehat{\Gamma}_e$ denotes the part of the boundary of $\widehat{K}^{ext}_e$ that has been mapped form $\Gamma_e$. Then, applying trace inequality in the reference patch $\widehat{K}^{ext}_e$, we have 
    \begin{equation*}
        \norm{p}_{\Gamma_e}^2
        \lesssim |\Gamma_e| \norm{\widehat{p}}_{\widehat{\Gamma}_e}^2
        \lesssim
        |\Gamma_e| \norm{\widehat{p}}_{\widehat{K}}^2
        \lesssim |\Gamma_e| |K^{ext}_e\cup K_e|^{-1} \norm{p}_{K^{ext}_e\cup K_e}^2\lesssim h_e^{-1} \norm{p}_{K^{ext}_e\cup K_e}^2\qquad\forall p \in \mathbb{P}_k(K^{ext}_e\cup K_e)\,,
    \end{equation*}
    where we note that $|\Gamma_e|$ is proportional to $h_e$ and $|K^{ext}_e \cup K_e|$ is proportional to $h_e^2$, and then \eqref{ineq:discrete-trace-h} holds. Next, \eqref{ineq:discrete-trace-h-2} holds following similar steps to ones above, but using the continuous trace inequality on the reference patch.
\end{proof}

\begin{lemma}
If $B_e$ is a ball with center at the middle point of $e$, such that $K_e^{ext} \cup K_e \subset B_e$, and for all $m \in \mathbb{Z}^+_0$,
\[
\boldsymbol{E} : \boldsymbol{H}^{m+1}(\Omega) \rightarrow \boldsymbol{H}^{m+1}(\mathbb{R}^d)
\]
is an extension operator such that
\begin{equation}
\label{eq:extension-op}
\boldsymbol{E}(\boldsymbol{\rho})\vert_{\Omega} = \boldsymbol{\rho} \quad \text{ for all }\,\boldsymbol{\rho} \in \boldsymbol{H}^{m+1}(\Omega) \qquad \text{ and } \qquad
\|\boldsymbol{E}(\boldsymbol{\rho})\|_{\boldsymbol{H}^{m+1}( \mathbb{R}^d)} \lesssim  \|\boldsymbol{\rho}\|_{\boldsymbol{H}^{m+1}(\Omega)}\,,
\end{equation}
then
\begin{align}
\label{ineq:Ipext}
 \|\bm I_{\bm p}\|_{K^{ext}_e} \,\lesssim 
  \,& (1\,+\,r_e^{1/2}) \,h_e^{m+1}\,	|\boldsymbol{E}(\boldsymbol{p})|_{\boldsymbol{H}^{m+1}(B_e)}
	\,+\,	r_e^{1/2}\,\| \boldsymbol{I}_{\boldsymbol{p}} \|_{K_e}\,, \\[1ex]
    \label{ineq:nablaIpext}
 \|\nabla \bm I_{\bm p}\|_{K^{ext}_e}
	\,\lesssim\,& h^{m}\,	|\boldsymbol{E}(\boldsymbol{p})|_{\boldsymbol{H}^{m+1}(B_e)}
	\,+\,r_e^{1/2}\, h_e^{m}\,\|\boldsymbol{E}( \boldsymbol{p})\|_{{\bm H}^{m+1}(K_e)}
	\,+\,r_e^{1/2}\, h_e^{-1}\,	\|\boldsymbol{I}_{\boldsymbol{p}} \|_{K_e}\,.
\end{align}
\end{lemma}
For the proof of this Lemma we refer to \cite[Lemma 3]{CaSo2021}.
\section{HDG--projections and their properties}\label{sec:HDGprojection}
%
For the error analysis we will use an extension of the HDG--projection developed in \cite{CoGoSa2010} for matrix/vector valued functions. For $(\bm \sigma, \bm w)\in \mathbb{Z}_h\times \bm W_h$ we define $\Pi_h(\bm \sigma, \bm w) := (\bm\Pi_{\mathbb{Z}}\bm \sigma,\bm\Pi_{\bm W}\bm w)$ to be the unique solutions of the system
\begin{subequations}
    \label{HDG_proyection2} 
    \begin{align}
        \label{HDG_proyection2_a}
        (\bm\Pi_{\mathbb Z}\bm \sigma, \bm s)_{K} & \,=\, (\bm r, \bm s)_{K} && \forall \,\bm s\in [\mathbb{P}_{k-1}(K)]^{d\times d}\,,\\
        \label{HDG_proyection2_b}
        (\Pi_{\bm W} \bm w, \bm v)_{K} &\,=\, (\bm w,\bm v)_{K} && \forall \,\bm v\in [\mathbb{P}_{k-1}(K)]^d\,,\\
        \label{HDG_proyection2_c}
        \langle\bm \Pi_{\mathbb  Z}\bm \sigma\bm n\,+\,\tau\,\Pi_{\bm W}\bm w,\bm \mu\rangle_{e} &\,=\, \langle\bm \sigma\bm n\,+\,\tau\, \bm w,\bm \mu\rangle_{e} && \forall\,\bm\mu \in [\mathbb{P}_{k}(e)]^d,\;\text{ and }\,\forall\, e\subset \partial K\,.
    \end{align}
\end{subequations}
As proven in \cite[Theorem 2.1]{CoGoSa2010},  the HDG--projection is well--defined and satisfies the  following properties:
\begin{itemize}
\item[$\bullet$] If $\bm r\in \bm H^{k+1}(K)$ and $w\in H^{k+1}(K)$,
\begin{subequations}
    \label{theorem:proyHDG}
    \begin{alignat}{3}
        \label{theorem:proyHDG_a}
        \norm{\bm\Pi_{\bm Z}\bm r\,-\,\bm r}_{K}&\,\lesssim\,  h_{K}^{k+1}\,|\bm r|_{\bm H^{k+1}(K)}\,+\, h_{K}^{k+1}\,\tau\,|w|_{H^{k+1}(K)}\,,\\[1ex]
        \label{theorem:proyHDG_b}
        \norm{\Pi_W w\,-\, w}_{K} &\,\lesssim\,  h_{K}^{k +1}\,|w|_{H^{k+1}(K)}\,+\, \frac{h_{K}^{k+1}}{\tau}\,|\nabla\cdot \bm r|_{H^{k}(K)}\,.
    \end{alignat}
\end{subequations}

\item[$\bullet$] If $e$ is an edge or face of $\Gamma_h$ then (cf. \cite{CoQiuSo2014})
    \begin{equation}\label{eq:I_r-Gamma-estimate}        \norm{(\bm\Pi_{\bm Z}\bm r\,-\,\bm r)\,{\bm n}_h}_{\Gamma_h,h^{\perp}}
        \,\lesssim\, h^{k+1}\,|\bm r|_{{\bm H}^{k+1}(\Omega)}
        \,+\,h^{k+1}\,\tau\, |w|_{ H^{k+1}(\Omega)}\,,
    \end{equation}
where the weighted norm $\|\cdot\|_{\Gamma_h,h^\perp}$ is as defined in \eqref{eq:weightedNorms}.
\end{itemize}

We will also make use of the orthogonal $L^2$--projection over an element $K$ of a function $v$, denoted by $P_{L^2(K)} v$. It is well known (\cite[Lemmas 1.58 and 1.59]{di2011mathematical}) that, if $v\in H^{l+1}(K)$ for $0\leq l\leq k$, 
\begin{subequations}
\begin{align}
    \label{inequality_projL2_a}
    |v\,-\,P_{L^2(K)}v|_{H^{m}(K)}&\,\lesssim\, h^{l+1-m}\,|v|_{H^{l+1}(K)}\quad\forall m\in \{0,\ldots,k\}\,,\\
    \label{inequality_projL2_b}
    \norm{v\,-\,P_{L^2(K)}v}_{\partial K}&\,\lesssim\,h^{l+1/2}\,|v|_{H^{l+1}(K)}\,.
\end{align}
\end{subequations}
We point out that the previous projection errors can be extended to the vector-valued case.

As final property we state the following estimate
\begin{lemma}\label{lem:grad-estimate}
    Let $\boldsymbol{r} \in \boldsymbol{H}^{k+1}(K)$ and $w \in H^{k+1}(K)$, then
    \begin{equation}\label{eq:grad-estimate}
        \norm{\nabla (\boldsymbol{\Pi}_{\boldsymbol{Z}}\boldsymbol{r} - \boldsymbol{r})}_K 
        \lesssim h_K^{k} \big( |\boldsymbol{r}|_{\boldsymbol{H}^{k+1}(K)} + |y|_{H^{k+1}(K)} \big)\,.
    \end{equation}
\end{lemma}
\begin{proof}
    Note that
    \begin{equation*}
        \norm{\nabla (\boldsymbol{\Pi}_{\boldsymbol{Z}}\boldsymbol{r} - \boldsymbol{r})}_K 
        \leq \norm{\nabla (\boldsymbol{\Pi}_{\boldsymbol{Z}}\boldsymbol{p} - \boldsymbol{P}_{\boldsymbol{L}^2(K)} \boldsymbol{p})}_K
        + \norm{\nabla (\boldsymbol{P}_{\boldsymbol{L}^2(K)} \boldsymbol{p} - \boldsymbol{p})}_K.
    \end{equation*}
    Using and inverse inequality (see \cite[Lemma 1.44]{di2011mathematical}) in the first term and \eqref{inequality_projL2_a} in the second term,
    \begin{equation*}
        \norm{\nabla (\boldsymbol{\Pi}_{\boldsymbol{Z}}\boldsymbol{r} - \boldsymbol{r})}_K 
        \lesssim h_K^{-1} \norm{\boldsymbol{\Pi}_{\boldsymbol{Z}}\boldsymbol{p} - \boldsymbol{P}_{\boldsymbol{L}^2(K)} \boldsymbol{p}}_K
        + h_K^k |\boldsymbol{p}|_{\boldsymbol{H}^{k+1}(K)}
        \leq h_K^{-1} \norm{\boldsymbol{\Pi}_{\boldsymbol{Z}}\boldsymbol{p} - \boldsymbol{p}}_K
        + h_K^{-1} \norm{\boldsymbol{p} - \boldsymbol{P}_{\boldsymbol{L}^2(K)} \boldsymbol{p}}_K
        + h_K^k |\boldsymbol{p}|_{\boldsymbol{H}^{k+1}(K)}.
    \end{equation*}
    Finally, applying \eqref{theorem:proyHDG_a} and \eqref{inequality_projL2_a}, we find that \eqref{eq:grad-estimate} holds.
\end{proof}

{\footnotesize
\bibliographystyle{abbrv}
\bibliography{refs}
}
\end{document}